\documentclass[10pt]{article}
\usepackage[a4paper,margin=1in]{geometry}
\usepackage{setspace}
\usepackage[T1]{fontenc}
\usepackage[utf8]{inputenc}
\usepackage{amsmath,amssymb,mathrsfs,bm}
\usepackage{mathtools} % Extends amsmath, load before amsthm if possible
\usepackage{amsthm}
\usepackage{dsfont}
\usepackage{stmaryrd} % llbracket & rrbracket
\usepackage[mathscr]{eucal}

\usepackage{graphicx}
\usepackage[table,xcdraw]{xcolor}
\definecolor{lightblue}{RGB}{70,130,180}
\usepackage{booktabs}
\usepackage{multirow,pbox}
\usepackage{caption,subcaption}
\usepackage{float}
\usepackage{algorithm}
\usepackage{algpseudocode}
\usepackage{diagbox}

\usepackage{enumitem}
\setlist{leftmargin=9mm}

\usepackage[square,authoryear]{natbib}
\usepackage[colorlinks,linkcolor=blue,citecolor=blue,urlcolor=blue,pagebackref]{hyperref}

\renewcommand*{\backrefalt}[4]{%
    \ifcase #1 \footnotesize{(not cited)}%
    \or \footnotesize{(cited on page~#2)}%
    \else \footnotesize{(cited on pages~#2)}%
    \fi}

\allowdisplaybreaks[4]
\numberwithin{equation}{section}
\usepackage[toc,page]{appendix}

\theoremstyle{plain}
\newtheorem{assumption}{Assumption}

\newlist{assumptionparts}{enumerate}{1}
\setlist[assumptionparts,1]{
  label=(\Alph{assumption}\arabic*),
  ref=(\Alph{assumption}\arabic*),
  font=\normalfont,
  leftmargin=*
}

\theoremstyle{plain}
\newtheorem{theorem}{Theorem}[section] % Main counter, resets every section
\newtheorem{lemma}[theorem]{Lemma}
\newtheorem{proposition}[theorem]{Proposition}

\theoremstyle{definition}
\newtheorem{definition}[theorem]{Definition}

\newtheorem{remark}[theorem]{Remark}

\usepackage{cancel}
\usepackage{soul}

\usepackage{dsfont}
\newcommand{\indicator}{\mathds{1}}

\def\Tr{\mathrm{Tr}}

\renewcommand{\Re}{\mathrm{Re}\,}
\renewcommand{\Im}{\mathrm{Im}\,}
\renewcommand{\Pr}{\mathbb{P}\,}
\newcommand{\E}{{\mathbb E}}

\newcommand{\R}{{\mathbb R}}

\newcommand{\C}{{\mathbb C}}
\newcommand*{\dif}{\mathop{}\!\mathrm{d}}
\newcommand{\deq}{\coloneqq}
\newcommand{\eqd}{\eqqcolon}
\renewcommand{\leq}{\leqslant}
\renewcommand{\le}{\leqslant}
\renewcommand{\geq}{\geqslant}
\renewcommand{\ge}{\geqslant}
\DeclareMathOperator{\supp}{supp}
\DeclareMathOperator{\dist}{dist}
\DeclareMathOperator{\spec}{spec}

\newcommand{\ii}{\mathrm{i}} % imaginary unit
\DeclarePairedDelimiter{\intset}{\llbracket}{\rrbracket}
\DeclarePairedDelimiter{\abs}{\lvert}{\rvert}
\DeclarePairedDelimiter{\norm}{\lVert}{\rVert}
\usepackage{bm}
\newcommand{\one}{\boldsymbol{1}}
\newcommand{\diag}{\mathrm{diag}}
\newcommand{\smallop}{\mathfrak{o}_{\mathbb{P}}}
\newcommand{\bigop}{\mathcal{O}_{\mathbb{P}}}
\newcommand{\smallo}{\mathfrak{o}}
\newcommand{\bigo}{\mathcal{O}}

\makeatletter
\renewcommand*{\top}{{\mathpalette\@transpose{}}}
\newcommand*{\@transpose}[2]{\raisebox{\depth}{$\m@th#1\intercal$}}
\makeatother
\newcommand{\widesim}[2][1.5]{
	\mathrel{\overset{#2}{\scalebox{#1}[1]{$\sim$}}}
}
\newcommand{\iidsim}{\widesim[2.33]{\mathrm{i.i.d.}}} % iid

\title{\textbf{Spike Estimation from Heteroscedastic Noise via Random Splitting}}
\author{Zhigang Bao\thanks{University of Hong Kong, \texttt{zgbao@hku.hk}},\;
Kha Man Cheong\thanks{Hong Kong University of Science and Technology, \texttt{kmcheong@connect.ust.hk}},\;
Yuji Li\thanks{University of Hong Kong, \texttt{u3011732@connect.hku.hk}},\;
Jiaxin Qiu\thanks{University of Hong Kong, \texttt{jxqiu@hku.hk}}
}

\begin{document}

\maketitle

\begin{abstract}
In this paper, we consider a spiked Wigner type matrix with a heteroscedastic and unknown variance profile. It is well known that in the supercritical regime of the BBP transition, strong spikes can create outliers in the spectrum. Unfortunately, in the heteroscedastic case, in general it is not possible to estimate the spike strength from these observed outlier consistently, as the latter is a solution to a Dyson equation with unknown parameters from the variance profile. In this paper, inspired by the work on sparse matrix completion \citep{BordenaveCosteNadakuditi2023}, we introduce an asymmetrized model by randomly splitting the spiked matrix into two parts, which transforms the noisy Wigner type matrix into a non Hermitian random matrix, while preserving the Hermitian spikes at the cost of a dilution. We establish a BBP type transition for the asymmetrized model, from which we can estimate the strength of the spikes precisely, even without knowing the variance profile of the noise part. We then further apply our approach to study the correlation between two correlated spiked models, where the spike/signal parts of the two models are correlated, and the noise parts are independent but may both be heteroscedastic. By applying our asymmetrization approach to the two models separately and also jointly, we are able to obtain a precise estimate of the correlation between the signal parts of the two models. 

\end{abstract}

%\tableofcontents

\section{Introduction}
\label{sec:introduction}

Low rank signal plus noise models provide a basic framework for high dimensional inference. 
Estimating the low rank signal from a single noisy observation is a fundamental problem in statistics, machine learning, and signal processing.
In this paper, we consider a data matrix $Y$ modelled as
\begin{equation}
\label{eq:unified-signal plus noise-model}
        Y=X+S\in\R^{p\times n},
        \qquad \operatorname{rank}(S)=r,
\end{equation}
where $X$ is a centered noise matrix and $S$ is a deterministic signal of fixed rank.
The inferential goals are to estimate the singular values of $S$, and recover information about the associated signal subspaces from the observed data matrix $Y$.
In this paper, the noise matrix $X$ is allowed to be heteroscedastic with a unknown variance profile
\[
        T=(t_{ij}),
        \qquad
        t_{ij}=\E X_{ij}^2\asymp n^{-1}.
\]
Heteroscedastic noise arises naturally in several matrix valued data types, such as the adjacency matrix of a network with heterogeneous expected degrees \citep{KarrerNewman2011,GKKZ2025lowrank}, count matrices from single cell and spatial genomics, Hi-C, and document term data
\citep{LandaZhangKluger2022,LandaKluger2025}, and photon limited imaging with Poisson distributed pixel counts \citep{LiuDobribanSinger2018}.  
These examples include both symmetric and rectangular data matrices, and the variance profile may be heterogeneous across both rows and columns.

We study two settings of \eqref{eq:unified-signal plus noise-model}.  In
the symmetric setting, $p=n$, the matrices $X$ and $S$ are symmetric, and
\[
        S=UDU^\top,
        \qquad
        D=\diag(d_1,\ldots,d_r),
\]
where $U$ contains orthonormal columns.  
In the rectangular setting, $p/n\to\phi\in(0,\infty)$ and
\[
        S=UDV^\top,
\]
where both $U$ and $V$ contain orthonormal columns. 
Our first goal is to estimate the spike strengths $d_k$ from a single noisy observation $Y$. 
The second goal is to estimate the correlation between two signal matrices $S_1$ and $S_2$ from two given noisy data sets $Y_1$ and $Y_2$.

The spectral properties of signal plus noise models have been extensively studied in Random Matrix Theory, High-dimensional Statistics and their applications; see, for example, \citet{BBP2006,FeralPeche2007,BaiYao2008,CapitaineDonatiMartinFeral2009,BenaychGeorgesGuionnetMaida2011,BenaychGeorgesNadakuditi2011,CapitaineDonatiMartinFeralFevrier2011,LoubatonVallet2011,BaiYao2012,BenaychGeorgesNadakuditi2012,HachemLoubatonMestreNajimVallet2012,BloemendalVirag2013,KnowlesYin2013,PizzoRenfrewSoshnikov2013,KnowlesYin2014,Capitaine2014,BloemendalVirag2016,BaoDingWang2021,GengLiuZou2024,BykhovskayaGorinSodin2025}, etc. But nearly all these works assumed that the noise entries have the same and known variance.
Our main focus of this work  is on estimating spike strengths from the observed spectrum without knowing the noise variance profile.

Even in the homogeneous Wigner model, a sample outlier does not directly estimate its underlying spike.
For example, if $t_{ij}=\sigma^2/n$, then a supercritical rank one signal of strength $d$ produces a sample outlier near $d+\sigma^2/d$ \citep{CapitaineDonatiMartinFeral2009,BenaychGeorgesNadakuditi2011}, rather than near
$d$ itself.  
Recovering $d$ therefore requires a noise dependent correction or transformation.
Under a general profile, the correction even depends on the signal directions.
Indeed, the noise resolvent $G_X(z) \deq (X-zI_n)^{-1}$ is approximated by
the diagonal matrix $M_X(z)=\diag(m_1(z),\ldots,m_n(z))$, where $m_i(z)$ solves the vector Dyson equation \citep{AjankiErdosKruger2017,AjankiErdosKruger2019QVE}
\begin{equation*}
        -\frac{1}{m_i(z)}
        =
        z+\sum_{j=1}^n t_{ij}m_j(z),
        \qquad i\in\intset{n} \deq \{1,2,\ldots,n\},
\end{equation*}
and the finite rank perturbation argument gives the outlier equation
\begin{equation*}
        \det\!\left(I_r+D\,U^\top M_X(z)U\right)=0.
\end{equation*}
Hence the map from an observed sample outlier to $d_k$ depends on both the unknown profile and the unknown signal directions.  
The same issue persists in the rectangular model if we aim to estimate the signal strength via the singular values of the data matrix. 
For rectangular models with Gaussian noise, \citet{BigotMale2021} describe singular value outliers through deterministic equations involving both the variance profile and the signal singular vectors.
Consequently, the direct spectral analysis of $Y$ does not yield a noise profile free and signal direction free estimation from observed outliers to spike strengths. This has been a fundamental difficulty on the heteroscedastic spike models. 

There exist several methods in the literature that can estimate the spike strengths from the observed data, but they all require either additional assumptions on the noise or additional observations.
\citet{ShabalinNobel2013} invert the rectangular outlier map and estimates the noise variance, but assumes i.i.d. Gaussian noise.  
\citep{Nadakuditi2014OptShrink} assumes bi-unitarily/bi-orthogonally invariant noise, and proposes the \texttt{OptShrink} estimator by inverting a noise $D$-transform \citep{BenaychGeorgesNadakuditi2012}.
\citet{SuWu2025} extend this approach to noise with separable covariance structure. 
\citet{ChenChengFan2021} estimate the spike strength in a rank one model under heteroscedastic noise using the leading eigenvalue of an asymmetric/non-Hermitian data matrix. Along this line, 
\citet{BaoCheongLeeLi2025} propose an asymmetrization based detection method to detect signal even when the noise itself produces large spikes, and the method requires two independent noisy observations.
\citet{GavishLeebRomanov2023}  propose a pseudo whitening based singular value shrinkage procedure by leveraging auxiliary pure noise samples and the outlier asymptotics of spiked $F$ matrices.
To our knowledge, there is no existing method that can estimate the spike strengths from a single noisy observation with heteroscedastic noise and an unknown variance profile for spiked symmetric matrices.

The use of a non-Hermitian matrix instead of a Hermitian one is a key input to our work.  This idea originates in Tao's observation that, for a bounded rank deformation of an i.i.d. non-Hermitian matrix, an outlier outside the circular law support converges to the corresponding signal eigenvalue, without an order one bias which exists in the spiked Hermitian models \citep{Tao2013}.  In \citet{ChenChengFan2021}, the authors extended this idea to the low rank estimation problem  with independent, possibly heteroscedastic, asymmetric noise, and further in \citet{BaoCheongLeeLi2025} the authors used the asymmetrization idea for signal detection under heteroscedastic and heavy-tailed noise, when the noise itself could create large singular values.  However, the former work assumes direct access to an asymmetric perturbation, whereas the latter requires two independent samples. Consequently, these constructions cannot in general be applied when only one symmetric observation is available, or an asymmtric rectangular matrix is observed but one wants to estimate the leading singular values in the signal which again boils down to the eigenvalue problem of a symmetric matrix.  Several other methods based on one sample  which bear an asymmetrization natural are also related to various extent.  For very sparse matrix completion, in \citet{BordenaveCosteNadakuditi2023}, the authors randomly split the observed entries into two disjoint matrices and analyze their asymmetric cross-product.  This construction is distinct from, although closely related to, the non-backtracking operator studied in the same paper, which can be regarded as another type of asymmetrization approach. However, the approach developed in the paper applies to very sparse random observations, and does not seem to apply to the dense matrix case. In \citet{DingHopkinsSteurer2020}, the authors instead use self-avoiding walks to estimate a rank one spike under independent heavy-tailed noise with only finite second moments. But the variances of the noise matrix are required to be the same and known, and the focus is the eigenvector.  Related non-backtracking and self-avoiding constructions also appear in other works such as  
\citet{Massoulie2014, BordenaveLelargeMassoulie2018, StephanZhu2024}.  These works motivate constructing asymmetry from a single sample by preventing the repeated use of the same noisy observation.

Our starting point of this paper is inspired by the work of \citet{BordenaveCosteNadakuditi2023}, who randomly split the observed matrix into complementary parts and analyze the resulting non-Hermitian cross-product.
Their setting is sparse matrix completion rather than a dense additive signal plus noise model \eqref{eq:unified-signal plus noise-model}, but their random splitting idea motivates the construction below. 
Let $P\in\{0,1\}^{p\times n}$ be independent of the observation and have $\mathrm{Bernoulli}(1/2)$ entries.  
In the rectangular model these entries are i.i.d.; in the symmetric model, $p=n$ and $P$ is generated symmetrically from independent upper triangular entries.  
Define
\begin{equation}\label{eq:split-matrix-def}
\begin{aligned}
        \mathcal Y
        & \deq
        \begin{pmatrix}
        0&P\circ Y\\
        \bigl((\one_p\one_n^\top-P)\circ Y\bigr)^\top&0
        \end{pmatrix}\\
        & = \begin{pmatrix}
        0&P\circ X\\
        \bigl((\one_p\one_n^\top-P)\circ X\bigr)^\top&0
        \end{pmatrix} + \begin{pmatrix}
        0&P\circ S\\
        \bigl((\one_p\one_n^\top-P)\circ S\bigr)^\top&0
        \end{pmatrix}
\end{aligned}
\end{equation}
The two off-diagonal blocks contain complementary subsets of the entries of $Y$.
The signal part of $Y$ admits the decomposition
\begin{equation*}
\begin{pmatrix}
0&P\circ S\\
\bigl((\one_p\one_n^\top-P)\circ S\bigr)^\top&0
\end{pmatrix} 
 =
\frac12
\begin{pmatrix}0&S\\S^\top&0\end{pmatrix}
+
\begin{pmatrix}0&\mathcal R\\-\mathcal R^\top&0\end{pmatrix},
\qquad 
\mathcal R
\deq
\left(P-\frac12\one_p\one_n^\top\right)\circ S.
\end{equation*}
The first matrix on the right hand side has nonzero eigenvalues
$\{\pm d_k/2:k\in\intset{r}\}$, in both the symmetric and rectangular models.
A rather weak incoherence condition on the signal directions  ensures that $\mathcal R$ is negligible in operator norm.  
Thus a positive outlier $\lambda_{k,+}(\mathcal Y)$ leads to the direct estimator
\begin{equation}
\label{eq:unified-signal-strength-estimator}
        \widehat d_k
        \deq
        2\Re\lambda_{k,+}(\mathcal Y).
\end{equation}

We will derive two main results based on the asymmetrized model via random splitting in \eqref{eq:split-matrix-def}. 

Our first result establishes a BBP type transition for our estimator.
The asymmetrized noise matrix, i.e., the first term in the RHS of \eqref{eq:split-matrix-def},  has asymptotic outer radius
$
        b_{\star,n}=\sqrt{\|T\|_2/2}.
$
For fixed rank positive spikes whose directions are sufficiently incoherent in the sense that $\max_{k\in\intset{r}}\|u_k\|_\infty=\smallo(1)$, every spike satisfying $d_k/2>b_{\star,n}$ by a fixed margin generates two outliers.
They can be labeled so that
\[
        \lambda_{k,\pm}(\mathcal Y)
        =
        \pm\frac{d_k}{2}
        +
        \smallop(1).
\]
Conversely, spikes that remain a fixed distance below $b_{\star,n}$ produce no eigenvalues outside the noise bulk.
It follows that, when the rank is known and the spikes are supercritical, $2\Re\lambda_{k,+}(\mathcal Y)$ consistently estimates $d_k$ without estimating $T$ or applying a profile dependent bias correction.

The second result concerns the estimation of correlation between two correlated spiked models.
We consider two noisy observations
\[
    Y_a = X_a + U_a D_a V_a^{\top}, \qquad a = 1,2,
\]
where the spike parts are correlated and the noise parts are independent but may both be heteroscedastic.
The statistical target is the cross sample overlap matrixs $U_1^\top U_2$ and $V_1^{\top}V_2$, which is identifiable only up to sign changes.
We use samplewise random splitting and the resulting eigenvector projections to construct estimates of the entrywise overlap magnitudes. 
This problem connects to recent work on signal recovery and prediction from correlated spiked models.
\citet{KeupZdeborova2025,NandyMa2024} develop approximate message passing (AMP) methods for signal recovery.
\citet{MergnyZdeborova2025} analyze outliers and signal direction overlaps of singular vectors of $Y_1^\top Y_2$.
\citet{Li2025CorrelatedSpikes,DuHuLepsveridze2026,YangSenLu2026} study detection and recovery thresholds using combinatorial and spectral methods.
However, the above references mainly assume that the spikes are generated from a joint probability distribution and the goal is to detect the presence of correlated spikes or to recover the individual spike vectors. Our problem instead treats the spikes and strengths as deterministic  and seeks to estimate their unknown overlap/correlation directly from two independent spiked Wigner observations. Hence, our contribution is a consistent estimation of signal overlaps under unknown and possibly different variance profiles.

Further, although our construction uses a similar random splitting idea as \citet{BordenaveCosteNadakuditi2023}, the resulting spectral problem is
technically rather different.  In \citet{BordenaveCosteNadakuditi2023}, entries are
initially observed with probability of order \(n^{-1}\), so the split matrices
have bounded expected degree.  Their analysis therefore relies on sparse graph
geometry: logarithmic powers are expanded as weighted path sums, local
neighborhoods are coupled to marked Galton--Watson trees, and tangle free
decompositions together with high trace estimates are used to construct
approximate eigenvectors and control the remaining spectrum.  By contrast, our random matrix models
are dense,  with entry correlations
inherited from the Wigner symmetry and the complementary splitting of a single
sample.  Consequently, no locally tree like approximation is available.
Following the dense random matrix strategy of \citet{BaoCheongLeeLi2025}, we instead use linearization and Hermitization,
analyze the corresponding Dyson equation, and establish isotropic
Green function estimates; the signal outliers are
then characterized by a finite dimensional determinant equation.  The main
additional difficulty relative to \citet{BaoCheongLeeLi2025} is that the
linearized noise has a nontrivial entry correlation structure, rather than independent
off diagonal blocks, and this dependence must be retained throughout the
resolvent analysis. In particular, we will apply the approach in \citet{AjankiErdosKruger2019} for random matrix with correlated entries. We also emphasize that although the previous work \citet{BaoCheongLeeLi2025} requires two samples, it has a different focus as it is mainly for the scenario when the noise itself could create much larger singular values than the signal part, and the asymmtrization is used to distinguish the signal from the noise spikes. It is not mainly aim for dealing with a heteroscedasitic noise with mild variance profile only. It is not clear if the one sample approach applied here can be extended to deal with the scenario when a big spike presents in the noise part as we previously discussed in \citet{BaoCheongLeeLi2025}. It will be left a a furture research direction.

Throughout this paper, we will use the following notations. 
We use $c, C$ for positive constants that may change from line to line.
The symbol $X \deq Y$ (or $Y \eqd X$) indicates that $X$ is defined as $Y$. 
For a positive integer $n$, we write $\intset{n}\deq\{1,\ldots,n\}$.
For any matrix $A$, we denote its $(i,j)$-th entry by $A_{ij}$, its transpose by $A^{\top}$, its Hermitian transpose by $A^*$, and its spectral norm by $\|A\|_2$.
For $A\in\mathbb C^{m\times n}$, we write
$\sigma_1(A)\geq\cdots\geq\sigma_{\min\{m,n\}}(A)\geq0$ for its singular
values and set $\sigma_{\min}(A)\deq\sigma_{\min\{m,n\}}(A)$.
If $A$ is square, then $\Tr(A)$ and $\spec(A)$ denote its trace and spectrum, respectively, and $\varrho(A)\deq\max\{|\lambda|:\lambda\in\spec(A)\}$ denotes its spectral radius.
Let $\lambda_1(A),\ldots,\lambda_m(A)$ denote the eigenvalues of $A \in \mathbb C^{m\times m}$, ordered by decreasing real parts.
We write $a_n\gtrsim b_n$ when $a_n\geq cb_n$ for some absolute constant $c>0$, and $a_n\lesssim b_n$ when $b_n\gtrsim a_n$. 
We write $a_n\asymp b_n$ if $c_0\leq a_n/b_n \leq C_0$ for some absolute constants $c_0, C_0 >0$. 
For random variables $\{X_n\}_{n=1}^{\infty}$ and positive real numbers $\{a_n\}_{n=1}^{\infty}$, we write $X_n=\bigop(a_n)$ if $X_n/a_n$ is bounded in probability, and $X_n=\smallop(a_n)$ if $X_n/a_n$ converges to zero in probability.

The rest of the paper is organized as follows. 
Section~\ref{sec:spiked-wigner-model} presents the random-splitting estimators for spike strengths, overlaps, and signal correlation in the symmetric Wigner model, together with their theoretical guarantees.
In Section~\ref{sec:rect-extension}, we extend these estimators and results to rectangular matrices.
Section~\ref{sec:numerical-simulations} reports numerical simulations for the Wigner model.
We present the proof strategy and proofs of the main Wigner results in Section~\ref{sec:proof-strategy}.
The supplementary material contains additional numerical results and proofs of the auxiliary lemmas.

\section{Symmetric Wigner Model}\label{sec:spiked-wigner-model}

\subsection{Random-Splitting Estimator}

Recall the signal plus noise model defined in \eqref{eq:unified-signal plus noise-model}. 
We first introduce the following rather general assumptions on the noise and signal.

\begin{assumption}\label{asm:wigner-assumptions}
\leavevmode
\begin{assumptionparts}
\item (Noise moments) \label{asm:moment}
The matrix $X=(X_{ij})_{i,j=1}^n$ is real symmetric, and the variables $\{X_{ij}:1\leq i\leq j\leq n\}$ are independent.
For constants $0<C_*\leq C^*<\infty$,
\[
        \E X_{ij}=0,
        \qquad
        \E X_{ij}^2=t_{ij},
        \qquad
        \frac{C_*}{n}\leq t_{ij}\leq\frac{C^*}{n}.
\]
Set
\[
        T\deq(t_{ij})_{i,j=1}^n,
        \qquad
        \mathbb T\deq\frac{T}{2}.
\]
For every fixed integer $k\geq1$, there is a constant $C_k<\infty$ such that
\[
        \sup_n\sup_{i,j\in\intset{n}}
        \E|\sqrt n\,X_{ij}|^k
        \leq C_k.
\]

\item (Incoherent signal directions) \label{asm:delocalization}
The deterministic unit vectors $u_1,\ldots,u_r$ are orthonormal and satisfy
\[
        \mu_n
        \deq
        \max_{k\in\intset{r}}\|u_k\|_\infty
        =
        \smallo(1).
\]

\item (Population spikes)
\label{asm:signal}
The rank $r$ is fixed. 
The spike strengths $d_1\geq\cdots\geq d_r>0$ are fixed deterministic numbers, independent of $n$, and $d_{\max}\deq d_1\leq C$ for a fixed constant $C$.
There exists a fixed integer $r_+\in\intset{r}$ such that, for some fixed $\delta>0$ and all sufficiently large $n$,
\begin{equation*}
\min_{1 \leq k \leq r_+}\frac{d_k}{2} \geq b_{\star,n}+\delta,
\qquad
\max_{r_+ < k \leq r}\frac{d_k}{2} \leq b_{\star,n}-\delta,
\qquad
b_{\star,n}\deq\sqrt{\|\mathbb T\|_2}.
\end{equation*}
The second condition is omitted when $r_+=r$.
\end{assumptionparts}
\end{assumption}

To estimate the spike strengths, we construct an asymmetric matrix $\mathcal Y$ defined in \eqref{eq:split-matrix-def}.
Let $\{\lambda_i(\mathcal Y)\}_{i=1}^{2n}$ be the eigenvalues of $\mathcal Y$, ordered by decreasing real parts. 
For $J=\diag(I_n,-I_n)$, the block form of $\mathcal Y$ gives $\mathcal Y J = -J \mathcal Y$.
This implies that the eigenvalues of $\mathcal Y$ are symmetric about the origin, and we can label them as
\[
        \lambda_{i,+}(\mathcal Y)\deq \lambda_i(\mathcal Y),
        \qquad
        \lambda_{i,-}(\mathcal Y)\deq -\lambda_i(\mathcal Y),
        \qquad 
        i\in\intset{n}.
\]
Writing
\begin{equation}\label{eq:split-matrix-decomposition}
        \mathsf A=P\circ X,
        \qquad
        \mathsf B=(\one_n\one_n^\top-P)\circ X,
        \qquad
        \mathcal X=\begin{pmatrix}0&\; \mathsf A\\ \mathsf B&\;0\end{pmatrix}.
\end{equation}
Here, $P$ is a symmetric matrix with independent upper triangular $\mathrm{Bernoulli}(1/2)$ entries.
For $\mathsf s\in\{+,-\}$, define
\begin{equation}
\label{eq:normalized-spike-basis}
        \mathfrak u_{k,\mathsf s}
        \deq  \frac{1}{\sqrt{2}} (u_k^{\top}, \; \mathsf s u_k^{\top})^{\top},
        \qquad
        \alpha_{k,\mathsf s}
        \deq
        \mathsf s\frac{d_k}{2}.
\end{equation}
Let $\mathcal U$ collect these $2r$ vectors and let $\mathcal D$ be the diagonal matrix containing the corresponding values $\alpha_{k,\mathsf s}$.
We have the following decomposition
\begin{equation}
\label{eq:split-signal-decomposition}
        \mathcal Y=\mathcal X+\mathcal E+\mathcal U\mathcal D\mathcal U^\top,
        \qquad
        \mathcal E \deq
        \begin{pmatrix}
        0&\mathcal R\\
        -\mathcal R&0
        \end{pmatrix},
        \qquad
        \mathcal R \deq
        \left(P-\frac12\one_n\one_n^\top\right)\circ S.
\end{equation}
Thus, the leading signal has eigenvalues $\pm d_k/2$, while the remainder $\mathcal R$ is confined to $\mathcal E$. 
The incoherence condition \ref{asm:delocalization} makes $\|\mathcal E\|_2=\smallop(1)$ (see Lemma~\ref{lem:mask-remainder-bound}).

\begin{algorithm}[htbp]
\caption{(Wigner case) Spike estimation by random splitting}
\label{alg:spike-estimation}
\begin{algorithmic}[1]
\Require A symmetric observation $Y\in\mathbb R^{n\times n}$ and a known $r_+$.
\Ensure Estimates $\widehat d_1,\ldots,\widehat d_{r_+}$.
\State Draw an independent symmetric mask $P$ with upper triangular entries distributed as $\mathrm{Bernoulli}(1/2)$.
\State Form $\mathcal Y$ as in \eqref{eq:split-matrix-def}.
\State Compute the eigenvalues of $\mathcal Y$ and set
\[
        \widehat d_k\deq2\Re\lambda_{k,+}(\mathcal Y),
        \qquad
        k\in\intset{r_+}.
\]
\State \textbf{Return} $\widehat d_1,\ldots,\widehat d_{r_+}$.
\end{algorithmic}
\end{algorithm}

\begin{remark}\label{rmk:unknown-rank}
Our Algorithm~\ref{alg:spike-estimation} assumes that $r_+$ is known.
When $r_+$ is unknown, one may estimate it by first estimating the bulk edge from eigenvalues whose arguments stay away from the real axis, and then count positive real axis outliers beyond that estimate.
Following \citet[Section~2]{BaoCheongLeeLi2025}, we define the data-based estimator and the corresponding candidate signal region
\begin{equation}\label{eq:sector-edge-estimator}
\begin{aligned}
        \lambda_{\max}^{\mathrm s}
        &\deq
        \max_{j \in \intset{2n}}
        |\lambda_j(\mathcal Y)|
        \indicator\bigg\{\arg\lambda_j(\mathcal Y)\in\bigg[\frac{\pi}{\log(2n)},\frac{\pi}{2}\bigg]\bigg\},\\
        \mathscr D_n
        & \deq
        \big\{
        z\in\mathbb C:
        \Re z\geq\lambda_{\max}^{\mathrm s}+(2n)^{-1/2}
        \big\}.
\end{aligned}
\end{equation}
The sector in \eqref{eq:sector-edge-estimator} is designed to avoid the signal outliers, so
that $\lambda_{\max}^{\mathrm s}$ serves as a data-based estimator for the bulk radius.
We identify $\lambda_j(\mathcal Y)$ as a signal eigenvalue whenever $\lambda_j(\mathcal Y)\in\mathscr D_n$. 
More specifically, we estimate $r_+$ by 
\begin{equation*}
        \widehat r_+ \deq \sum_{j=1}^{2n} \indicator \{\lambda_j(\mathcal Y) \in \mathscr D_n\}.
\end{equation*}
\end{remark}

Fix $\vartheta\in(0,1/2)$ and define the error scale
\begin{equation}
\label{eq:outlier-error-scale}
        \eta_n
        \deq
        n^{-1/2+\vartheta}+d_{\max}\mu_n.
\end{equation}
Assumptions~\ref{asm:delocalization} and \ref{asm:signal} imply $\eta_n=\smallo(1)$.
The following theorem establishes the consistency of the spike estimator in Algorithm~\ref{alg:spike-estimation}.

\begin{theorem}[Consistency of Algorithm~\ref{alg:spike-estimation}]
\label{thm:spike-estimation-consistency}
Under Assumption~\ref{asm:wigner-assumptions}, Algorithm~\ref{alg:spike-estimation} satisfies
\[
        \max_{k\in\intset{r_+}}
        |\widehat d_k-d_k|
        = 
        \bigop(\eta_n)
        =
        \smallop(1).
\]
\end{theorem}

Theorem~\ref{thm:spike-estimation-consistency} is a direct consequence of the following proposition, which establishes a phase transition for the asymmetrized model $\mathcal Y$.

\begin{proposition}[Phase transition]
\label{prop:outlier-location}
Suppose that Assumption~\ref{asm:wigner-assumptions} hold.
For every fixed $\kappa\in(0,\delta)$, with probability tending to one, the eigenvalues of $\mathcal Y$ in $\{z\in\mathbb C:|z|\geq b_{\star,n}+\kappa\}$ consist of exactly $2r_+$ eigenvalues, counted with algebraic multiplicity, and
\begin{equation}
\label{eq:outlier-location-rate}
        \max_{k\in\intset{r_+},\;\mathsf s\in\{+,-\}}
        \left|
        \lambda_{k,\mathsf s}(\mathcal Y) - \alpha_{k, \mathsf s}
        \right|
        =
        \bigop(\eta_n),
\end{equation}
where $\alpha_{k,\mathsf s}$ is defined in \eqref{eq:normalized-spike-basis}.
There are no other eigenvalues of $\mathcal Y$ in $\{z\in\mathbb C:|z|\geq b_{\star,n}+\kappa\}$.
\end{proposition}

\subsection{Correlated Spiked Model and Identifiability}

For $a\in\{1,2\}$, observe
\begin{equation}
\label{eq:two-sample-model}
        Y_a
        =
        X_a+U_aD_aU_a^\top,
        \qquad
        D_a=\diag(d_{a1},\ldots,d_{ar}).
\end{equation}
The noise matrices $X_1$ and $X_2$ are independent, may have different variance profiles $T_1$ and $T_2$, and each satisfies Assumption~\ref{asm:moment}.
For $a\in\{1,2\}$, set $\mathbb T_a\deq T_a/2$.
For each $a\in\{1,2\}$, suppose that Assumption~\ref{asm:wigner-assumptions} holds with $(X,U,D,T,r_+)$ replaced by $(X_a,U_a,D_a,T_a,r_{a,+})$ and $\operatorname{rank}(D_1) = \operatorname{rank}(D_2) = r$.
For notational simplicity, we denote $r_a \deq r_{a,+}$.
The quantities $r_1, r_2 \in \intset{r}$ are assumed known and need not be equal. 
Write $S_a = U_aD_aU_a^{\top}$ and define the supercritical part by
\[
S_{a,+} \deq \sum_{i\in\intset{r_a}} d_{ai}u_{a,i}u_{a,i}^{\top}, 
\qquad
U_{a,+} \deq (u_{a,1}, \ldots, u_{a,r_a}).
\]
We aim to estimate the overlap matrix between $U_{1,+}$ and $U_{2,+}$:
\[
        R\deq U_{1,+}^\top U_{2,+}
        =(u_{1,k}^\top u_{2,\ell})_{k\in\intset{r_1},\ell\in\intset{r_2}}
        \in \mathbb R^{r_1\times r_2}.
\]

Since both $U_{1,+}$ and $U_{2,+}$ have orthonormal columns, the singular values of $R$ are the canonical correlations between $\operatorname{span}(U_{1,+})$ and $\operatorname{span}(U_{2,+})$.
Writing the principal angles as $0\leq\theta_{n,1}\leq\cdots\leq\theta_{n,\min\{r_1,r_2\}}\leq\pi/2$, we have $\sigma_k(R)=\cos\theta_{n,k}$ 
for $k\in\intset{\min\{r_1,r_2\}}$.
Hence, the entries of $R$ measure the overlap between the two spike subspaces $\operatorname{span}(U_{1,+})$ and $\operatorname{span}(U_{2,+})$.

The sign of entries of $R$ is not identifiable from $Y_1$ and $Y_2$ alone. 
For a positive integer $m$, define 
\[
        \mathscr S_m
        \deq
        \left\{
        \diag(\varepsilon_1,\ldots,\varepsilon_m):\varepsilon_i\in\{-1,1\}
        \right\}.
\]
For $\Xi_1,\Xi_2 \in \mathscr S_r$, we have $U_a D_a U_a^\top=(U_a \Xi_a)D_a(U_a\Xi_a)^\top$ for $a=1,2$, but the overlap matrix $R$ changes to $\Xi_1R\Xi_2$, while the distribution of the observed data is unchanged.
Therefore the identifiable overlap parameters are the entrywise magnitude $|R_{ij}|$.

The projector-based estimator below uses samplewise random splitting to recover a representative of the identifiable sign-equivalence class.
Define the equivalence class
\begin{equation}
\label{eq:overlap-sign-equivalence}
        [R]_{\pm}
        \deq
        \{\Xi_1R\Xi_2:\Xi_1 \in \mathscr S_{r_1},\Xi_2\in\mathscr S_{r_2}\}.
\end{equation}
Since $[\Xi_1 R \Xi_2]_{\pm} = [R]_{\pm}$ for any $\Xi_1\in\mathscr S_{r_1}, \Xi_2\in\mathscr S_{r_2}$, estimating one representative of $[R]_{\pm}$ is already sufficient to recover the entire equivalence class.
Moreover, each representative of $[R]_{\pm}$ preserves some quantities of interest, such as $|R_{ij}|$, $R^{\circ 2}$, $\operatorname{rank}(R)$, $\|R\|_F$, and $\|R\|_2$.
Also,
$
        (\Xi_1 R \Xi_2)(\Xi_1 R \Xi_2)^{\top}
        =\Xi_1RR^{\top}\Xi_1^{\top},
$
so $RR^{\top}$ and the transformed version have identical eigenvalues.
Hence, all principal angle information between $\operatorname{span}(U_{1,+})$ and $\operatorname{span}(U_{2,+})$ is preserved.
We aim to estimate a representative of $[R]_{\pm}$ consistently. 

Assume that the rank $r$ is known and fixed.
For $a=1,2$, define
\[
        b_{\star,a,n}
        \deq
        \sqrt{\|\mathbb T_a\|_2},
        \qquad
        \mu_{a,n}
        \deq
        \max_{i\in\intset{r}}\|u_{a,i}\|_\infty,
        \qquad
        \eta_{a,n}
        \deq
        n^{-1/2+\vartheta}
        +d_{a,\max}\mu_{a,n},
\]
where $d_{a,\max} \deq d_{a1}=\bigo(1)$ and $\mu_{a,n}=\smallo(1)$ by Assumption~\ref{asm:delocalization}.
To guarantee the identifiability of the overlap matrix, we impose the following seperation condition.
\begin{assumption}\label{asm:overlap-entry-assumptions}
For some fixed $\tau>0$, assume
\[
        \min_{a\in\intset{2}}
        \min_{i,k\in\intset{r_a}, \; i\neq k}
        |d_{ai}-d_{ak}|
        \geq\tau.
\]
For each $a\in\intset{2}$, the second minimum is interpreted as $+\infty$ when $r_a=1$.
\end{assumption}
\begin{remark}
Assumption~\ref{asm:overlap-entry-assumptions} requires all supercritical spikes to be distinct.
To see why this matters for the overlap estimation, suppose that $\mathcal K\subseteq\intset{r_a}$ indexes a block of equal spikes, with $d_{ai}=d$ for each $i\in\mathcal K$.
Let $m\deq|\mathcal K|\geq2$, and write $\widetilde U=(u_{a,i})_{i\in\mathcal K}$.
For every $m\times m$ orthogonal matrix $Q$,
$
        (\widetilde UQ)(\widetilde UQ)^\top
        =
        \widetilde U\widetilde U^\top.
$
Thus, replacing $\widetilde U$ by $\widetilde UQ$ leaves the signal matrix unchanged, but he corresponding overlap block becomes $Q^\top R_{\mathcal K,\cdot}$ or $R_{\cdot,\mathcal K}Q$.
These rotations can change the individual overlap magnitudes $|R_{ij}|$ while leaving the distribution of the observed data unchanged.
Consequently, these magnitudes are not identifiable when there are multiple spiked eigenvalues.
\end{remark}

Finally, we introduce a normalized correlation between the supercritical parts of two signal matrices.
Denote $\langle A, B \rangle_F \deq \Tr(A^{\top}B)$. 
Define
\[
        \rho_{\mathrm{sig}}
        \deq 
        \frac{\langle S_{1,+}, S_{2,+}\rangle_F}{\|S_{1,+}\|_F\|S_{2,+}\|_F} 
        = \frac{\sum_{k=1}^{r_1} \sum_{\ell=1}^{r_2} d_{1k} d_{2\ell} |u_{1,k}^{\top}u_{2,\ell}|^2 }{\sqrt{\sum_{k=1}^{r_1} d_{1k}^2} \sqrt{\sum_{k=1}^{r_2} d_{2k}^2}}
\]

be the normalized correlation between supercritical parts of $S_1$ and $S_2$. 
It measures the alignment between the two signal matrices, and is invariant under sign changes of the spike directions. 
When $r_1=r_2=1$, the normalized correlation reduces to the squared overlap between the two spike directions, i.e., 
\[
        \rho_{\mathrm{sig}} = |u_{1,1}^{\top}u_{2,1}|^2 = \cos^2 \angle(u_{1,1}, u_{2,1}).
\]

\subsection{A Projector-Based Overlap Estimator}

In this section, we construct estimators for a representative of the sign-equivalence class in \eqref{eq:overlap-sign-equivalence} and $\rho_{\mathrm{sig}}$.

Draw independent masks $P_1$ and $P_2$ and form $\mathcal Y_1$ and $\mathcal Y_2$ by applying \eqref{eq:split-matrix-def} to $Y_1$ and $Y_2$.
For each sample, order the eigenvalues by decreasing real part as $\lambda_{a,(1)},\ldots,\lambda_{a,(2n)}$.
For $i\in\intset{r_a}$, select
\[
        \widehat\lambda_{a,i,+}
        \deq
        \lambda_{a,(i)},
        \qquad
        \widehat\lambda_{a,i,-}
        \deq
        \lambda_{a,(2n-i+1)}.
\]
For each selected eigenvalue, choose the normalized right and left eigenvectors $\widetilde r_{a,i,\mathsf s}$ and $\widetilde\ell_{a,i,\mathsf s}$ satisfying
\[
        \|\widetilde r_{a,i,\mathsf s}\|_2=1,
        \qquad
        \widetilde\ell_{a,i,\mathsf s}^\top
        \widetilde r_{a,i,\mathsf s}
        =1,
\]
and set $\widetilde{\mathcal P}_{a,i,\mathsf s}=\widetilde r_{a,i,\mathsf s}\widetilde\ell_{a,i,\mathsf s}^\top$.
Define the estimated direction projector
\begin{equation}
\label{eq:one-sample-direction-projector}
        \widehat\Pi_{a,i}
        \deq
        \left[
        \widetilde{\mathcal P}_{a,i,+}
        +
        \widetilde{\mathcal P}_{a,i,-}
        \right]_{11},
\end{equation}
where $[A]_{11}$ is the upper-left $n\times n$ block.
At the population level,
\[
        \left[
        \mathfrak u_{a,i,+}\mathfrak u_{a,i,+}^\top
        +
        \mathfrak u_{a,i,-}\mathfrak u_{a,i,-}^\top
        \right]_{11}
        =
        u_{a,i}u_{a,i}^\top.
\]
For $i\in\intset{r_1}$ and $j\in\intset{r_2}$, define
\begin{equation}
\label{eq:overlap-magnitude-estimators}
        \widehat q_{ij}
        \deq
        \Re\Tr(\widehat\Pi_{1,i}\widehat\Pi_{2,j}),
        \qquad
        \widehat r_{ij}^{\,\mathrm{abs}}
        \deq
        \sqrt{\min\{1,\max\{0,\widehat q_{ij}\}\}}.
\end{equation}
Abbreviate $\widetilde r_{a,i}=\widetilde r_{a,i,+}$ and $\widetilde\ell_{a,i}=\widetilde\ell_{a,i,+}$.
Choose the sign of each  right eigenvector associated with a positive outlier so that its first nonzero coordinate is positive, and write
\begin{equation}\label{eq:oriented-eigenvectors}
        \widetilde r_{a,i}
        =
        \begin{pmatrix}x_{a,i}\\x'_{a,i}\end{pmatrix},
        \qquad
        \widetilde\ell_{a,i}
        =
        \begin{pmatrix}y_{a,i}\\y'_{a,i}\end{pmatrix}.
\end{equation}
With the convention $\operatorname{sgn}(0)=1$, define
\begin{equation}
\label{eq:overlap-signed-estimator}
        \widehat R^{\,\mathrm{eqv}}_{n,ij}
        \deq
        \operatorname{sgn}(y_{1,i}^\top x_{2,j})
        \widehat r_{ij}^{\,\mathrm{abs}}.
\end{equation}

Moreover, set $\widehat d_{ai}\deq2\Re\widehat\lambda_{a,i,+}$, use the samplewise estimator from Algorithm~\ref{alg:spike-estimation}, and define the  signal correlation estimator as
\begin{equation}\label{eq:wigner-overlap-signal-correlation-estimator}
\widehat{\rho}_{\mathrm{sig}}
        \deq
        \frac{\sum_{i=1}^{r_1}\sum_{j=1}^{r_2} \widehat d_{1i} \widehat d_{2j}(\widehat R^{\,\mathrm{eqv}}_{n,ij})^2}{\sqrt{\sum_{i=1}^{r_1} \widehat d_{1i}^2}\sqrt{\sum_{j=1}^{r_2} \widehat d_{2j}^2}}.
\end{equation}

Algorithm~\ref{alg:overlap-matrix-estimation} summarizes the overlap and signal correlation estimators. 

\begin{algorithm}[htbp]
\caption{(Wigner case) Overlap and signal correlation estimation}
\label{alg:overlap-matrix-estimation}
\begin{algorithmic}[1]
\Require Observed matrices $Y_1,Y_2\in\mathbb R^{n\times n}$ and known rank $r_1, r_2$.
\Ensure Overlap estimate $\widehat R_n^{\,\mathrm{eqv}} \in \mathbb R^{r_1 \times r_2}$ and signal correlation estimate $\widehat{\rho}_{\mathrm{sig}}$.
\For{$a=1,2$}
\State Draw one independent symmetric Bernoulli mask $P_a$ and form one matrix $\mathcal Y_a$ as in \eqref{eq:split-matrix-def}.
\State Compute one eigendecomposition of this same $\mathcal Y_a$ and set $\widehat d_{ai}=2\Re\widehat\lambda_{a,i,+}$ for $i\in\intset{r_a}$.
\State From these same outliers, form the normalized and oriented eigenvectors and $\widehat\Pi_{a,i}$ as in \eqref{eq:one-sample-direction-projector}.
\EndFor
\State Compute the magnitudes in \eqref{eq:overlap-magnitude-estimators} and signed representatives in \eqref{eq:oriented-eigenvectors}--\eqref{eq:overlap-signed-estimator}.
\State Compute $\widehat{\rho}_{\mathrm{sig}}$ from \eqref{eq:wigner-overlap-signal-correlation-estimator}.
\State \textbf{Return} $\widehat R_n^{\,\mathrm{eqv}} \deq (\widehat R^{\,\mathrm{eqv}}_{n,ij})_{i\in\intset{r_1}, j\in\intset{r_2}}$ and $\widehat{\rho}_{\mathrm{sig}}$.
\end{algorithmic}
\end{algorithm}

\begin{theorem}[Consistency of Algorithm~\ref{alg:overlap-matrix-estimation}]
\label{thm:entrywise-overlap}
Consider the  matrix model \eqref{eq:two-sample-model}. 
\hypersetup{allcolors=blue}For each $a\in\{1,2\}$, suppose that Assumption~\ref{asm:wigner-assumptions} holds with $(X,U,D,T,r_+)$ replaced by $(X_a,U_a,D_a,T_a,r_a)$, and suppose that Assumption~\ref{asm:overlap-entry-assumptions} holds.
The overlap estimator satisfies
\[
        \min_{\Xi_1\in\mathscr S_{r_1},\, \Xi_2\in\mathscr S_{r_2}}
        \left\|
        \widehat R_n^{\,\mathrm{eqv}}
        -
        \Xi_1R\Xi_2
        \right\|_{\max}
        =
        \smallop(1),
\]
where $\|\cdot\|_{\max}$ denotes the entrywise maximum norm. Moreover, the signal correlation estimator satisfies
\[
        |\widehat{\rho}_{\mathrm{sig}}-\rho_{\mathrm{sig}}|
        =
        \smallop(1).
\]
\end{theorem}

The proof of Theorem~\ref{thm:entrywise-overlap} relies on the following proposition, which establishes the consistency of the one-sample empirical projections $\widetilde{\mathcal P}_{k,\mathsf s}$ toward the population projections $\mathcal P_{k,\mathsf s}$ in the sense of bilinear forms.

Suppose that $d_1>\cdots>d_{r_+}>0$.
For $k\in\intset{r_+}$ and $\mathsf s\in\{+,-\}$, define the population projection
\[
        \mathcal P_{k,\mathsf s}
        \deq
        \mathfrak u_{k,\mathsf s}\mathfrak u_{k,\mathsf s}^*.
\]
From the proof of Proposition~\ref{prop:split-evec-projection}, the spiked eigenvalues $\lambda_{k,\mathsf s}(\mathcal Y)$ are algebraically simple and real with probability tending to one. 
Choose corresponding right and left eigenvectors $\widetilde r_{k,\mathsf s}$ and $\widetilde\ell_{k,\mathsf s}$ normalized so that
\[
        \mathcal Y\widetilde r_{k,\mathsf s}
        =
        \lambda_{k,\mathsf s}(\mathcal Y)\widetilde r_{k,\mathsf s},
        \qquad
        \mathcal Y^*\widetilde\ell_{k,\mathsf s}
        =
        \lambda_{k,\mathsf s}(\mathcal Y)\widetilde\ell_{k,\mathsf s},
        \qquad
        \widetilde\ell_{k,\mathsf s}^*\widetilde r_{k,\mathsf s}
        =
        1.
\]
Define the empirical projection corresponding to $\lambda_{k,\mathsf s}(\mathcal Y)$ by
\[
        \widetilde{\mathcal P}_{k,\mathsf s}
        \deq
        \widetilde r_{k,\mathsf s}\widetilde\ell_{k,\mathsf s}^*.
\]

\begin{proposition}
\label{prop:split-evec-projection}
Suppose that Assumption~\ref{asm:wigner-assumptions} hold
and that $d_1>\cdots>d_{r_+}>0$.
For every $k\in\intset{r_+}$, $\mathsf s\in\{+,-\}$, and deterministic vectors $a_n,b_n\in\mathbb C^{2n}$ with $\|a_n\|_2\vee\|b_n\|_2\leq1$,
\begin{equation}
\label{eq:evec-proj-main-bilinear}
        a_n^*
        (\widetilde{\mathcal P}_{k,\mathsf s}-\mathcal P_{k,\mathsf s})
        b_n
        = 
        \smallop(1).
\end{equation}
\end{proposition}

\section{Rectangular Model}
\label{sec:rect-extension}

In this section, we extend the random-splitting method to the rectangular model. 
We estimate the supercritical spikes, the overlaps between the corresponding left and right signal directions, and correlation between the supercritical signal parts.
To simplify the notations, we reuse some notation from Section~\ref{sec:spiked-wigner-model} if there is no  confusion.

We consider the rectangular signal plus noise model
\[
        Y=X+S\in\R^{p\times n},
        \qquad
        S=UDV^\top
        =\sum_{k = 1}^r d_ku_kv_k^\top,
\]
where $U\in\R^{p\times r}$ and $V\in\R^{n\times r}$ have orthonormal columns.
Let $P$ be a $p\times n$ matrix with i.i.d. $\mathrm{Bernoulli}(1/2)$ entries and be independent of $X$.
We construct the asymmetrized matrix
\begin{equation}
\label{eq:rect-Y-def}
        \mathcal Y\deq
        \begin{pmatrix}
        0&P\circ Y\\
        ((\one_p\one_n^\top-P)\circ Y)^\top&0
        \end{pmatrix}.
\end{equation}
For $k\in\intset{r}$ and $\mathsf s\in\{+,-\}$, define
\begin{equation}
\label{eq:rect-signal-vectors}
        \mathfrak u_{k,\mathsf s}
        \deq\frac1{\sqrt2}(u_k^{\top},\, \mathsf s v_k^{\top})^{\top},
        \qquad
        \alpha_{k,\mathsf s}\deq\mathsf s\frac{d_k}{2}.
\end{equation}
Let $\mathcal U$ collect these $2r$ orthonormal vectors and let $\mathcal D$ be diagonal with entries $\alpha_{k,\mathsf s}$.  
Define $\mathcal X$ and $\mathcal E$ similarly to \eqref{eq:split-matrix-decomposition}.
We have the decomposition
\[
        \mathcal Y = \mathcal X+\mathcal E+\mathcal U\mathcal D\mathcal U^\top.
\]
Let $T=(t_{ij})_{p\times n}$ be the variance profile and set
\begin{equation*}
        \mathbb T\deq\frac12
        \begin{pmatrix}0&T\\T^\top&0\end{pmatrix},
        \qquad
        b_{\star,n}\deq\sqrt{\norm{\mathbb T}_2}
        =\sqrt{\frac{\norm{T}_2}{2}}.
\end{equation*}

For the rectangular model, we impose the following assumptions on the noise and the signals. 
These assumptions are analogues of Assumption~\ref{asm:wigner-assumptions} in Section~\ref{sec:spiked-wigner-model}. 

\begin{assumption}\label{asm:rectangular-assumptions}
\leavevmode
\begin{assumptionparts}
\item (Noise moments)
\label{asm:rect-noise}
As $p,n\to\infty$, $p/n\to\phi\in(0,\infty)$.  The variables
$\{X_{ij}:i\in\intset{p},j\in\intset{n}\}$ are mutually independent and
satisfy
\[
        \E X_{ij}=0,
        \qquad \E X_{ij}^2=t_{ij},
        \qquad \frac{C_*}{n}\leq t_{ij}\leq\frac{C^*}{n}
\]
for fixed $0<C_*\leq C^*<\infty$.  For every fixed integer $k\geq1$, there is a constant $C_k<\infty$ such that
\[
        \sup_n\max_{i\in\intset{p},\,j\in\intset{n}}
        \E\abs{\sqrt n X_{ij}}^k \leq C_k.
\]

\item (Incoherent signal directions)
\label{asm:rect-deloc}
The deterministic unit vectors
$u_1,\ldots,u_r\in\R^p$ and $v_1,\ldots,v_r\in\R^n$ are orthonormal and satisfy
\[
        \mu_{U,n} \deq \max_{k\in\intset{r}}\norm{u_k}_\infty=\smallo(1),
        \qquad
        \mu_{V,n} \deq \max_{k\in\intset{r}}\norm{v_k}_\infty = \smallo(1).
\]

\item (Population spikes)
\label{asm:rect-spikes}
The rank $r$ is fixed.
The singular spikes $d_1\geq\cdots\geq d_r>0$ are fixed deterministic
constants, independent of $n$, and $d_{\max}\deq d_1\leq C$ for a fixed constant $C$.
There exists a fixed integer $r_+\in\intset{r}$ such that, for some fixed $\delta>0$ and all sufficiently large $n$,
\begin{equation*}
        \min_{k\in\intset{r_+}}\frac{d_k}{2}
        \geq b_{\star,n}+\delta,
        \qquad
        \max_{r_+<k\leq r}\frac{d_k}{2}
        \leq b_{\star,n}-\delta.
\end{equation*}
The second condition is omitted when $r_+=r$.
\end{assumptionparts}
\end{assumption}

Fix $\vartheta\in(0,1/2)$ and define the rectangular error scale
\begin{equation*}
        \eta_n\deq n^{-1/2+\vartheta}
        +d_{\max}(\mu_{U,n}+\mu_{V,n}).
\end{equation*}
Under Assumptions~\ref{asm:rect-deloc}--\ref{asm:rect-spikes},
$\eta_n=\smallo(1)$.
The following proposition establishes the rectangular analogue of the phase transition in Proposition~\ref{prop:outlier-location}.
\begin{proposition}
\label{prop:rect-phase}
Suppose Assumption~\ref{asm:rectangular-assumptions} hold.
For every fixed $\kappa\in(0,\delta)$, with probability tending to one, $\mathcal Y$ has exactly $2r_+$ eigenvalues, counted with algebraic multiplicity, in $\{z\in\C:\abs z\geq b_{\star,n}+\kappa\}$, and
\begin{equation*}
        \max_{k\in\intset{r_+}, \; \mathsf s\in\{+,-\}}
        \big|\lambda_{k,\mathsf s}(\mathcal Y) - \alpha_{k,\mathsf s}\big|
        =\bigop(\eta_n),
\end{equation*}
where $\alpha_{k,\mathsf s} = \mathsf s d_k / 2$.
There are no other eigenvalues of $\mathcal Y$ in $\{z\in\mathbb C : |z| \geq b_{\star,n} + \kappa\}$.
\end{proposition}

Based on Proposition~\ref{prop:rect-phase}, we design Algorithm~\ref{alg:rect-spike} to estimate the singular spikes by applying the random splitting procedure.
The consistency is established in Theorem~\ref{thm:rect-spike-consistency}.

\begin{algorithm}[H]
\caption{(Rectangular case) Spike estimation by random splitting}
\label{alg:rect-spike}
\begin{algorithmic}[1]
\Require A rectangular observation $Y\in\R^{p\times n}$ and a known rank $r_+$.
\Ensure Estimates $\widehat d_1,\ldots,\widehat d_{r_+}$.
\State Draw one independent mask $P$ with i.i.d. $\mathrm{Bernoulli}(1/2)$ entries.
\State Form $\mathcal Y$ as in \eqref{eq:rect-Y-def}.
\State Compute the eigenvalues of $\mathcal Y$ and set
\[
       \widehat d_k\deq2\Re\widehat\lambda_{k,+}, \qquad k\in\intset{r_+}.
\]
\State \textbf{Return} $\widehat d_1,\ldots,\widehat d_r$.
\end{algorithmic}
\end{algorithm}

\begin{theorem}[Consistency of Algorithm~\ref{alg:rect-spike}]
\label{thm:rect-spike-consistency}
Under Assumption~\ref{asm:rectangular-assumptions}, Algorithm~\ref{alg:rect-spike} produces consistent estimates of the singular spikes:
\begin{equation*}
        \max_{k\in\intset{r_+}}\abs{\widehat d_k-d_k}
        =\bigop(\eta_n)=\smallop(1).
\end{equation*}
\end{theorem}
Theorem~\ref{thm:rect-spike-consistency} is a direct consequence of Proposition~\ref{prop:rect-phase}.

We next consider the rectangular model for two observed matrices.
For $a\in\{1,2\}$, one observes
\begin{equation}
\label{eq:rect-two-sample-model}
        Y_a=X_a+S_a,
        \qquad S_a=U_aD_aV_a^\top,
        \qquad D_a=\diag(d_{a1},\ldots,d_{ar}),
\end{equation}
where $U_a,V_a$ are deterministic and the $X_1, X_2$ are independent.
Their variance profiles may differ.  
For each noise matrix, define 
\begin{equation*}
        T_a\deq(\E X_{a,ij}^2)_{i,j},
        \qquad
        \mathbb T_a\deq\frac12
        \begin{pmatrix}0&T_a\\T_a^\top&0\end{pmatrix},
        \qquad
        b_{\star,a,n}\deq\sqrt{\norm{\mathbb T_a}_2}
        =\sqrt{\frac{\norm{T_a}_2}{2}}.
\end{equation*}
\hypersetup{allcolors=blue}
For each $a\in\{1,2\}$, suppose that $(X_a,U_a,V_a,D_a,T_a)$ satisfies Assumption~\ref{asm:rectangular-assumptions} with supercritical count $r_{a,+}$.
The two signal matrices have the same fixed rank $r$.
For notational simplicity, write $r_a\deq r_{a,+}$.
The quantities $r_1,r_2\in\intset{r}$ are assumed known and need not be equal.
Define the supercritical signal parts and their left and right singular vectors by
\[
        S_{a,+}\deq\sum_{i\in\intset{r_a}}d_{ai}u_{a,i}v_{a,i}^\top,
        \qquad
        U_{a,+}\deq(u_{a,1},\ldots,u_{a,r_a}),
        \qquad
        V_{a,+}\deq(v_{a,1},\ldots,v_{a,r_a}).
\]
The left and right overlap matrices are
\begin{equation*}
        R_U\deq U_{1,+}^\top U_{2,+},
        \qquad R_V\deq V_{1,+}^\top V_{2,+}.
\end{equation*}
Define the equivalence classes $[R_U]_{+}$ and $[R_V]_{+}$ similarly to \eqref{eq:overlap-sign-equivalence}.

Moreover, let the normalized signal correlation between the $S_1$ and $S_2$ be
\begin{equation*}
        \rho_{\mathrm{sig}}
        \deq
        \frac{\langle S_{1,+},S_{2,+}\rangle_F}{\|S_{1,+}\|_F\|S_{2,+}\|_F}
        =
        \frac{\sum_{i=1}^{r_1}\sum_{j=1}^{r_2}d_{1i}d_{2j}(R_U)_{ij}(R_V)_{ij}}
        {(\sum_{i=1}^{r_1} d_{1i}^2)^{1/2}(\sum_{j=1}^{r_2} d_{2j}^2)^{1/2}}.
\end{equation*}
Our goal is to estimate this correlation and representatives of $[R_U]_{+}$ and $[R_V]_{+}$.

Similar to Assumption~\ref{asm:overlap-entry-assumptions} , we also require the following separation condition on the singular spikes. 
\begin{assumption}
\label{asm:rect-two-separation}
For some fixed $\tau>0$,
\[
        \min_{a\in\intset{2}}\min_{i,j\in\intset{r_a},\,i\neq j}\abs{d_{ai}-d_{aj}}\geq\tau.
\]
For each $a$, the second minimum is interpreted as $+\infty$ when $r_a=1$.
\end{assumption}
\hypersetup{allcolors=blue}
Draw independent masks $P_1,P_2$ with i.i.d. $\mathrm{Bernoulli}(1/2)$ entries, independently of $X_1,X_2$.
For each $a\in\{1,2\}$, use $P_a$ to form $\mathcal Y_a$ as in \eqref{eq:rect-Y-def} and write its decomposition as
\begin{equation*}
        \mathcal Y_a=\mathcal X_a+\mathcal E_a
        +\mathcal U_a\mathcal D_a\mathcal U_a^\top.
\end{equation*}
For $i\in\intset{r_a}$ and $\mathsf s \in \{+,-\}$,  choose right and left eigenvectors  associated with $\widehat \lambda_{a,i,\mathsf s}$ normalized by
\begin{equation}
\label{eq:rect-left-right-normalization}
        \norm{\widetilde r_{a,i,\mathsf s}}_2=1,
        \quad
        \widetilde\ell_{a,i,\mathsf s}^\top
        \widetilde r_{a,i,\mathsf s}=1,
\end{equation}
and set
\[
\widetilde{\mathcal P}_{a,i,\mathsf s}
\deq\widetilde r_{a,i,\mathsf s}\widetilde\ell_{a,i,\mathsf s}^\top.
\]

With $\mathfrak u_{a,i,\mathsf s}$ defined by applying
\eqref{eq:rect-signal-vectors} samplewise, we have
\begin{equation}
\label{eq:rect-block-projector-identity}
        \mathfrak u_{a,i,+}\mathfrak u_{a,i,+}^\top
        +\mathfrak u_{a,i,-}\mathfrak u_{a,i,-}^\top
        =\begin{pmatrix}
          u_{a,i}u_{a,i}^\top&0\\
          0&v_{a,i}v_{a,i}^\top
        \end{pmatrix}.
\end{equation}
This identity motivates extracting both diagonal blocks from the same pair
of empirical outlier projections.

For a $(p+n)\times(p+n)$ matrix $M$, let $[M]_{11}$ and $[M]_{22}$
denote its upper left $p\times p$ and  lower right $n\times n$ blocks.  Define
\begin{equation}
\label{eq:rect-projector-estimators}
        \widehat\Pi^U_{a,i}
        \deq[\widetilde{\mathcal P}_{a,i,+}
              +\widetilde{\mathcal P}_{a,i,-}]_{11},
        \qquad 
        \widehat\Pi^V_{a,i}
        \deq[\widetilde{\mathcal P}_{a,i,+}
              +\widetilde{\mathcal P}_{a,i,-}]_{22}.
\end{equation}
Their targets are $u_{a,i}u_{a,i}^\top$ and $v_{a,i}v_{a,i}^\top$.  
For $i\in\intset{r_1}$ and $j\in\intset{r_2}$, define 
\begin{equation}
\label{eq:rect-magnitude-estimators}
\begin{aligned}
        \widehat q^U_{ij}
        &\deq\Re\Tr(\widehat\Pi^U_{1,i}\widehat\Pi^U_{2,j}),
        &\widehat r^{U,\mathrm{abs}}_{ij}
        &\deq\sqrt{\min\{1,\max\{0, \widehat q^U_{ij}\}\}},\\
        \widehat q^V_{ij}
        &\deq\Re\Tr(\widehat\Pi^V_{1,i}\widehat\Pi^V_{2,j}),
        &\widehat r^{V,\mathrm{abs}}_{ij}
        &\deq\sqrt{\min\{1,\max\{0, \widehat q^V_{ij}\}\}}.
\end{aligned}
\end{equation}
Choose the sign of each right eigenvector associated with a positive outlier so that its first nonzero coordinate is positive.
Change the sign of the corresponding left eigenvector to preserve \eqref{eq:rect-left-right-normalization}, and partition the eigenvectors as
\begin{equation*}
        \widetilde r_{a,i,+}
        =\begin{pmatrix}x^U_{a,i}\\x^V_{a,i}\end{pmatrix},
        \qquad
        \widetilde\ell_{a,i,+}
        =\begin{pmatrix}y^U_{a,i}\\y^V_{a,i}\end{pmatrix}.
\end{equation*}
With $\operatorname{sgn}(0)=1$, define 
\begin{equation}
\label{eq:rect-signed-estimators}
        \widehat R^{U,\mathrm{eqv}}_{ij}
        \deq\operatorname{sgn}((y^U_{1,i})^\top x^U_{2,j})
        \widehat r^{U,\mathrm{abs}}_{ij},
        \qquad
        \widehat R^{V,\mathrm{eqv}}_{ij}
        \deq\operatorname{sgn}((y^V_{1,i})^\top x^V_{2,j})
        \widehat r^{V,\mathrm{abs}}_{ij}.
\end{equation}
Finally, set $\widehat d_{ai} \deq 2 \Re \widehat \lambda_{a,i,+}$ for $i\in\intset{r_a}$ and define the signal correlation estimator as
\begin{equation}
\label{eq:rect-plug-in}
        \widehat\rho_{\mathrm{sig}}
        \deq
        \frac{\sum_{i,j}\widehat d_{1i}\widehat d_{2j}
        \widehat R^{U,\mathrm{eqv}}_{ij}
        \widehat R^{V,\mathrm{eqv}}_{ij}}
        {(\sum_i\widehat d_{1i}^2)^{1/2}
         (\sum_j\widehat d_{2j}^2)^{1/2}}.
\end{equation}

The following algorithm summarizes the complete signal correlation estimator for the rectangular case, and its consistency is established in Theorem~\ref{thm:rect-overlap-consistency}.
\begin{algorithm}[H]
\caption{(Rectangular case)  Overlap and signal correlation estimation}
\label{alg:rect-correlation}
\begin{algorithmic}[1]
\Require Two observations $Y_1,Y_2\in\R^{p\times n}$ and known $r_1, r_2$.
\Ensure Left/right overlap estimates and the signal correlation estimate $\widehat\rho_{\mathrm{sig}}$.
\For{$a=1,2$}
\State Draw a Bernoulli mask $P_a$ and form $\mathcal Y_a$ as in \eqref{eq:rect-Y-def}.
\State Compute the eigendecomposition of $\mathcal Y_a$; set $\widehat d_{ai}=2\Re\widehat\lambda_{a,i,+}$ for $i\in\intset{r_a}$.
\State Form the normalized eigenvectors and $\widehat\Pi^U_{a,i},\widehat\Pi^V_{a,i}$ in \eqref{eq:rect-left-right-normalization}--\eqref{eq:rect-projector-estimators}.
\EndFor
\State Compute the left/right overlap magnitudes in \eqref{eq:rect-magnitude-estimators} and signed estimates in \eqref{eq:rect-signed-estimators}.
\State Compute $\widehat\rho_{\mathrm{sig}}$ from \eqref{eq:rect-plug-in}.
\State \textbf{Return} $\widehat R^{U,\mathrm{eqv}}$, $\widehat R^{V,\mathrm{eqv}}$, and $\widehat\rho_{\mathrm{sig}}$.
\end{algorithmic}
\end{algorithm}

\begin{theorem}
\label{thm:rect-overlap-consistency}
Consider the rectangular model \eqref{eq:rect-two-sample-model}. 
For each $a\in\intset{2}$, suppose that Assumptions~\ref{asm:rectangular-assumptions} and \ref{asm:rect-two-separation} hold for $Y_a$ with its own parameters. 
Then the left and right overlap estimators satisfy
\begin{align*}
        &\min_{\Xi_1\in\mathscr S_{r_1},\,\Xi_2\in\mathscr S_{r_2}} 
        \left\|\widehat R^{U,\mathrm{eqv}}-\Xi_1R_U\Xi_2\right\|_{\max}
        =\smallop(1),\\
        &\min_{\Xi_1\in\mathscr S_{r_1},\,\Xi_2\in\mathscr S_{r_2}} 
        \left\|\widehat R^{V,\mathrm{eqv}}-\Xi_1R_V\Xi_2\right\|_{\max}
        =\smallop(1).
\end{align*}
The signal correlation estimator satisfies
\[
    \widehat\rho_{\mathrm{sig}}-\rho_{\mathrm{sig}}=\smallop(1).
\]
\end{theorem}

\section{Numerical Simulations}
\label{sec:numerical-simulations}

This section perform a Monte Carlo simulation for the proposed estimators in Algorithms~\ref{alg:spike-estimation}--\ref{alg:overlap-matrix-estimation}.
Additional simulation results for Propositions~\ref{prop:outlier-location}
and~\ref{prop:rect-phase} and Algorithms~\ref{alg:rect-spike}--\ref{alg:rect-correlation} are provided in Supplementary Material.

\paragraph*{Common setup}
For each configuration, use $n\in\{250,500,750,1000\}$ and $B=500$ independent data replications.
Taking $\kappa \in \{1,4,8\}$ and define the variance profile $T^{(\kappa)}$ as follows:
\begin{align*}
        T^{(\kappa)}
        =
        \frac1n
        \begin{pmatrix}
        \kappa J_{n/2} & J_{n/2}\\
        J_{n/2} & J_{n/2}
        \end{pmatrix},
\end{align*}
where $J_m \deq \one_m\one_m^\top$.
For $T=(t_{ij})\in\{T^{(\kappa)}, \kappa = 1,4,8\}$, set $\mathbb T\deq T/2$ and $b_{\star,n}\deq\sqrt{\|\mathbb T\|_2}$.
For $1\leq i\leq j\leq n$, generate the noise matrix $X$ by
$
        X_{ij}
        =
        \sqrt{t_{ij}}\,\xi_{ij},
        X_{ji}
        =
        X_{ij}.
$
In the primary experiment, the variables $\xi_{ij}$ are independent standard Gaussian variables.
In the robustness experiment, replace them with independent Rademacher variables satisfying $\Pr(\xi_{ij}=1)=\Pr(\xi_{ij}=-1)=1/2$.
For every use of random splitting, draw $P_{ij}\iidsim\mathrm{Bernoulli}(1/2)$ for $1\leq i\leq j\leq n$ and set $P_{ji}=P_{ij}$.
All auxiliary masks are independent of the observations and are redrawn independently across replications.

Set $m=n/2$. 
For $i\in\intset{m}$ and $k\in\intset{r-1}$, define the orthonormal vectors by
\[
        h_1(i)=\frac1{\sqrt m},
        \qquad
        h_{k+1}(i)=
        \sqrt{\frac2m}
        \cos\!\left\{
        \frac{\pi k(i-1/2)}{m}
        \right\},
\]
where $h_k(i)$ denotes the $i$-th entry of $h_k$.
Then, let 
\[
u_k = (\sqrt{0.9}h_k^{\top}, \sqrt{0.1} h_k^{\top})^{\top},\qquad k\in\intset{r}.
\]
These vectors satisfy the required incoherence condition \ref{asm:delocalization}.

\subsection{Simulation for Algorithm~\ref{alg:spike-estimation}}
We use two different variance profiles $T^{(1)}$ and $T^{(4)}$.
Set $r=3$ and  $U=(u_1,u_2,u_3)$.
Consider the normalized spike configuration $\boldsymbol\gamma=(2.00,1.60,1.30)$ and $(2.00,1.50,1.50)$.
With $d_k=2b_{\star,n}\gamma_k$ and $D=\diag(d_1,d_2,d_3)$, form $Y=X+UDU^\top$ and apply Algorithm~\ref{alg:spike-estimation} for $n\in\{250,500,750,1000\}$ to assess its finite-sample performance.

Let $\widehat d_k^{(b)}$ denote the estimate returned in replication $b$.
We report the mean absolute error and empirical standard deviation 
\[
        \widehat{\mathsf{MAE}}_k
        \deq
        \frac1{B}
        \sum_{b=1}^B
        \big|\widehat d_k^{(b)} - d_k\big|,
        \qquad
        \widehat{\mathsf{SD}}_k
        \deq
        \bigg[
        \frac1{B-1}
        \sum_{b=1}^B
        \big(\widehat d_k^{(b)}-\overline d_k\big)^2
        \bigg]^{1/2},
\]
where $\overline d_k \deq B^{-1} \sum_{b=1}^B \widehat d_k^{(b)}$ denotes the mean of estimates.

The simulation results are summarized in Table~\ref{tab:algorithm1-results}.
It reports the mean absolute errors and
empirical standard deviations for Algorithm~\ref{alg:spike-estimation} under
the homogeneous and heterogeneous variance profiles for both normalized spike configurations.
Each entry has the form $\widehat{\mathsf{MAE}}_k\;(\widehat{\mathsf{SD}}_k)$.
For both normalized spike configurations and variance profiles, $\widehat{\mathsf{MAE}}_k$ and $\widehat{\mathsf{SD}}_k$ generally decrease as $n$ increases.
The Gaussian and Rademacher results are close, indicating the robustness of the proposed estimator.

We also compare Algorithm~\ref{alg:spike-estimation} with two alternative methods.
For readers's convenience, we provide a brief description of each method.
To our knowledge, there are no existing methods for estimating spikes for Wigner type matrices under a general variance profile.
Both methods are designed for the homogeneous noise setting.
For the comparison, we set $r=1$ and $d_1=2b_{\star,n}\gamma$ with $\gamma=1.50$, and use $T^{(\kappa)}$ for $\kappa\in\{1,4,8\}$ under both Gaussian and Rademacher noise at $n=1000$.

\begin{itemize}
        \item \citet{BykhovskayaGorinSodin2025}'s \texttt{BGS25} estimator: This estimator is based on the classical deformed Wigner outlier map.
        Let $\lambda_1(Y)\geq\cdots\geq\lambda_n(Y)$, and let $\ell\approx0.81$ be the positive number satisfying
        \[
                \int_{-2}^{-\ell}\frac{\sqrt{4-x^2}}{2\pi}\,\mathrm{d}x
                =
                \int_{\ell}^{2}\frac{\sqrt{4-x^2}}{2\pi}\,\mathrm{d}x
                =\frac14,
        \]
        and set
        \[
                \sigma_0^2
                \deq
                \int_{-\ell}^{\ell}
                \frac{x^2\sqrt{4-x^2}}{2\pi}\,\mathrm{d}x.
        \]
        They estimate the noise variance by
        \[
                \widehat\sigma^{\,2}
                \deq
                \frac{1}{\sigma_0^2 n}
                \sum_{i=\lfloor n/4\rfloor+1}^{\lfloor 3n/4\rfloor}
                \lambda_i(Y)^2.
        \]
        The proposed estimator for $d_k$ is
        \[
        \widehat d_{k,\mathtt{BGS25}}
        =
        \begin{cases}
        \displaystyle
        \frac{\lambda_k(Y)+ \sqrt{\lambda_k(Y)^2-4\widehat\sigma^{2}}}{2},
        & \lambda_k(Y)>2\widehat\sigma,\\[7pt]
        0, & \lambda_k(Y)\leq2\widehat\sigma.
        \end{cases}
        \]

        \item The second comparison method is based on the outlier equation for an additive finite rank perturbation. \citet{BenaychGeorgesNadakuditi2011} show that a supercritical population spike $d$ and its associated sample spike $\lambda$ satisfies $d + 1/m(\lambda) = 0$, where $m(\lambda)$ denotes the Stieltjes transform of the noise eigenvalue distribution. 
        We estimate $m(z)$ from the empirical bulk spectrum after deleting the leading $r$ eigenvalues:
        \[
                \widehat m(z)
                \deq
                \frac{1}{n-r}\sum_{j=r+1}^n\frac{1}{\lambda_j(Y)-z}.
        \]
        Then, we define the spike estimator by 
        \[ 
                \widehat d_{k,\mathtt{BGN11}} \deq -\frac{1}{\widehat m\big(\lambda_k(Y)\big)}.
        \]
        The closely related method was proposed by \citet{BaiDing2012} and \citet{FanGuoZheng2020} for the spiked covariance model and spiked correlation model, respectively.
\end{itemize}

The simulation results for the comparison methods are summarized in Figure~\ref{fig:algorithm1-comparison}.
At $n=1000$, the three estimators have similar distributions under
$T^{(1)}$, whereas Algorithm~\ref{alg:spike-estimation} is centered much
closer to the target under $T^{(4)}$ and $T^{(8)}$ for both noise laws.

\subsection{Simulation for Algorithm~\ref{alg:overlap-matrix-estimation}}
Set $r=2$ and
\[
        A
        =
        \begin{pmatrix}
        0.8&0.3\\
        -0.2&0.6
        \end{pmatrix},
        \qquad
        R(\omega)
        \deq
        \omega\frac{A}{\|A\|_2},
        \qquad
        \omega\in\{0.25,0.60,0.90\}.
\]
For each $\omega$, let
\[
        U_1=(u_1,u_2),
        \qquad
        U_2
        =
        U_1R(\omega)
        +
        (u_3,u_4)
        \left(I_2-R(\omega)^\top R(\omega)\right)^{1/2},
\]
where the square root is the positive semidefinite square root.
Then $U_1^\top U_1=U_2^\top U_2=I_2$ and $U_1^\top U_2=R(\omega)$, and columns of $U_1$ and $U_2$ satisfy the incoherence condition \ref{asm:delocalization}.
Use $D_1=\diag(5,4)$, $D_2=\diag(6,4.5)$,
and the profile pairs 
\[ 
        (T_1,T_2) 
        \quad \in \quad 
        \big\{ (T^{(1)},T^{(1)}), \; (T^{(1)},T^{(4)}), \; (T^{(4)},T^{(4)})\big\}.
\]
Generate $X_1$ and $X_2$ independently and form $Y_a=X_a+U_aD_aU_a^\top$ for $a\in\{1,2\}$.
The three values of $\omega$ represent weak, moderate, and strong subspace overlap, respectively.
For every configuration, $\|R(\omega)\|_2=\omega<1$.
This matrix has nonzero off-diagonal entries and therefore tests recovery of the full overlap matrix.
The corresponding population signal correlations $\rho_{\mathrm{sig}}$ are
$0.0510$, $0.2939$, and $0.6612$, respectively.
Generate $B=500$ independent replications and apply Algorithm~\ref{alg:overlap-matrix-estimation} to estimate the overlap matrix $R(\omega)$ and the signal correlation $\rho_{\mathrm{sig}}$.

For each replication $b$, define
\[
        \mathsf E_{\mathrm{eqv}}^{(b)}
        \deq
        \min_{\Xi_1,\Xi_2\in\mathscr S_2}
        \left\|
        \widehat R_n^{\,\mathrm{eqv},(b)}
        -
        \Xi_1R(\omega)\Xi_2
        \right\|_{\max},
        \qquad
        \mathsf E_{\mathrm{sig}}^{(b)}
        \deq
        \Big|
        \widehat\rho_{\mathrm{sig}}^{(b)}-\rho_{\mathrm{sig}}
        \Big|.
\]
We report the mean errors
$B^{-1}\sum_{b=1}^B\mathsf E_{\mathrm{eqv}}^{(b)}$ and
$B^{-1}\sum_{b=1}^B\mathsf E_{\mathrm{sig}}^{(b)}$, see Tables~\ref{tab:algorithm2-results} and \ref{tab:algorithm2-signal-correlation-results}. 
For both metrics, every variance-profile pair and value of $\omega$ exhibits a
decrease in mean error as $n$ increases, and the Gaussian and Rademacher results
are close. 

\begin{table}[htbp]
\centering
% \scriptsize
\footnotesize
\setlength{\tabcolsep}{2.5pt}
\begin{tabular}{@{}lccc@{\hspace{8pt}}ccc@{}}
\toprule
&\multicolumn{3}{c}{Homogeneous $T^{(1)}$}
 & \multicolumn{3}{c}{Heterogeneous $T^{(4)}$}\\
\cmidrule(lr){2-4}\cmidrule(lr){5-7}
$n$ & $k=1$ & $k=2$ & $k=3$
 & $k=1$ & $k=2$ & $k=3$\\
\midrule
&\multicolumn{6}{l}{$\boldsymbol\gamma=(2.00,1.60,1.30)$}\\
\cmidrule{2-7}
& \multicolumn{6}{l}{Gaussian}\\
250 & $0.0860\;(0.1037)$ & $0.1070\;(0.1195)$ & $0.1602\;(0.1469)$
 & $0.1534\;(0.1895)$ & $0.1807\;(0.2097)$ & $0.2653\;(0.2474)$\\
500 & $0.0599\;(0.0723)$ & $0.0686\;(0.0798)$ & $0.0994\;(0.1119)$
 & $0.1066\;(0.1341)$ & $0.1183\;(0.1430)$ & $0.1934\;(0.2058)$\\
750 & $0.0465\;(0.0581)$ & $0.0557\;(0.0669)$ & $0.0771\;(0.0941)$
 & $0.0893\;(0.1123)$ & $0.1011\;(0.1224)$ & $0.1415\;(0.1651)$\\
1000 & $0.0424\;(0.0519)$ & $0.0503\;(0.0601)$ & $0.0670\;(0.0799)$
 & $0.0762\;(0.0946)$ & $0.0875\;(0.1061)$ & $0.1203\;(0.1410)$\\
& \multicolumn{6}{l}{Rademacher}\\
250 & $0.0909\;(0.1084)$ & $0.1021\;(0.1124)$ & $0.1608\;(0.1483)$
 & $0.1460\;(0.1788)$ & $0.1661\;(0.1947)$ & $0.2657\;(0.2496)$\\
500 & $0.0604\;(0.0739)$ & $0.0698\;(0.0830)$ & $0.1045\;(0.1135)$
 & $0.1051\;(0.1330)$ & $0.1232\;(0.1482)$ & $0.1931\;(0.2079)$\\
750 & $0.0495\;(0.0609)$ & $0.0550\;(0.0659)$ & $0.0789\;(0.0920)$
 & $0.0860\;(0.1071)$ & $0.0969\;(0.1204)$ & $0.1516\;(0.1764)$\\
1000 & $0.0408\;(0.0512)$ & $0.0480\;(0.0565)$ & $0.0614\;(0.0781)$
 & $0.0799\;(0.1003)$ & $0.0886\;(0.1090)$ & $0.1164\;(0.1415)$\\
\midrule
&\multicolumn{6}{l}{$\boldsymbol\gamma=(2.00,1.50,1.50)$}\\
\cmidrule{2-7}
& \multicolumn{6}{l}{Gaussian}\\
250 & $0.0854\;(0.1013)$ & $0.0836\;(0.1028)$ & $0.2302\;(0.1391)$
 & $0.1551\;(0.1877)$ & $0.1494\;(0.1803)$ & $0.3848\;(0.2405)$\\
500 & $0.0580\;(0.0710)$ & $0.0625\;(0.0723)$ & $0.1271\;(0.0825)$
 & $0.1151\;(0.1437)$ & $0.1135\;(0.1272)$ & $0.2438\;(0.1653)$\\
750 & $0.0482\;(0.0589)$ & $0.0535\;(0.0581)$ & $0.0977\;(0.0631)$
 & $0.0920\;(0.1130)$ & $0.0979\;(0.1080)$ & $0.1855\;(0.1259)$\\
1000 & $0.0404\;(0.0506)$ & $0.0515\;(0.0527)$ & $0.0836\;(0.0578)$
 & $0.0753\;(0.0940)$ & $0.0893\;(0.0947)$ & $0.1469\;(0.1010)$\\
& \multicolumn{6}{l}{Rademacher}\\
250 & $0.0906\;(0.1082)$ & $0.0772\;(0.0928)$ & $0.2188\;(0.1331)$
 & $0.1551\;(0.1870)$ & $0.1506\;(0.1825)$ & $0.4023\;(0.2506)$\\
500 & $0.0594\;(0.0733)$ & $0.0611\;(0.0688)$ & $0.1239\;(0.0816)$
 & $0.1057\;(0.1345)$ & $0.1217\;(0.1353)$ & $0.2461\;(0.1704)$\\
750 & $0.0473\;(0.0588)$ & $0.0501\;(0.0540)$ & $0.0988\;(0.0662)$
 & $0.0889\;(0.1118)$ & $0.0974\;(0.1069)$ & $0.1819\;(0.1264)$\\
1000 & $0.0403\;(0.0499)$ & $0.0475\;(0.0492)$ & $0.0781\;(0.0547)$
 & $0.0736\;(0.0907)$ & $0.0906\;(0.0878)$ & $0.1385\;(0.0981)$\\
\bottomrule
\end{tabular}
\caption{Mean absolute errors and empirical standard deviations, shown in
parentheses, for Algorithm~\ref{alg:spike-estimation} under the homogeneous
profile $T^{(1)}$ (left) and heterogeneous profile $T^{(4)}$ (right),
for two normalized spike configurations, based on $B=500$ replications.}
\label{tab:algorithm1-results}
\end{table}

\begin{figure}[htbp]
\centering
\includegraphics[width=\textwidth]{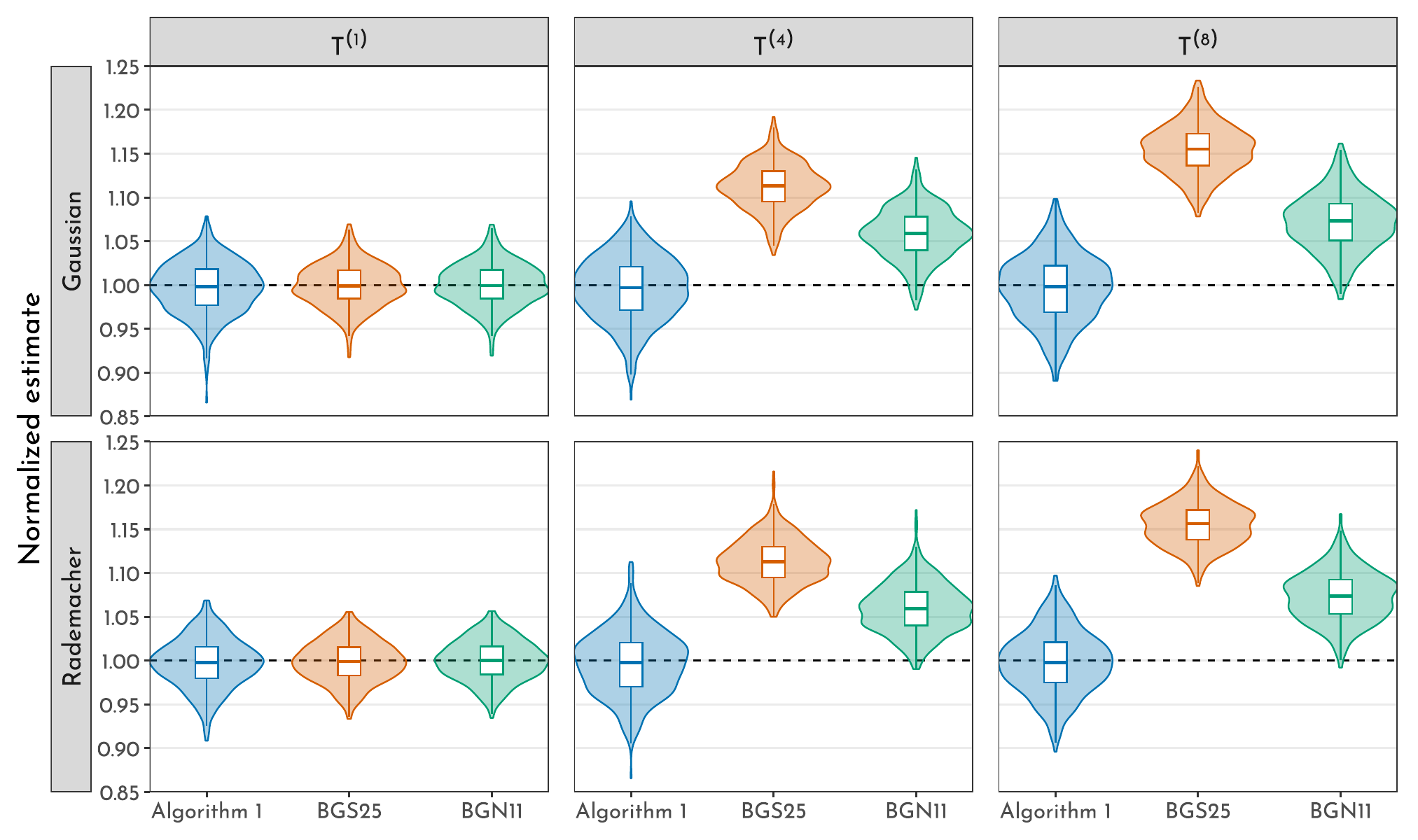}
\caption{Violin plots of the normalized spike estimates
$\widehat d_1/d_1$ at $n=1000$ for Algorithm~\ref{alg:spike-estimation},
\texttt{BGS25}, and \texttt{EmpM} under $T^{(1)}$, $T^{(4)}$, and
$T^{(8)}$, with Gaussian noise (top) and Rademacher noise (bottom).  Each
violin is based on $B=500$ replications; the embedded boxplot shows the
median and interquartile range, and the dashed horizontal line marks the
target value $1$.  The normalized spike strength is $\gamma=1.50$.}
\label{fig:algorithm1-comparison}
\end{figure}

\begin{table}[htbp]
\centering
% \scriptsize
\footnotesize
\setlength{\tabcolsep}{2.5pt}
\begin{tabular}{@{}l*{9}{c}@{}}
\toprule
\multicolumn{1}{c}{}
 & \multicolumn{3}{c}{$(T^{(1)},T^{(1)})$}
 & \multicolumn{3}{c}{$(T^{(1)},T^{(4)})$}
 & \multicolumn{3}{c}{$(T^{(4)},T^{(4)})$}\\
\cmidrule(lr){2-4}\cmidrule(lr){5-7}\cmidrule(lr){8-10}
$n$ 
 & $\omega=0.25$ & $\omega=0.60$ & $\omega=0.90$
 & $\omega=0.25$ & $\omega=0.60$ & $\omega=0.90$
 & $\omega=0.25$ & $\omega=0.60$ & $\omega=0.90$\\
\midrule
&\multicolumn{9}{l}{Gaussian}\\
250 & $0.0397$ & $0.0495$ & $0.0658$ & $0.0583$ & $0.0708$ & $0.0885$ & $0.0804$ & $0.1038$ & $0.1230$\\
500 & $0.0270$ & $0.0337$ & $0.0438$ & $0.0403$ & $0.0485$ & $0.0620$ & $0.0549$ & $0.0700$ & $0.0901$\\
750 & $0.0206$ & $0.0273$ & $0.0347$ & $0.0302$ & $0.0400$ & $0.0489$ & $0.0435$ & $0.0595$ & $0.0718$\\
1000 & $0.0177$ & $0.0228$ & $0.0302$ & $0.0275$ & $0.0326$ & $0.0421$ & $0.0372$ & $0.0478$ & $0.0617$\\ 
\cmidrule{2-10} 
&\multicolumn{9}{l}{Rademacher}\\
250 & $0.0409$ & $0.0483$ & $0.0636$ & $0.0600$ & $0.0726$ & $0.0865$ & $0.0778$ & $0.1050$ & $0.1296$\\
500 & $0.0262$ & $0.0343$ & $0.0453$ & $0.0403$ & $0.0493$ & $0.0570$ & $0.0555$ & $0.0717$ & $0.0912$\\
750 & $0.0210$ & $0.0274$ & $0.0347$ & $0.0311$ & $0.0382$ & $0.0497$ & $0.0434$ & $0.0568$ & $0.0734$\\
1000 & $0.0181$ & $0.0241$ & $0.0307$ & $0.0265$ & $0.0337$ & $0.0434$ & $0.0380$ & $0.0470$ & $0.0627$\\
\bottomrule
\end{tabular}
\caption{Mean sign-equivalence-class errors for the Wigner overlap estimator in
Algorithm~\ref{alg:overlap-matrix-estimation} under three variance-profile
pairs, based on $B=500$ replications.}
\label{tab:algorithm2-results}
\end{table}

\begin{table}[htbp]
\centering
\footnotesize
\setlength{\tabcolsep}{2.5pt}
\begin{tabular}{@{}l*{9}{c}@{}}
\toprule
\multicolumn{1}{c}{}
 & \multicolumn{3}{c}{$(T^{(1)},T^{(1)})$}
 & \multicolumn{3}{c}{$(T^{(1)},T^{(4)})$}
 & \multicolumn{3}{c}{$(T^{(4)},T^{(4)})$}\\
\cmidrule(lr){2-4}\cmidrule(lr){5-7}\cmidrule(lr){8-10}
$n$
 & $\omega=0.25$ & $\omega=0.60$ & $\omega=0.90$
 & $\omega=0.25$ & $\omega=0.60$ & $\omega=0.90$
 & $\omega=0.25$ & $\omega=0.60$ & $\omega=0.90$\\
\midrule
&\multicolumn{9}{l}{Gaussian}\\
250 & $0.0051$ & $0.0119$ & $0.0252$ & $0.0077$ & $0.0171$ & $0.0286$ & $0.0108$ & $0.0235$ & $0.0370$\\
500 & $0.0036$ & $0.0077$ & $0.0131$ & $0.0053$ & $0.0118$ & $0.0173$ & $0.0069$ & $0.0148$ & $0.0226$\\
750 & $0.0029$ & $0.0061$ & $0.0102$ & $0.0042$ & $0.0096$ & $0.0127$ & $0.0054$ & $0.0128$ & $0.0168$\\
1000 & $0.0025$ & $0.0053$ & $0.0078$ & $0.0037$ & $0.0074$ & $0.0108$ & $0.0052$ & $0.0114$ & $0.0145$\\
\cmidrule{2-10}
&\multicolumn{9}{l}{Rademacher}\\
250 & $0.0054$ & $0.0110$ & $0.0246$ & $0.0081$ & $0.0172$ & $0.0301$ & $0.0101$ & $0.0247$ & $0.0394$\\
500 & $0.0037$ & $0.0079$ & $0.0143$ & $0.0055$ & $0.0117$ & $0.0174$ & $0.0072$ & $0.0155$ & $0.0236$\\
750 & $0.0030$ & $0.0063$ & $0.0095$ & $0.0044$ & $0.0089$ & $0.0130$ & $0.0061$ & $0.0117$ & $0.0176$\\
1000 & $0.0026$ & $0.0056$ & $0.0077$ & $0.0036$ & $0.0078$ & $0.0103$ & $0.0050$ & $0.0111$ & $0.0143$\\
\bottomrule
\end{tabular}
\caption{Mean absolute errors of the Wigner signal correlation estimator
$\widehat\rho_{\mathrm{sig}}$ in
Algorithm~\ref{alg:overlap-matrix-estimation} under three variance-profile
pairs, based on $B=500$ replications.}
\label{tab:algorithm2-signal-correlation-results}
\end{table}

\section{Proof Strategy for the Spiked Wigner Model}
\label{sec:proof-strategy}

\subsection{Auxiliary Lemmas}

We start from  the decomposition in \eqref{eq:split-signal-decomposition}.
The $\mathcal U\mathcal D\mathcal U^\top$ term has the desired eigenvalues $\alpha_{k,\mathsf s}=\mathsf s d_k/2$.
The following lemma controls the remainder $\mathcal E$ in the decomposition \eqref{eq:split-signal-decomposition}.

\begin{lemma}
\label{lem:mask-remainder-bound}
Recall the definition of $\mathcal E$ in \eqref{eq:split-signal-decomposition}.
Suppose that Assumption~\ref{asm:delocalization} and \ref{asm:signal} hold,  then
\begin{equation*}
        \|\mathcal E\|_2
        =
        \bigop(\mu_n)
        =
        \smallop(1).
\end{equation*}
\end{lemma}

Since $\mathcal X$ is non-Hermitian, its eigenvalues cannot be controlled directly by a Hermitian local law.
For $z\in\mathbb C$, using Girko's Hermitization trick \citep{Girko1985}, we analyze the spectrum of $\mathcal X$ by introducing the $4n\times 4n$ Hermitian linearization
\begin{equation}
\label{eq:hermitization}
        H_z
        \deq
        \begin{pmatrix}
        0&\mathcal X-zI_{2n}\\
        (\mathcal X-zI_{2n})^*&0
        \end{pmatrix}.
\end{equation}
For any $\zeta\in\mathbb C^+$, the resolvent of $H_z$ is denoted by
\[
        \mathfrak G_z(\zeta)
        \deq (H_z-\zeta I_{4n})^{-1}.
\]
The eigenvalues of $H_z$ are the signed singular values of $\mathcal X-zI_{2n}$.
In particular, $z$ is an eigenvalue of $\mathcal X$ if and only if zero is an eigenvalue of $H_z$.

Following \citet{AjankiErdosKruger2019}, we will show that the resolvent $\mathfrak G_z(\zeta)$ can be approximated by a deterministic matrix $M_z(\zeta)$, which is the
unique solution with positive imaginary part to the matrix Dyson
equation (MDE)
\begin{equation}
\label{eq:MDE}
        -M_z(\zeta)^{-1}
        =
        \zeta I_{4n}-\E H_z+\mathcal S[M_z(\zeta)],
        \qquad
        \mathcal S[B]
        \deq
        \E[\mathcal H B\mathcal H],
\end{equation}
where $\mathcal H \deq H_z-\E H_z$.

\begin{lemma}[MDE solution]
\label{lem:block-structure}
The MDE solution to \eqref{eq:MDE} has the form
\[
        M_z(\zeta)
        =
        \begin{pmatrix}
        M_{11}&0&M_{13}&0\\
        0&M_{11}&0&M_{13}\\
        M_{31}&0&M_{11}&0\\
        0&M_{31}&0&M_{11}
        \end{pmatrix},
\]
where $M_{11}=\diag\bigl(m_{z,1}(\zeta),\ldots,m_{z,n}(\zeta)\bigr)$, 
$M_{13}=\diag\bigl(w_{z,1}(\zeta),\ldots,w_{z,n}(\zeta)\bigr)$, 
and $M_{31}=\diag\bigl(\widetilde w_{z,1}(\zeta),\ldots,\widetilde w_{z,n}(\zeta)\bigr)$
are diagonal matrices, with
\begin{align}
        m_{z,i}(\zeta)
        &=
        -
        \frac{\zeta+a_{z,i}(\zeta)}
        {(\zeta+a_{z,i}(\zeta))^2-|z|^2},
        \qquad 
        a_{z,i}(\zeta)
        =
        \sum_{k=1}^n \mathbb T_{ik}m_{z,k}(\zeta),
        \label{eq:MDE_m_a}\\
        w_{z,i}(\zeta)
        &=
        \frac{z}
        {(\zeta+a_{z,i}(\zeta))^2-|z|^2},
        \qquad
        \widetilde w_{z,i}(\zeta)
        =
        \frac{\overline z}
        {(\zeta+a_{z,i}(\zeta))^2-|z|^2}.\notag
\end{align}
\end{lemma}

The explicit coordinate representation in Lemma~\ref{lem:block-structure} is the starting point for the support analysis in Lemma~\ref{lem:support}.
It reduces the matrix valued MDE to the scalar system \eqref{eq:MDE_m_a}, whose stability near zero determines whether the self consistent density has a gap around the origin.
Set $N\deq4n$ and let $\nu_z$ be the self consistent density of states associated with $M_z$, defined by
\begin{equation}\label{eq:self consistent-density}
        \frac1N\Tr M_z(\zeta)
        =
        \int_{\mathbb R}\frac{1}{x-\zeta}\,\dif\nu_z(x).
\end{equation}

The following lemma shows that the support of the self-consistent density is away from 0 by a distance of $\delta'$.

\begin{lemma}
\label{lem:support}
Fix $\varepsilon>0$.
Under Assumption~\ref{asm:moment}, there exists $\delta'(\varepsilon)>0$ such that
\[
        \dist(0,\supp\nu_z)>\delta'(\varepsilon)
\]
for every $z\in\mathbb C$ with $|z|>b_{\star,n}+\varepsilon$.
\end{lemma}

Throughout the remainder of the paper, we adopt the notion of stochastic domination notation introduced by \citet{Erds2013averaging}.
\begin{definition}[Stochastic domination]
Let
\[
        \xi
        =
        \bigl(\xi_N(u):N\in\mathbb N,\ u\in\mathcal I_N\bigr),
        \qquad
        \chi
        =
        \bigl(\chi_N(u):N\in\mathbb N,\ u\in\mathcal I_N\bigr)
\]
be two families of real random variables, where $\chi$ is nonnegative and $\mathcal I_N$ is a possibly $N$-dependent parameter set.
We say that $\xi$ is \emph{stochastically dominated} by $\chi$, uniformly in $u$, if, for arbitrary small $\varepsilon>0$ and large $D>0$, there exists $N_0(\varepsilon,D)$ such that
\[
        \sup_{u\in\mathcal I_N}
        \Pr\!\left(
        |\xi_N(u)|>N^\varepsilon \chi_N(u)
        \right)
        \leq
        N^{-D}
\]
for all $N\geq N_0(\varepsilon,D)$.
We write $\xi=\bigo_\prec(\chi)$ or $\xi\prec \chi$ when $\xi$ is stochastically dominated by $\chi$ uniformly in $u$.
If $\xi$ and $\chi$ are deterministic, then $\xi\prec \chi$ means that, for every $\varepsilon>0$,
\[
        |\xi_N(u)|
        \leq
        N^\varepsilon \chi_N(u)
\]
uniformly in $u$ for all sufficiently large $N\geq N_0(\varepsilon)$.
\end{definition}

\begin{lemma}[Local law of $H_z$]
\label{lem:local-law}
Fix $C_z<\infty$ and assume that $|z|\leq C_z$.
Fix $\tau>0$ and define
\[
        \mathbb D_{\mathrm{out}}^\tau
        \deq
        \left\{
        \zeta\in\mathbb C^+:
        |\zeta|\leq N^{C_0},
        \ \dist(\zeta,\supp\nu_z)\geq N^{-\tau}
        \right\},
\]
where $C_0>100$ is an arbitrary fixed constant.
Under Assumption~\ref{asm:moment},
\[
        \sup_{\zeta\in\mathbb D_{\mathrm{out}}^\tau}
        (1+|\zeta|)^2
        \left|
        x^*
        \bigl(\mathfrak G_z(\zeta)-M_z(\zeta)\bigr)
        y
        \right|
        \prec
        \frac{\|x\|_2\|y\|_2}{\sqrt N}
\]
for deterministic vectors $x,y\in\mathbb C^N$, and
\[
        \sup_{\zeta\in\mathbb D_{\mathrm{out}}^\tau}
        (1+|\zeta|)^2
        \left|
        \frac1N
        \Tr\!\left[
        B\bigl(\mathfrak G_z(\zeta)-M_z(\zeta)\bigr)
        \right]
        \right|
        \prec
        \frac{\|B\|_2}{N}
\]
for every deterministic matrix $B\in\mathbb C^{N\times N}$.
\end{lemma}
\begin{proof}
The lemma  is actually a consequence of the local law of \citet[Theorem~2.1]{Erdos2019random} applied to \eqref{eq:hermitization}. One needs to check if our model satisfies the Assumptions (A)--(D) in \cite{Erdos2019random}. Especially, we shall verify some  cumulant and correlation decay conditions for our matrix entries. Such verification is deferred to Section~\ref{sec:proof-local-law} of the supplement. 
 \end{proof}

For $z\in\mathbb C$ such that $\varrho(\mathcal X)<|z|$, define
\begin{equation}
\label{eq:G-calG-def}
        G(z)\deq(z^2I-\mathsf A\mathsf B)^{-1},
        \qquad
        \mathcal G(z)\deq(z^2I-\mathsf B\mathsf A)^{-1}.
\end{equation}
Notice that $G(z)\mathsf A=\mathsf A\mathcal G(z)$ and $\mathsf B G(z)=\mathcal G(z)\mathsf B$.
Using these identities, the resolvent of $\mathcal X$ can be expressed in block form as
\begin{equation}
        \label{eq:block-inverse}
        (zI_{2n}-\mathcal X)^{-1}
        =
        \begin{pmatrix}
        zG(z)&\;G(z)\mathsf A\\
        \mathsf B G(z)&\;z\mathcal G(z)
        \end{pmatrix}.
\end{equation}
The following lemma shows that the smallest singular values of $z-\mathcal{X}$ and $z-\mathcal{X}-\mathcal{E}$ are bounded away from 0.

\begin{lemma}
\label{lem:uniform-exterior-singular-gap}
Fix $\varepsilon>0$ and $C<\infty$,  let 
\begin{equation}\label{eq:annulus-def}
       \mathscr A_{n,\varepsilon,C}
       \deq 
\{z\in\mathbb C:b_{\star,n}+\varepsilon\leq|z|\leq C\}.
\end{equation}
Under Assumption~\ref{asm:moment}, there exists $\gamma_{\varepsilon,C}>0$ such that
\begin{equation}
\label{eq:uniform-exterior-gap-split-noise}
        \inf_{z\in \mathscr A_{n,\varepsilon,C}}
        \sigma_{\min}(zI_{2n}-\mathcal X)
        \geq
        \gamma_{\varepsilon,C}
\end{equation}
with probability tending to one.
Consequently, if the matrix $\mathcal E$ satisfies $\|\mathcal E\|_2=\smallop(1)$, then
\[
        \inf_{z\in \mathscr A_{n,\varepsilon,C}}
        \sigma_{\min}(zI_{2n}-\mathcal X-\mathcal E)
        \geq
        \frac12\gamma_{\varepsilon,C}
\]
with probability tending to one.
\end{lemma}
Equation~\eqref{eq:block-inverse} and Lemma~\ref{lem:local-law} yield the following general quadratic form estimates.
\begin{lemma}[Quadratic form estimate] 
\label{lem:quadratic form-convergence}
Fix $\varepsilon>0$ and $C<\infty$, and let $\mathscr K_n$ be any deterministic compact subset of the annulus $\mathscr A_{n,\varepsilon,C}$ defined in \eqref{eq:annulus-def}
Under Assumption~\ref{asm:moment}, for every pair of deterministic vectors $x,y\in\mathbb C^n$ satisfying $\|x\|_2 \vee \|y\|_2 = \bigo(1)$,                  
\begin{align}
        \sup_{z\in \mathscr K_n}
        \left|x^*\bigl\{G(z)-z^{-2}I_n\bigr\}y\right|
        &\prec n^{-1/2},
        &
        \sup_{z\in \mathscr K_n}
        \left|x^*\bigl\{\mathcal G(z)-z^{-2}I_n\bigr\}y\right|
        &\prec n^{-1/2},
        \label{eq:quantitative-G-calG-forms}\\
        \sup_{z\in \mathscr K_n}\left|x^*\mathsf B G(z)y\right|
        &\prec n^{-1/2},
        &
        \sup_{z\in \mathscr K_n}\left|x^*\mathsf A\mathcal G(z)y\right|
        &\prec n^{-1/2}.
        \label{eq:quantitative-BG-AcalG-forms}
\end{align}
\end{lemma}
Lemma~\ref{lem:quadratic form-convergence} is used directly in the outlier analysis and is also an input to Lemma~\ref{lem:resolvent-est}.

By Lemmas~\ref{lem:mask-remainder-bound} and \ref{lem:uniform-exterior-singular-gap}, $zI_{2n}-\mathcal X-\mathcal E$ is invertible with high probability for $z \in \mathscr A_{n,\varepsilon,C}$. 
Set
\begin{equation}\label{eq:perturbed-split-noise-resolvent-def}
        \widetilde R_{\mathcal X}(z)\deq(zI_{2n}-\mathcal X-\mathcal E)^{-1}.
\end{equation}
The following lemma shows that the projected resolvent of $\mathcal X$ is close to its deterministic limit uniformly over the exterior annulus.

\begin{lemma}[Resolvent estimates]
\label{lem:resolvent-est}
Fix $\varepsilon>0$ and $C<\infty$, and let $\mathscr K_n$ be any deterministic compact subset of the annulus $\mathscr A_{n,\varepsilon,C}$ defined in \eqref{eq:annulus-def}.
Suppose that Assumptions~\ref{asm:moment}--\ref{asm:delocalization} hold, that $r$ is fixed, and $d_{\max}=\bigo(1)$.
Then, for every deterministic unit vector $a\in\mathbb C^{2n}$,
\begin{align*}
        \sup_{z\in \mathscr K_n}
        \left\|
        a^*\widetilde R_{\mathcal X}(z)\mathcal U-\frac1z a^*\mathcal U
        \right\|_2
        &=
        \bigop(\eta_n). 
\end{align*}
Here $\eta_n$ is defined in \eqref{eq:outlier-error-scale}, and $\mathcal U$ is the signal basis defined in \eqref{eq:normalized-spike-basis}.
\end{lemma}

Lemma~\ref{lem:resolvent-est} follows from the block inverse \eqref{eq:block-inverse}, the resolvent identity, Lemmas~\ref{lem:uniform-exterior-singular-gap} and \ref{lem:quadratic form-convergence}.

\subsection{Proof of Proposition~\ref{prop:outlier-location}}

Recall the decomposition \eqref{eq:split-signal-decomposition} and the definition of $\widetilde R_{\mathcal X}(z)$ from \eqref{eq:perturbed-split-noise-resolvent-def}. 
Set
\begin{equation*}
        F_n(z)
        \deq
        I_{2r}-\mathcal D \mathcal U^\top \widetilde R_{\mathcal X}(z)\mathcal U,
        \qquad
        F_0(z)
        \deq
        I_{2r}-\frac1z\mathcal D.
\end{equation*}
A standard norm bound for Wigner type matrices
\citep[Corollary 1.2]{Erdos2019bounds}, the deterministic bounds
\[
\|\mathcal E\|_2\leq \frac12\|S\|_{\mathrm F}
\leq \frac{\sqrt{r}}{2}d_{\max}
\quad\text{and}\quad
\|\mathcal U\mathcal D\mathcal U^\top\|_2=\frac{d_{\max}}2,
\]
and \eqref{eq:split-signal-decomposition} give a deterministic $\mathsf C <\infty$
such that, for every fixed $D>0$ and all sufficiently large $n$,
\begin{equation}\label{eq:outlier-norm-cutoff}
        \Pr(\|\mathcal Y\|_2\geq \mathsf C )\leq n^{-D}.
\end{equation}
Fix $\varepsilon>0$ and work on $\mathscr A_{n,\varepsilon,\mathsf C }$ (see \eqref{eq:annulus-def}).
With probability tending to one, Lemmas~\ref{lem:mask-remainder-bound} and \ref{lem:uniform-exterior-singular-gap} make $zI_{2n}-\mathcal X-\mathcal E$ uniformly invertible on $\mathscr A_{n,\varepsilon,\mathsf C }$.
Applying Lemma~\ref{lem:resolvent-est} implies that
\begin{equation}
\label{eq:outlier-F-comparison}
        \sup_{z\in\mathscr A_{n,\varepsilon,\mathsf C }}
        \|F_n(z)-F_0(z)\|_2
        =
        \bigop(\eta_n).
\end{equation}
The deterministic function has the explicit determinant
\begin{equation}\label{eq:det-F0}
        \det F_0(z)
        =
        \prod_{k=1}^r
        \prod_{\mathsf s\in\{+,-\}}
        \left(
        1-\frac{\alpha_{k,\mathsf s}}{z}
        \right),
\end{equation}
so its only exterior zeros are the signed spike locations $\alpha_{k,\mathsf s}$.

Let
\[
        \Lambda_+
        \deq
        \{\alpha_{k,\mathsf s}:k\in\intset{r_+},\ \mathsf s\in\{+,-\}\}
\]
be the set of distinct signed spike locations.
Let $\delta>0$ be constant in Assumption~\ref{asm:signal}.  
Fix $\kappa\in(0,\delta)$ and take $\varepsilon=\kappa$ in the annuals $\mathscr A_{n,\varepsilon,\mathsf C}$.  
Increase the deterministic norm cutoff $\mathsf C $ if necessary so that
$
        \max_{\alpha\in\Lambda_+}|\alpha|<\mathsf C .
$
For each $\alpha\in\Lambda_+$, let $m_\alpha$ be the multiplicity of $\alpha$ as a zero of $\det F_0$. 
The equation \eqref{eq:det-F0} shows that $\alpha$ is a zero of $\det F_0$ of order $m_\alpha$, and $\sum_{\alpha\in\Lambda_+}m_\alpha=2r_+$.
Fix $\zeta>0$.  
Equation~\eqref{eq:outlier-F-comparison} gives $M\equiv M_\zeta<\infty$ such that, for all sufficiently large $n$,
\[
        \Pr\left(
        \sup_{z\in\mathscr A_{n,\kappa,\mathsf C }}
        \|F_n(z)-F_0(z)\|_2
        \leq M\eta_n
        \right)
        \geq1-\zeta.
\]
Let $\Omega_n(M)$ be the intersection of the event in the previous probability bound and 
\[
        \left\{
        \inf_{z\in\mathscr A_{n,\kappa,\mathsf C }}
        \sigma_{\min}(zI_{2n}-\mathcal X-\mathcal E)
        \geq\frac12\gamma_{\kappa,\mathsf C }
        \right\} 
        \bigcap 
        \{\|\mathcal Y\|_2<\mathsf C \}.
\]
Lemmas~\ref{lem:mask-remainder-bound} and \ref{lem:uniform-exterior-singular-gap} and the bound \eqref{eq:outlier-norm-cutoff} give
\[
        \Pr\bigl(\Omega_n(M)\bigr)
        \geq1-\zeta-\smallo(1).
\]
On this event, the determinant factorization \eqref{eq:outlier-determinant-factorization}
holds there.

Set $c_r\deq2^{1/(2r)}-1$ and choose $C_\eta>\mathsf C M/c_r$. 
For $\alpha\in\mathbb C$ and $t>0$, set $D_\alpha(t) \deq \{z\in\mathbb C:|z-\alpha|<t\}.$
Let
\[ 
        \mathscr W_n
        \deq
        \bigcup_{\alpha\in\Lambda_+}
        D_\alpha(C_\eta\eta_n)
\]
Since $\eta_n=\smallo(1)$, we have $C_{\eta}\eta_n < \delta - \kappa$ for all sufficiently large $n$, and then the closed disks $\overline{D_\alpha(C_\eta\eta_n)}$ lie in $\mathscr A_{n,\kappa,\mathsf C }$, and they are pairwise disjoint.
For $z\in\mathscr A_{n,\kappa,\mathsf C }\setminus\mathscr W_n$,  if $k\leq r_+$, then $\alpha_{k,\mathsf s}\in\Lambda_+$, and hence
$|z-\alpha_{k,\mathsf s}|\geq C_\eta\eta_n$;
if $k>r_+$, then Assumption~\ref{asm:signal} and the definition of $\mathscr A_{n,\kappa,\mathsf C }$ give $|z-\alpha_{k,\mathsf s}|\geq |z|-|\alpha_{k,\mathsf s}|\geq \delta+\kappa$.
Consequently, for all sufficiently large $n$,
\begin{align*}
        \|F_0(z)^{-1}\|_2
        =
        \max_{k,\mathsf s}
        \frac{|z|}{|z-\alpha_{k,\mathsf s}|}
        \leq
        \max\left\{
        \frac{\mathsf C}{C_\eta\eta_n},
        \frac{\mathsf C}{\delta+\kappa}
        \right\}
        =
        \frac{\mathsf C}{C_\eta\eta_n}.
\end{align*}
Consequently, on $\Omega_n(M)$,
\begin{equation}\label{eq:outlier-F-relative-bound}
        \sup_{z\in\mathscr A_{n,\kappa,\mathsf C }\setminus\mathscr W_n}
        \left\|F_0(z)^{-1}\{F_n(z)-F_0(z)\}\right\|_2
        \leq
        \frac{\mathsf C M}{C_\eta}
        <c_r < 1.
\end{equation}
Hence, for $z\in \mathscr A_{n,\kappa,\mathsf C }\setminus\mathscr W_n$, the matrix
\[
        F_n(z)
        =
        F_0(z)
        \left[I_{2r}+F_0(z)^{-1}\{F_n(z)-F_0(z)\}\right]
\]
is invertible, and thus $\det F_n(z) \neq 0$.

For each $\alpha\in\Lambda_+$, both $F_n$ and $F_0$ are analytic on a
neighborhood of $\overline{D_\alpha(C_\eta\eta_n)}$ on $\Omega_n(M)$.  Set
\[
        K_n(z)
        \deq
        F_0(z)^{-1}\{F_n(z)-F_0(z)\},
        \qquad
        z\in\partial D_\alpha(C_\eta\eta_n).
\]
The bound \eqref{eq:outlier-F-relative-bound} gives
\[
        \sup_{z\in\partial D_\alpha(C_\eta\eta_n)}
        \|K_n(z)\|_2
        <
        c_r
        =
        2^{1/2r}-1,
\]
and
\[
        F_n(z)=F_0(z)\{I_{2r}+K_n(z)\},
        \qquad
        z\in\partial D_\alpha(C_\eta\eta_n).
\]
Writing $\lambda_1(K_n(z)),\ldots,\lambda_{2r}(K_n(z))$ for the eigenvalues of
$K_n(z)$ and using $|\lambda_j(K_n(z))|\leq\|K_n(z)\|_2$, we obtain
\begin{align*}
        \left|\det\{I_{2r}+K_n(z)\}-1\right|
        &=
        \left|
        \prod_{j=1}^{2r}\{1+\lambda_j(K_n(z))\}-1
        \right|\\
        &\leq
        \{1+\|K_n(z)\|_2\}^{2r}-1\\
        &<
        (1+c_r)^{2r}-1
        =1.
\end{align*}
Therefore, for every $z\in\partial D_\alpha(C_\eta\eta_n)$,
\begin{align*}
        |\det F_n(z)-\det F_0(z)|
        =
        |\det F_0(z)|
        \left|\det\{I_{2r}+K_n(z)\}-1\right|
        <
        |\det F_0(z)|.
\end{align*}
Rouch\'e's theorem shows that $\det F_n$ and $\det F_0$ have the same
number of zeros inside $D_\alpha(C_\eta\eta_n)$, counted with multiplicity.
This disk contains no zero of $\det F_0$ other than $\alpha$, whose order
is $m_\alpha$.  Hence it contains exactly $m_\alpha$ zeros of $\det F_n$,
counted with multiplicity.

On $\Omega_n(M)$, the matrix determinant lemma gives
\begin{equation}
\label{eq:outlier-determinant-factorization}
        \det(zI_{2n}-\mathcal Y)
        =
        \det(zI_{2n}-\mathcal X-\mathcal E)\det F_n(z),
        \qquad 
        z\in \mathscr A_{n,\varepsilon,\mathsf C}.
\end{equation}
Since $zI_{2n}-\mathcal X-\mathcal E$ is invertible on $\mathscr A_{n,\varepsilon,\mathsf C}$, the order of each zero of $\det F_n$ equals its order as a zero of $\det(zI_{2n}-\mathcal Y)$. 
Collecting the eigenvalues in $D_{\alpha}(C_{\eta}\eta_n)$ for $\alpha\in\Lambda_+$, and repeating them according to algebraic multiplicity, we may label them as
$\lambda_{k,\mathsf s}(\mathcal Y)$ for $k\in\intset{r_+}$ and $\mathsf s\in\{+,-\}$,
so that, on $\Omega_n(M)$,
\[
        \max_{k\in\intset{r_+},\;\mathsf s\in\{+,-\}}
        \left|
        \lambda_{k,\mathsf s}(\mathcal Y)-\alpha_{k,\mathsf s}
        \right|
        \leq C_\eta\eta_n.
\]
Since $\Pr(\Omega_n(M))\geq1-\zeta-\smallo(1)$ and $\zeta>0$ was arbitrary,
this proves \eqref{eq:outlier-location-rate}.

Since $\sum_{\alpha\in\Lambda_+}m_\alpha=2r_+$, there are exactly $2r_+$
eigenvalues of $\mathcal Y$ in $\mathscr A_{n,\kappa,\mathsf C }$ and no
others there.  The norm cutoff gives
$|\lambda|\leq\|\mathcal Y\|_2<\mathsf C $ for every eigenvalue $\lambda$
of $\mathcal Y$.  Therefore, there are exactly $2r_+$ eigenvalues of
$\mathcal Y$ in $\{z:|z|\geq b_{\star,n}+\kappa\}$ and no others there.

\subsection{Proof of Proposition~\ref{prop:split-evec-projection}}

Fix $k\in\intset{r_+}$ and $\mathsf s\in\{+,-\}$.
Choose a fixed $\varepsilon\in(0,\delta)$. Assumption~\ref{asm:signal} then gives $b_{\star,n}+\varepsilon<\min_{k\in\intset{r_+}} d_k/2$ for all sufficiently large $n$.
Set
\begin{align}
\label{eq:Delta-def}
        \Delta_{k,\mathsf s}
        &\deq
        1\wedge|\alpha_{k,\mathsf s}|
        \wedge(|\alpha_{k,\mathsf s}|-b_{\star,n}-\varepsilon)
        \wedge
        \min_{\substack{
        \ell\in\intset{r},\ \mathsf t\in\{+,-\}\\
        (\ell,\mathsf t)\neq(k,\mathsf s)}}
        |\alpha_{k,\mathsf s}-\alpha_{\ell,\mathsf t}|,\\
\label{eq:outlier-contours-def}
        D_{k,\mathsf s}
        &\deq
        \left\{z\in\mathbb C:
        |z-\alpha_{k,\mathsf s}|<\frac{\Delta_{k,\mathsf s}}4
        \right\},
        \qquad
        \Gamma_{k,\mathsf s}
        \deq
        \partial D_{k,\mathsf s}.
\end{align}
The contour $\Gamma_{k,\mathsf s}$ is positively oriented.
The spike locations $\alpha_{k,\mathsf s}$ are pairwise distinct, so Assumption~\ref{asm:signal} implies that $\Delta_{k,\mathsf s}$ is bounded away from zero for all sufficiently large $n$, uniformly in $(k,\mathsf s)$. 
Consequently, $\eta_n=\smallo(1)$ gives
\[ 
\eta_n=\smallo(\Delta_{k,\mathsf s})
\] 
uniformly in $(k,\mathsf s)$.
Proposition~\ref{prop:outlier-location} and $\eta_n=\smallo(\Delta_{k,\mathsf s})$ imply that, with probability tending to one, the disk $D_{k,\mathsf s}$ contains exactly one eigenvalue of $\mathcal Y$, and $\Gamma_{k,\mathsf s}$ contains none.
Hence this eigenvalue is algebraically simple with probability tending to one.

Since the spectrum of $\mathcal Y$ is finite, there exists an open neighborhood $U_{k,\mathsf s}$ of $\overline{D_{k,\mathsf s}}$ such that
$ U_{k,\mathsf s}\cap\operatorname{spec}(\mathcal Y) = \{\lambda_{k,\mathsf s}(\mathcal Y)\}.
$
For $z\in U_{k,\mathsf s}\setminus
\{\lambda_{k,\mathsf s}(\mathcal Y)\}$, the resolvent has the following Laurent expansion
\[
(zI_{2n}-\mathcal Y)^{-1}
=
\frac{
\widetilde{\mathcal P}_{k,\mathsf s}
}{
z-\lambda_{k,\mathsf s}(\mathcal Y)
}
+
H_{k,\mathsf s}(z),
\]
where $H_{k,\mathsf s}$ is analytic on $U_{k,\mathsf s}$.
The Cauchy residue theorem therefore gives
\begin{equation}\label{eq:outlier-projection-integral}
        \frac{1}{2\pi\ii}
        \oint_{\Gamma_{k,\mathsf s}}
        (zI_{2n}-\mathcal Y)^{-1}\,\dif z 
        = 
        \widetilde{\mathcal P}_{k,\mathsf s}.
\end{equation}
Lemmas~\ref{lem:mask-remainder-bound} and \ref{lem:uniform-exterior-singular-gap} also imply that $\widetilde R_{\mathcal X}(z)$ is analytic on $U_{k, \mathsf s}$.
Set
\[
        L(z)
        \deq
        \mathcal D^{-1}
        -
        \mathcal U^*\widetilde R_{\mathcal X}(z)\mathcal U,
        \qquad
        L_0(z)
        \deq
        \mathcal D^{-1}-\frac1zI_{2r},
        \qquad
        K_0(z)
        \deq
        L_0(z)^{-1}.
\]
The definition of $D_{k,\mathsf s}$ in \eqref{eq:outlier-contours-def} therefore gives
\begin{equation}\label{eq:K0-bound}
        \sup_{z\in\Gamma_{k,\mathsf s}}
        \|K_0(z)\|_2
        = 
        \sup_{z\in\Gamma_{k,\mathsf s}}
        \max_{\ell,\mathsf t}
        \bigg|
        \frac{z \alpha_{\ell, \mathsf t}}{z-\alpha_{\ell, \mathsf t}}
        \bigg|
        \leq
        \frac{C}{\Delta_{k,\mathsf s}}.
\end{equation}
Write
$
        E(z)
        \deq
        L(z)-L_0(z).
$
Lemma~\ref{lem:resolvent-est} yields
\begin{equation}\label{eq:E-bound}
        \sup_{z\in\Gamma_{k,\mathsf s}}
        \|E(z)\|_2
        =
        \bigop(\eta_n).
\end{equation}

This, together with \eqref{eq:K0-bound} and \eqref{eq:E-bound}, implies that
\begin{align*}
        q_n
        \deq
        \sup_{z\in\Gamma_{k,\mathsf s}}
        \|K_0(z)E(z)\|_2 
        =
        \bigop\!\left(
        \frac{\eta_n}{\Delta_{k,\mathsf s}}
        \right)
        =
        \smallop(1).
\end{align*}
For every $z\in\Gamma_{k,\mathsf s}$,
\[
        L(z)
        =
        L_0(z)\{I_{2r}+K_0(z)E(z)\}.
\]
On the event $\{q_n\leq 1/2\}$, the Neumann series gives
\[
        \sup_{z\in\Gamma_{k,\mathsf s}}
        \|\{I_{2r}+K_0(z)E(z)\}^{-1}\|_2
        \leq
        \sum_{j=0}^\infty q_n^j
        \leq        
        \frac{1}{1-q_n}
        \leq
        2.
\]
Because $q_n=\smallop(1)$, this event has probability tending to one. Hence $L(z)$ is invertible on $\Gamma_{k,\mathsf s}$ with probability tending to one, and
\begin{equation}\label{eq:L-inverse-bound}
        \sup_{z\in\Gamma_{k,\mathsf s}}
        \|L(z)^{-1}\|_2
        =
        \sup_{z\in\Gamma_{k,\mathsf s}}
        \|\{I_{2r}+K_0(z)E(z)\}^{-1}K_0(z)\|_2
        =
        \bigop(\Delta_{k,\mathsf s}^{-1})
\end{equation}
on this event, where the last inequality uses \eqref{eq:K0-bound}.
The resolvent identity then gives
\begin{equation}\label{eq:L-inverse-comparison}
        \sup_{z\in\Gamma_{k,\mathsf s}}
        \|L(z)^{-1}-K_0(z)\|_2
        =
        \bigop\!\left(
        \frac{\eta_n}{\Delta_{k,\mathsf s}^2}
        \right).
\end{equation}
The Woodbury identity and the analyticity of $\widetilde R_{\mathcal X}$ reduce the integral \eqref{eq:outlier-projection-integral} to
\begin{equation}\label{eq:evec-proj-woodbury}
        a_n^*\widetilde{\mathcal P}_{k,\mathsf s} b_n
        =
        \frac{1}{2\pi\ii}
        \oint_{\Gamma_{k,\mathsf s}}
        a_n^*\widetilde R_{\mathcal X}(z)\mathcal U
        L(z)^{-1}
        \mathcal U^*\widetilde R_{\mathcal X}(z)b_n
        \,\dif z.
\end{equation}
Lemma~\ref{lem:resolvent-est}, Equation~\eqref{eq:L-inverse-comparison}, and $|\Gamma_{k,\mathsf s}|=\pi\Delta_{k,\mathsf s}/2$ give 
\[
        a_n^*\widetilde{\mathcal P}_{k,\mathsf s} b_n
        =
        \frac{1}{2\pi\ii}
        \oint_{\Gamma_{k,\mathsf s}}
        \frac1{z^2}
        a_n^*\mathcal U K_0(z)\mathcal U^*b_n
        \,\dif z
        +
        \bigop\!\left(
        \frac{\eta_n}{\Delta_{k,\mathsf s}}
        \right).
\]
Since $0\notin\overline{D_{k,\mathsf s}}$ and $\alpha_\nu\in D_{k,\mathsf s}$ exactly when $\nu=(k,\mathsf s)$, we have
\[
        \frac{1}{2\pi\ii}
        \oint_{\Gamma_{k,\mathsf s}}
        \frac{\alpha_\nu}{z(z-\alpha_\nu)}
        \,\dif z
        =
        \indicator\{\nu=(k,\mathsf s)\}.
\]
Summing these residues gives $a_n^*\mathcal P_{k,\mathsf s}b_n$, which proves \eqref{eq:evec-proj-main-bilinear}.

\subsection{Proof of Theorem~\ref{thm:entrywise-overlap}}

Fix $\varepsilon_0\in(0,\tau)$ and, for $a\in\{1,2\}$, $i\in\intset{r_a}$, and $\mathsf s\in\{+,-\}$, set
\[
        \alpha_{a,i,\mathsf s}
        \deq
        \mathsf s\frac{d_{ai}}2,
        \qquad
        \mathfrak u_{a,i,\mathsf s}
        \deq
        \frac1{\sqrt2}
        \begin{pmatrix}
        u_{a,i}\\
        \mathsf s u_{a,i}
        \end{pmatrix}.
\]
Define $\Delta_{a,i,\mathsf s}$ analogously to \eqref{eq:Delta-def} with $\varepsilon=\varepsilon_0$, reference radius $b_{\star,a,n}$, and all $2r$ signed spike locations $\{\alpha_{a,k,\mathsf t}:k\in\intset{r},\ \mathsf t\in\{+,-\}\}$, where $\alpha_{a,k,\mathsf t}=\mathsf t d_{ak}/2$ for every $k\in\intset{r}$.
By Assumptions~\ref{asm:wigner-assumptions} and~\ref{asm:overlap-entry-assumptions}, we have
\begin{equation}
\label{eq:samplewise-outlier-separation}
        \liminf_{n\to\infty}
        \min_{a,i,\mathsf s}\Delta_{a,i,\mathsf s}>0,
        \qquad
        \max_{a,i,\mathsf s}
        \frac{\eta_{a,n}}{\Delta_{a,i,\mathsf s}}
        \to0.
\end{equation}

For $a\in\{1,2\}$, let $\mathcal G_{a,n}$ be the event that, for every
$k\in\intset{r}$ and $\mathsf t\in\{+,-\}$, the disk centered at
$\alpha_{a,k,\mathsf t}$ with radius $\Delta_{a,k,\mathsf t}/4$ contains
exactly one eigenvalue of $\mathcal Y_a$, counted with algebraic multiplicity.
Applying Proposition~\ref{prop:outlier-location} to sample $a$ and using
\eqref{eq:samplewise-outlier-separation}, we obtain
\[
        \frac{
        \max_{k,\mathsf t}
        |\widehat\lambda_{a,k,\mathsf t}
        -\alpha_{a,k,\mathsf t}|
        }{
        \min_{k,\mathsf t}\Delta_{a,k,\mathsf t}
        }
        =
        \bigop\!\left(
        \max_{k,\mathsf t}
        \frac{\eta_{a,n}}{\Delta_{a,k,\mathsf t}}
        \right)
        =
        \smallop(1).
\]
Thus, with probability tending to one, every selected eigenvalue lies in its corresponding disk.  
These $2r_a$ disks are pairwise disjoint and contained in $\{z:|z|>b_{\star,a,n}+\varepsilon_0\}$.  
The exact exterior count in Proposition~\ref{prop:outlier-location}, applied with $\kappa=\varepsilon_0$, then implies that each disk contains exactly one eigenvalue, counted with algebraic multiplicity. 
Consequently,
\begin{equation}
\label{eq:samplewise-good-event}
        \Pr(\mathcal G_{a,n}^{\mathsf c})=\smallo(1).
\end{equation}

On $\mathcal G_{a,n}$, every selected eigenvalue
$\widehat\lambda_{a,i,\mathsf s}$ is algebraically simple.  It is also real.
Indeed, $\mathcal Y_a$ is real, and the disk centered at the real number
$\alpha_{a,i,\mathsf s}$ is invariant under complex conjugation.  If
$\widehat\lambda_{a,i,\mathsf s}$ were nonreal, its distinct complex conjugate
would lie in the same disk, contradicting the definition of
$\mathcal G_{a,n}$.

We next verify that the eigenvector normalization in the estimator is
well defined.  For any real, algebraically simple eigenvalue, let
$r_0$ and $\ell_0$ be nonzero real right and left eigenvectors, respectively,
and set
\[
        A
        \deq
        \mathcal Y_a-\widehat\lambda_{a,i,\mathsf s}I_{2n}.
\]
Then
\[
        Ar_0=0,
        \qquad
        \ker(A^\top)=\operatorname{span}\{\ell_0\}.
\]
If $\ell_0^\top r_0=0$, then $r_0$ is orthogonal to $\ker(A^\top)$ and hence
belongs to the column space of $A$.  Thus $Av=r_0$ for some $v$, while
$Ar_0=0$.  This gives a Jordan chain of length two and contradicts algebraic
simplicity.  Therefore $\ell_0^\top r_0\neq0$.  After normalizing $r_0$ and
rescaling $\ell_0$, we may impose
\[
        \|r_0\|_2=1,
        \qquad
        \ell_0^\top r_0=1.
\]
It follows that the eigenvector convention in the definition of
$\widetilde{\mathcal P}_{a,i,\mathsf s}$ is well defined on
$\mathcal G_{a,n}$.  On this event,
$\widetilde{\mathcal P}_{a,i,\mathsf s}$ is the rank one Riesz projection
appearing in Proposition~\ref{prop:split-evec-projection}.

\begin{lemma}
\label{lem:samplewise-projector-approximation}
For $a\in\{1,2\}$, $i\in\intset{r}$, and $\mathsf s\in\{+,-\}$, define
\[
        \mathcal P_{a,i,\mathsf s}
        \deq
        \mathfrak u_{a,i,\mathsf s}
        \mathfrak u_{a,i,\mathsf s}^\top,
        \qquad
        \Pi_{a,i}
        \deq
        u_{a,i}u_{a,i}^\top,
\]
and
\[
        \mathscr E^{\mathrm{out}}_{a,i,\mathsf s}
        \deq
        \widetilde{\mathcal P}_{a,i,\mathsf s}
        -
        \mathcal P_{a,i,\mathsf s},
        \qquad
        \mathscr E_{a,i}
        \deq
        \widehat\Pi_{a,i}-\Pi_{a,i}.
\]
Under the assumptions of Theorem~\ref{thm:entrywise-overlap},
\begin{equation}
\label{eq:samplewise-direction-rank-norm}
        \operatorname{rank}
        (\mathscr E^{\mathrm{out}}_{a,i,\mathsf s})
        \leq
        2,
        \qquad
        \|\mathscr E^{\mathrm{out}}_{a,i,\mathsf s}\|_2
        =
        \bigop(1),
        \qquad
        \operatorname{rank}(\mathscr E_{a,i})
        \leq
        3,
        \qquad
        \|\mathscr E_{a,i}\|_2
        =
        \bigop(1).
\end{equation}
For every fixed $\eta>0$,
\begin{equation}
\label{eq:samplewise-direction-isotropic}
        \sup_{\substack{x,y\in\mathbb C^{2n}\ \mathrm{deterministic}\\
        \|x\|_2\vee\|y\|_2\leq1}}
        \Pr\!\left(
        |x^*\mathscr E^{\mathrm{out}}_{a,i,\mathsf s}y|
        >
        \eta
        \right)
        =
        \smallo(1),
\end{equation}
and the analogous assertion holds for $\mathscr E_{a,i}$ and deterministic vectors in $\mathbb C^n$.
If $x_n,y_n$ are independent of $(X_a,P_a)$ and satisfy
$\|x_n\|_2\vee\|y_n\|_2=\bigop(1)$, then
\begin{equation}
\label{eq:samplewise-direction-isotropic-random}
        x_n^*\mathscr E_{a,i}y_n=\smallop(1),
        \qquad
        x_n^*\mathscr E^{\mathrm{out}}_{a,i,\mathsf s}y_n=\smallop(1).
\end{equation}
\end{lemma}

Proposition~\ref{prop:split-evec-projection} and Lemma~\ref{lem:samplewise-projector-approximation} therefore apply to each sample.
Write $\widehat\Pi_{a,i}=\Pi_{a,i}+\mathscr E_{a,i}$.
For fixed $i,j$,
\[
        \Tr(\widehat\Pi_{1,i}\widehat\Pi_{2,j})
        -
        \Tr(\Pi_{1,i}\Pi_{2,j})
        =
        \Tr(\mathscr E_{1,i}\Pi_{2,j})
        +
        \Tr(\Pi_{1,i}\mathscr E_{2,j})
        +
        \Tr(\mathscr E_{1,i}\mathscr E_{2,j}).
\]
The first two terms are $\smallop(1)$ by the deterministic-vector part of Lemma~\ref{lem:samplewise-projector-approximation}.
For the third term, condition on sample 1 and take a singular value decomposition
\[
        \mathscr E_{1,i}
        =
        \sum_{k=1}^{3}\sigma_k a_kb_k^*,
\]
where zero singular values are added if necessary.  Then
\[
        \Tr(\mathscr E_{1,i}\mathscr E_{2,j})
        =
        \sum_{k=1}^{3}\sigma_k b_k^*\mathscr E_{2,j}a_k.
\]
The vectors $a_k,b_k$ are independent of $X_2,P_2$.
Lemma~\ref{lem:samplewise-projector-approximation} shows that every bilinear
form in the sum is $\smallop(1)$.  The same lemma gives
$\max_k\sigma_k=\bigop(1)$, and hence
\[
        \Tr(\mathscr E_{1,i}\mathscr E_{2,j})=\smallop(1).
\]
Finally,
\[
        \Tr(\Pi_{1,i}\Pi_{2,j})
        =
        (u_{1,i}^\top u_{2,j})^2
        =R_{ij}^2.
\]
Taking real parts gives
\begin{equation}\label{eq:squared-overlap-consistency}
        \max_{i\in\intset{r_1}, j \in \intset{r_2}}
        \big|\widehat q_{ij}-R_{ij}^2\big|
        =
        \smallop(1).
\end{equation}
Let
\[
        \mathsf c(x)\deq\min\{1,\max\{0,x\}\}.
\]
Since $R_{ij}^2\in[0,1]$, projection onto $[0,1]$ and the inequality
$|\sqrt{x}-\sqrt{y}|\leq\sqrt{|x-y|}$ for $x,y\geq0$ give
\begin{align*}
        \big|\widehat r^{\,\mathrm{abs}}_{ij}-|R_{ij}|\big|
        \leq
        \sqrt{\big|\mathsf c(\widehat q_{ij})-R_{ij}^2\big|}
        \leq
        \sqrt{\big|\widehat q_{ij}-R_{ij}^2\big|}.
\end{align*}
This, together with \eqref{eq:squared-overlap-consistency}, gives
\begin{equation}
\label{eq:absolute-overlap-consistency}
        \max_{i\in\intset{r_1}, j \in \intset{r_2}}
        \left|
        \widehat r_{ij}^{\,\mathrm{abs}}
        -
        |R_{ij}|
        \right|
        =
        \smallop(1).
\end{equation}

Set
\[
        h_{a,i}
        \deq
        \mathfrak u_{a,i,+}
        =
        \frac1{\sqrt2}
        \begin{pmatrix}
        u_{a,i}\\
        u_{a,i}
        \end{pmatrix},
        \qquad
        J_1
        \deq
        \begin{pmatrix}
        I_n&0\\
        0&0
        \end{pmatrix},
        \qquad
        c_{a,i}
        \deq
        h_{a,i}^\top\widetilde r_{a,i,+}.
\]
Lemma~\ref{lem:samplewise-projector-approximation} gives
\[
        c_{a,i}
        \widetilde\ell_{a,i,+}^\top h_{a,i}
        =
        1+\smallop(1),
        \qquad
        \|\widetilde\ell_{a,i,+}\|_2
        =
        \|\widetilde{\mathcal P}_{a,i,+}\|_2
        =
        \bigop(1).
\]
The  bounds above and the fact that $r$ is fixed imply that the event
\[
        \mathcal C_n
        \deq
        \Big\{
        \min_{a,i}|c_{a,i}|>0
        \Big\}
\]
has probability tending to one, and $c_{a,i}^{-1}=\bigop(1)$ on $\mathcal C_n$.
On this event, for deterministic or independent $\bigop(1)$ test vectors $v$, by Lemma \ref{lem:samplewise-projector-approximation},
\[
        \smallop(1) = h_{a,i}^{\top} \big(\widetilde{\mathcal P}_{a,i,+} - h_{a,i}h_{a,i}^{\top} \big) v 
        = c_{a,i}\widetilde\ell_{a,i,+}^\top v - h_{a,i}^\top v,
\]
which gives
\[
        \widetilde\ell_{a,i,+}^\top v
        =
        c_{a,i}^{-1}h_{a,i}^\top v
        +
        \smallop(1).
\] 
Similarly, we have
\[
        v^\top\widetilde r_{a,i,+}
        =
        c_{a,i}v^\top h_{a,i}
        +
        \smallop(1).
\]
Applying these relations successively across the two independent samples yields
\[
        y_{1,i}^\top x_{2,j}
        =
        \widetilde\ell_{1,i,+}^\top
        J_1
        \widetilde r_{2,j,+}
        =
        \frac12
        c_{1,i}^{-1}c_{2,j}R_{ij}
        +
        \smallop(1).
\]
If $|R_{ij}|$ is bounded away from zero, we get
\begin{equation}
\label{eq:eqv-estimator-sign-motivation}
        \operatorname{sgn}(y_{1,i}^\top x_{2,j})
        =
        \operatorname{sgn}(c_{1,i})\operatorname{sgn}(c_{2,j})\operatorname{sgn}(R_{ij})
\end{equation}
with probability tending to one. 
Equation~\eqref{eq:absolute-overlap-consistency} provides the estimate
\[
        \widehat r^{\,\mathrm{abs}}_{ij}
        =|R_{ij}|+\smallop(1).
\]
Combining this magnitude with the sign in
\eqref{eq:eqv-estimator-sign-motivation} gives, for entries bounded away from
zero,
\begin{equation*}
        \widehat R^{\,\mathrm{eqv}}_{n,ij}
        \deq
        \operatorname{sgn}(y_{1,i}^\top x_{2,j})
        \widehat r^{\,\mathrm{abs}}_{ij}
        =
        \operatorname{sgn}(c_{1,i})\operatorname{sgn}(c_{2,j})R_{ij}+\smallop(1).
\end{equation*} 
If $R_{ij}$ tends to zero, no stable sign estimate is needed because $\widehat r^{\,\mathrm{abs}}_{ij}$ also tends to zero.
By convention, we set $\operatorname{sgn}(0)=1$.

Define $\Xi_{a,n}=\diag(\operatorname{sgn}(c_{a,1}),\ldots,\operatorname{sgn}(c_{a,r_a}))$.
For any fixed $\delta>0$, \eqref{eq:eqv-estimator-sign-motivation} holds simultaneously for all entries satisfying $|R_{ij}|\geq\delta$ with probability tending to one.  
For the remaining entries, \eqref{eq:absolute-overlap-consistency} bounds the error caused by an arbitrary sign by $2\delta+\smallop(1)$.  
Letting $\delta\downarrow0$ and using that $r$ is fixed yield
\[
        \left\|
        \widehat R_n^{\,\mathrm{eqv}}-\Xi_{1,n}R\Xi_{2,n}
        \right\|_{\max}
        =\smallop(1),
\]
which implies the first assertion of Theorem~\ref{thm:entrywise-overlap}.

It remains to prove the asserted consistency of the signal correlation estimator.  
By \eqref{eq:overlap-signed-estimator} and \eqref{eq:absolute-overlap-consistency},
\begin{align}
\label{eq:proof-squared-overlap-consistency}
        \max_{i\in\intset{r_1}, j \in \intset{r_2}}
        \left|
        (\widehat R^{\,\mathrm{eqv}}_{n,ij})^2-R_{ij}^2
        \right|
        &\leq
        2\max_{i\in\intset{r_1}, j \in \intset{r_2}}
        \left|
        \widehat r^{\,\mathrm{abs}}_{ij}-|R_{ij}|
        \right|
        =\smallop(1),
\end{align}
where we used that both $\widehat r^{\,\mathrm{abs}}_{ij}$ and $|R_{ij}|$ belong to $[0,1]$.  Applying Theorem~\ref{thm:spike-estimation-consistency} to $Y_1$ and $Y_2$ gives
\begin{equation}
\label{eq:proof-two-sample-spike-consistency}
        \max_{a\in\{1,2\}}\max_{i\in\intset{r}}
        |\widehat d_{ai}-d_{ai}|=\smallop(1).
\end{equation}
Set
\[
        \widehat N_n
        \deq
        \sum_{i\in\intset{r_1}}\sum_{j\in\intset{r_2}}
        \widehat d_{1i}\widehat d_{2j}
        (\widehat R^{\,\mathrm{eqv}}_{n,ij})^2,
        \qquad
        N_n
        \deq
        \sum_{i\in\intset{r_1}}\sum_{j\in\intset{r_2}} d_{1i}d_{2j}R_{ij}^2,
\]
and
\[
        \widehat A_{a,n}
        \deq
        \left(\sum_{i\in\intset{r_a}}\widehat d_{ai}^2\right)^{1/2},
        \qquad
        A_a
        \deq
        \left(\sum_{i\in\intset{r_a}} d_{ai}^2\right)^{1/2}.
\]
The estimates \eqref{eq:proof-squared-overlap-consistency} and \eqref{eq:proof-two-sample-spike-consistency} imply
\begin{equation*}
        |\widehat N_n-N_n|=\smallop(1),
        \qquad
        |\widehat A_{a,n}-A_a|
        \leq
        \left(\sum_{i\in\intset{r_a}}(\widehat d_{ai}-d_{ai})^2\right)^{1/2}
        =\smallop(1).
\end{equation*}
In particular, $A_1A_2>0$ and
$\widehat A_{1,n}\widehat A_{2,n}$ is bounded away from zero with
probability tending to one.  Moreover, $|N_n|=\bigo(1)$.  Therefore,
using \eqref{eq:wigner-overlap-signal-correlation-estimator} and the
definition of $\rho_{\mathrm{sig}}$,
\begin{align*}
        |\widehat\rho_{\mathrm{sig}}-\rho_{\mathrm{sig}}|
        &\leq
        \frac{|\widehat N_n-N_n|}
        {\widehat A_{1,n}\widehat A_{2,n}}
        +|N_n|
        \left|
        \frac{1}{\widehat A_{1,n}\widehat A_{2,n}}
        -\frac{1}{A_1A_2}
        \right|
        =\smallop(1),
\end{align*}
which proves the remaining assertion of Theorem~\ref{thm:entrywise-overlap}.

\section*{Acknowledgement}  Z. Bao would like to thank Yizhe Zhu for helpful discussions and references. Z. Bao, Y. Li and J. Qiu are partially supported by Hong Kong RGC Grant GRF  17304225 and 17303826. K. Cheong is partially supported by Hong Kong PhD Fellowship Scheme (HKPFS).

\section*{AI usage declaration} The project started in January 2026 and was largely completed in June 2026. All statistical methodology and mathematical derivations are the authors’ own. Nevertheless, since June 2026, we have used ChatGPT 5.6 to generate more simulation code, polish the writing, and check the proofs.

\newpage

\appendix

\vspace*{2em}
\begin{center}
\LARGE \bfseries Supplementary Material for ``Spike Estimation from Heteroscedastic Noise via Random Splitting''
\end{center}
\vspace{1.5em}

% \providecolor{lightblue}{RGB}{173,216,230}

\section{Additional Simulation Results}

This section provides additional simulation results for
Propositions~\ref{prop:outlier-location} and~\ref{prop:rect-phase}, and
Algorithms~\ref{alg:rect-spike} and~\ref{alg:rect-correlation} in the main
text.

\subsection{Simulations for Propositions~\ref{prop:outlier-location}
and~\ref{prop:rect-phase}}

We compare the Wigner and rectangular models through a common normalized-spike
design.  Every panel uses $n=1000$, standard Gaussian noise, an independent
$\mathrm{Bernoulli}(1/2)$ splitting mask, and one realization; hence no
Monte Carlo averaging is used.  In the rectangular model, we set
$p=3n/5=600$.

For $J_{a,b}\deq\one_a\one_b^\top$ and $\kappa\in\{1,4\}$, define
\[
        T_{a,n}^{(\kappa)}
        =
        \frac1n
        \begin{pmatrix}
        \kappa J_{a/2,n/2}&J_{a/2,n/2}\\
        J_{a/2,n/2}&J_{a/2,n/2}
        \end{pmatrix}.
\]
Thus $T^{(\kappa)}\deq T_{n,n}^{(\kappa)}$ is the Wigner profile, whereas
$T_{\mathrm{rec}}^{(\kappa)}\deq T_{p,n}^{(\kappa)}$ is its rectangular
analogue.  The choice $\kappa=1$ is homogeneous, while $\kappa=4$ increases
the variance in the upper left block.
For a selected profile $T=(t_{ij})$, the noise entries have the form
$X_{ij}=\sqrt{t_{ij}}\,\xi_{ij}$ with standard Gaussian $\xi_{ij}$.  In the
Wigner model, the upper triangular entries of $X$ and of the mask are generated
independently and then symmetrized.  In the rectangular model, all entries of
$X$ and of the $p\times n$ mask are generated independently.

For each even $a\in\{p,n\}$, put $m_a=a/2$ and, for
$i\in\intset{m_a}$, define
\[
        h_1^{(a)}(i)=\frac1{\sqrt{m_a}},
        \qquad
        h_{k+1}^{(a)}(i)
        =
        \sqrt{\frac2{m_a}}
        \cos\!\left\{\frac{\pi k(i-1/2)}{m_a}\right\},
        \qquad k\in\{1,2\}.
\]
The block-oriented directions are
\[
        g_k^{(a)}
        =
        \left(
        \sqrt{0.9}\,(h_k^{(a)})^\top,
        \sqrt{0.1}\,(h_k^{(a)})^\top
        \right)^\top,
        \qquad k\in\{1,2,3\}.
\]
For each $a$, these vectors are orthonormal, and every vector has squared mass
$0.9$ in its first half and $0.1$ in its second half.  In the Wigner model,
we take $U=(g_1^{(n)},\ldots,g_r^{(n)})$ and $S=UDU^\top$.  In the
rectangular model, we take
$U=(g_1^{(p)},\ldots,g_r^{(p)})$,
$V=(g_1^{(n)},\ldots,g_r^{(n)})$, and $S=UDV^\top$.

For both $\kappa=1$ and $\kappa=4$, the rank one experiment uses
$\gamma_1\in\{0.80,1.20,1.60\}$ and is shown for
the Wigner and rectangular models in Figures~\ref{fig:supp-phase-transition}
and~\ref{fig:supp-phase-transition-rectangular}, respectively.  The rank-three
experiment appears in Figures~\ref{fig:supp-distinct-repeated-spikes} and
\ref{fig:supp-distinct-repeated-spikes-rectangular}.  In both models, the
distinct-spike configuration is
$\boldsymbol\gamma=(2.00,1.60,1.30)$, while the repeated-spike configuration is
$\boldsymbol\gamma=(2.00,1.50,1.50)$, with $d_2=d_3$.

\begin{figure}[htbp]
\centering

% First subfigure setup
\begin{subfigure}[b]{\textwidth}
\centering
\includegraphics[width=\textwidth]{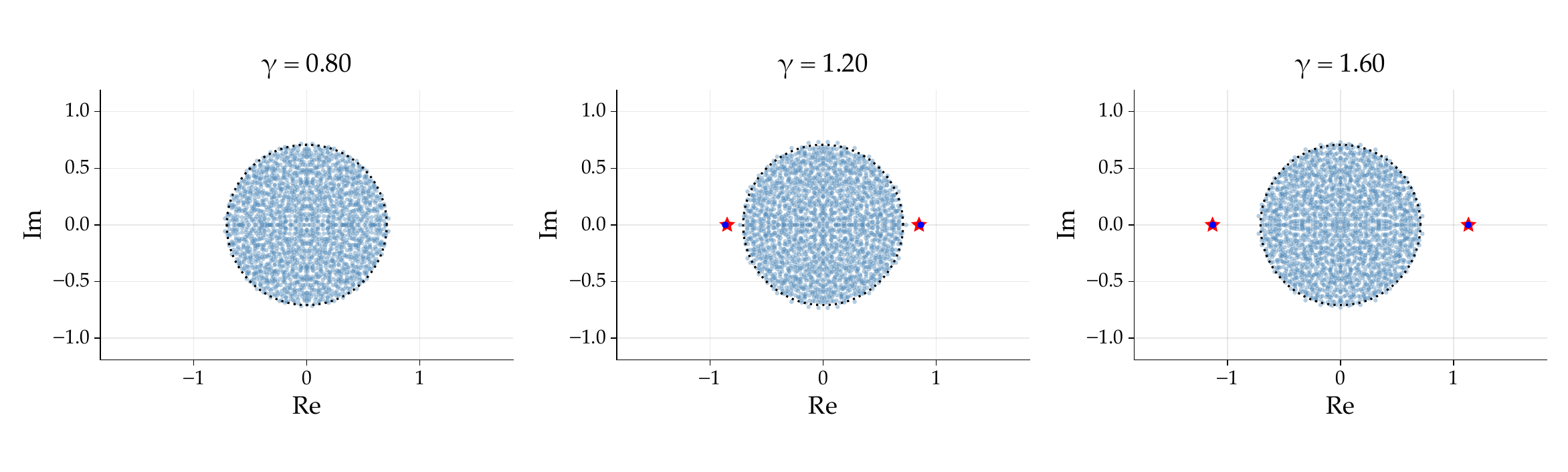}
\caption{Homogeneous variance profile}
\end{subfigure}
\vfill
% Second subfigure setup
\begin{subfigure}[b]{\textwidth}
\centering
\includegraphics[width=\textwidth]{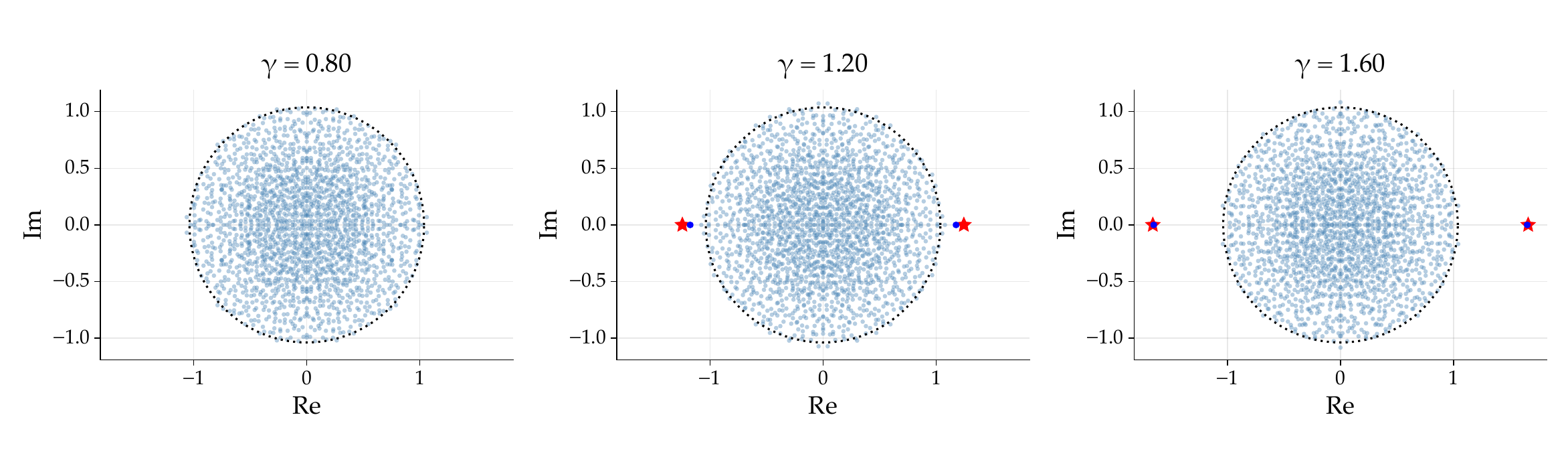}
\caption{Heterogeneous variance profile}
\end{subfigure}

\caption{Single-spike phase transition under the homogeneous variance
        profile $T^{(1)}$ (top row) and the heterogeneous variance profile
        $T^{(4)}$ (bottom row).  From left to right,
        $\gamma_1\in\{0.80,1.20,1.60\}$.  The light-blue points ({\color{blue!40!white}\textbullet}) show the
        spectrum of $\mathcal Y$, the dotted black circle marks $|z|=b_{\star,n}$, the
        red stars ({\color{red}$\bigstar$}) mark the theoretical locations $\pm d_1/2$ in the
        supercritical panels, and the blue points ({\color{blue}\textbullet}) mark the matched outliers.}
        \label{fig:supp-phase-transition}
\end{figure}

\begin{figure}[htbp]
\centering

% First subfigure setup
\begin{subfigure}[b]{0.49\textwidth}
\centering
\includegraphics[width=\textwidth]{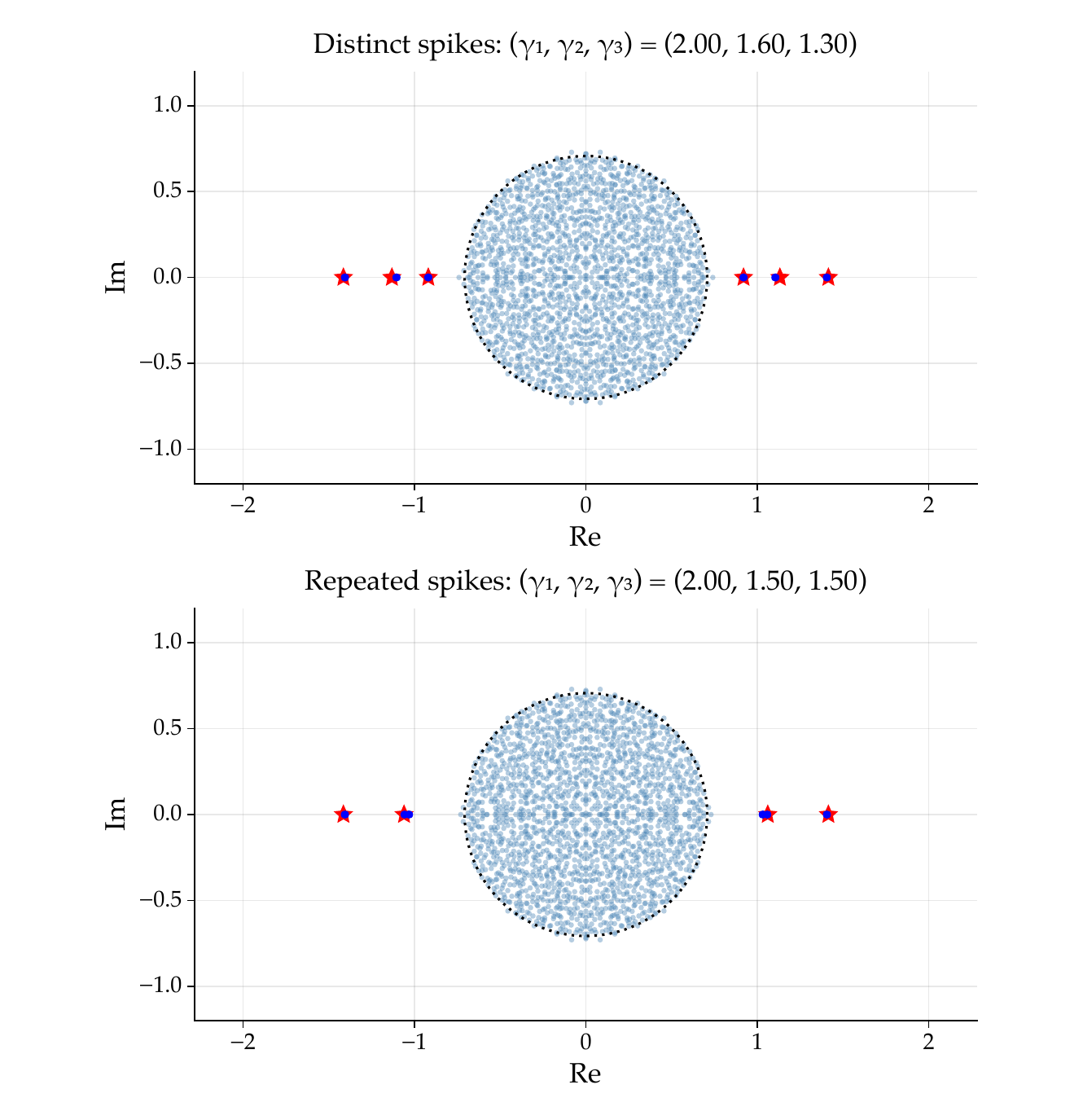}
\caption{Homogeneous variance profile}
\end{subfigure}
\hfill
% Second subfigure setup
\begin{subfigure}[b]{0.49\textwidth}
\centering
\includegraphics[width=\textwidth]{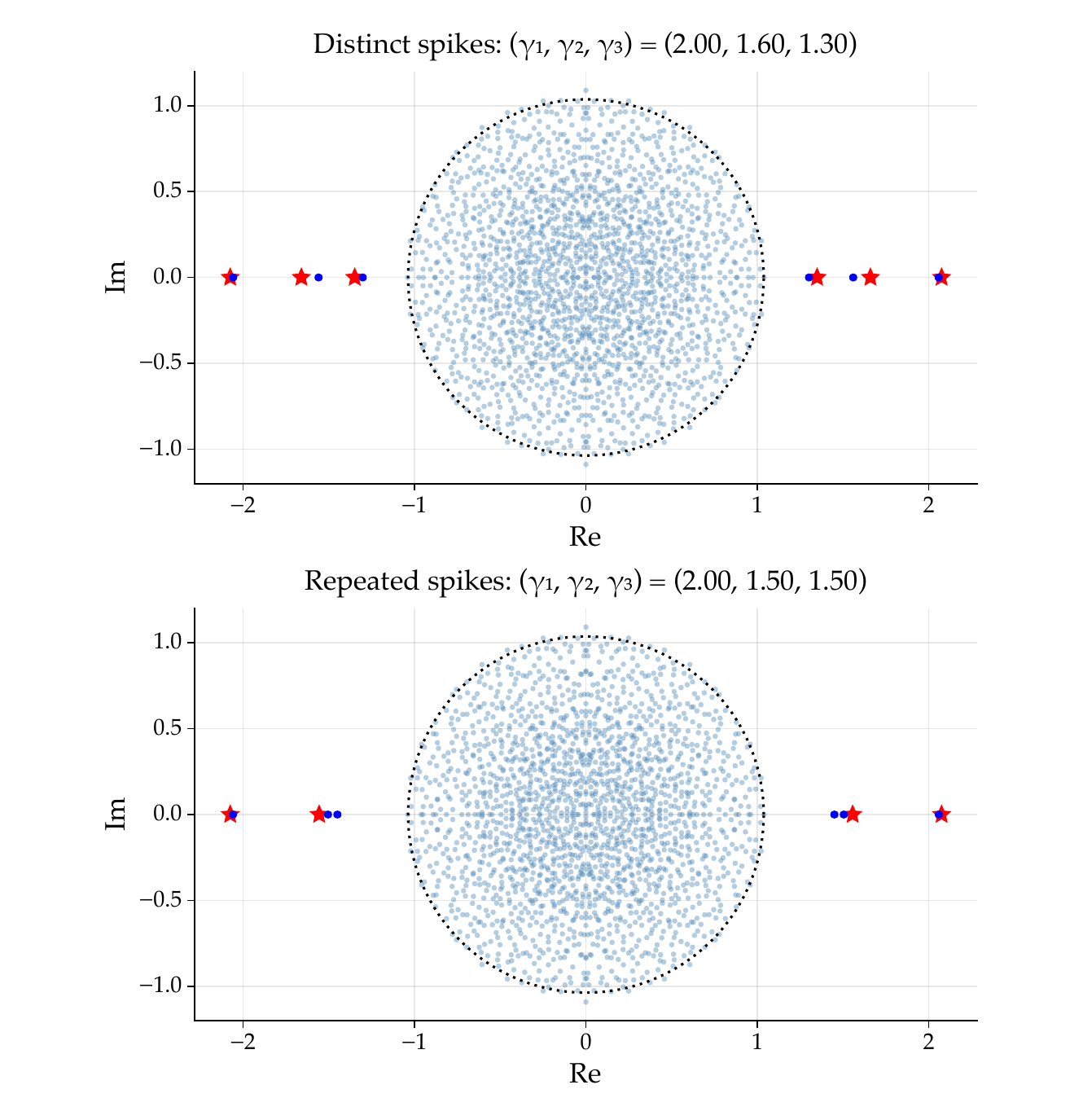}
\caption{Heterogeneous variance profile}
\end{subfigure}
\caption{Outlier locations for distinct and repeated rank-three spikes.
        The top row uses $\boldsymbol\gamma=(2.00,1.60,1.30)$, and the bottom
        row uses $\boldsymbol\gamma=(2.00,1.50,1.50)$.  The left and right
        columns correspond to $T^{(1)}$ and $T^{(4)}$, respectively.
        The plotting symbols have the same meanings as in
        Figure~\ref{fig:supp-phase-transition}.}
\label{fig:supp-distinct-repeated-spikes}
\end{figure}

\begin{figure}[htbp]
\centering

\begin{subfigure}[b]{\textwidth}
\centering
\includegraphics[width=\textwidth]{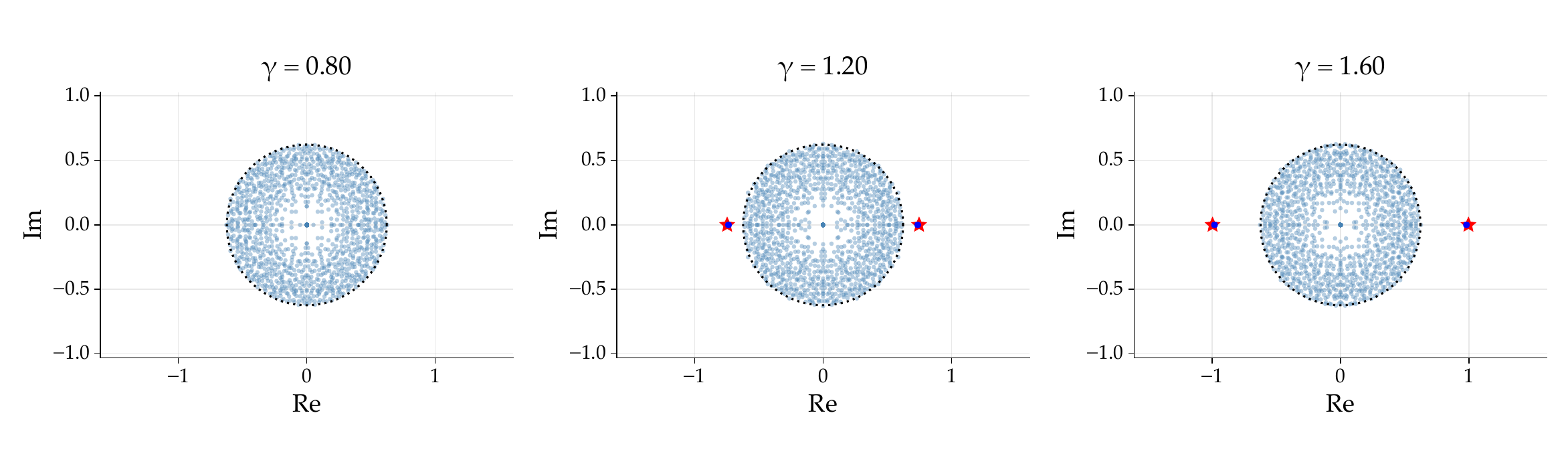}
\caption{Homogeneous rectangular variance profile}
\end{subfigure}
\vfill
\begin{subfigure}[b]{\textwidth}
\centering
\includegraphics[width=\textwidth]{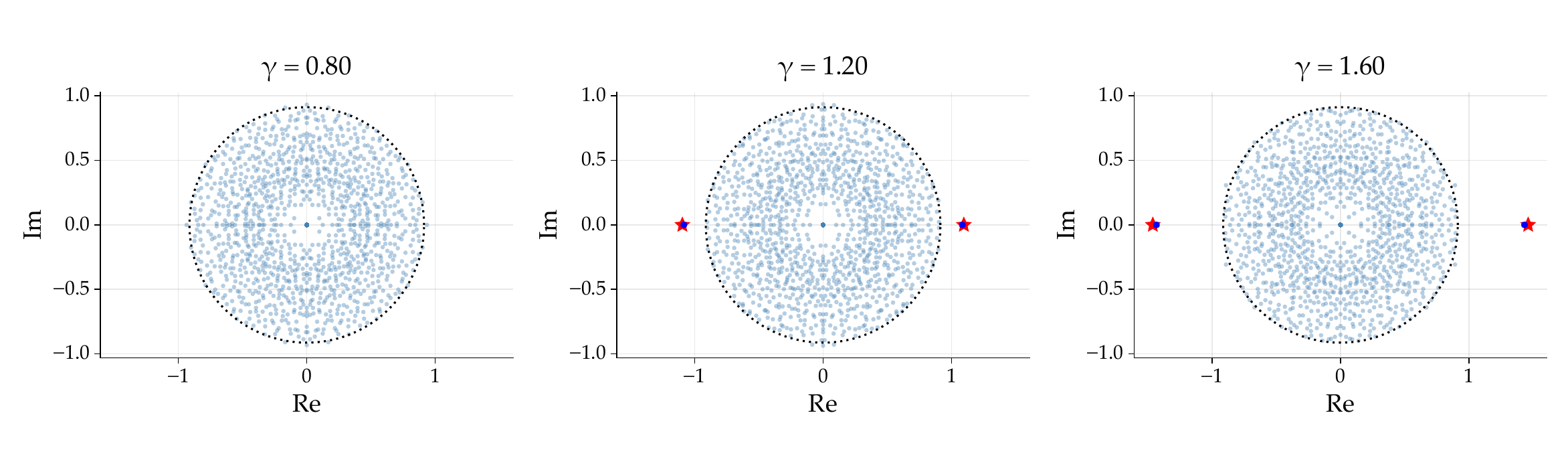}
\caption{Heterogeneous rectangular variance profile}
\end{subfigure}

\caption{Single-spike phase transition in the rectangular model with
        $n=1000$ and $p=600$, under $T_{\mathrm{rec}}^{(1)}$ (top) and
        $T_{\mathrm{rec}}^{(4)}$ (bottom).  From left to right,
        $\gamma_1\in\{0.80,1.20,1.60\}$.  The plotting symbols have the same
        meanings as in Figure~\ref{fig:supp-phase-transition}.}
\label{fig:supp-phase-transition-rectangular}
\end{figure}

\begin{figure}[htbp]
\centering

\begin{subfigure}[b]{0.49\textwidth}
\centering
\includegraphics[width=\textwidth]{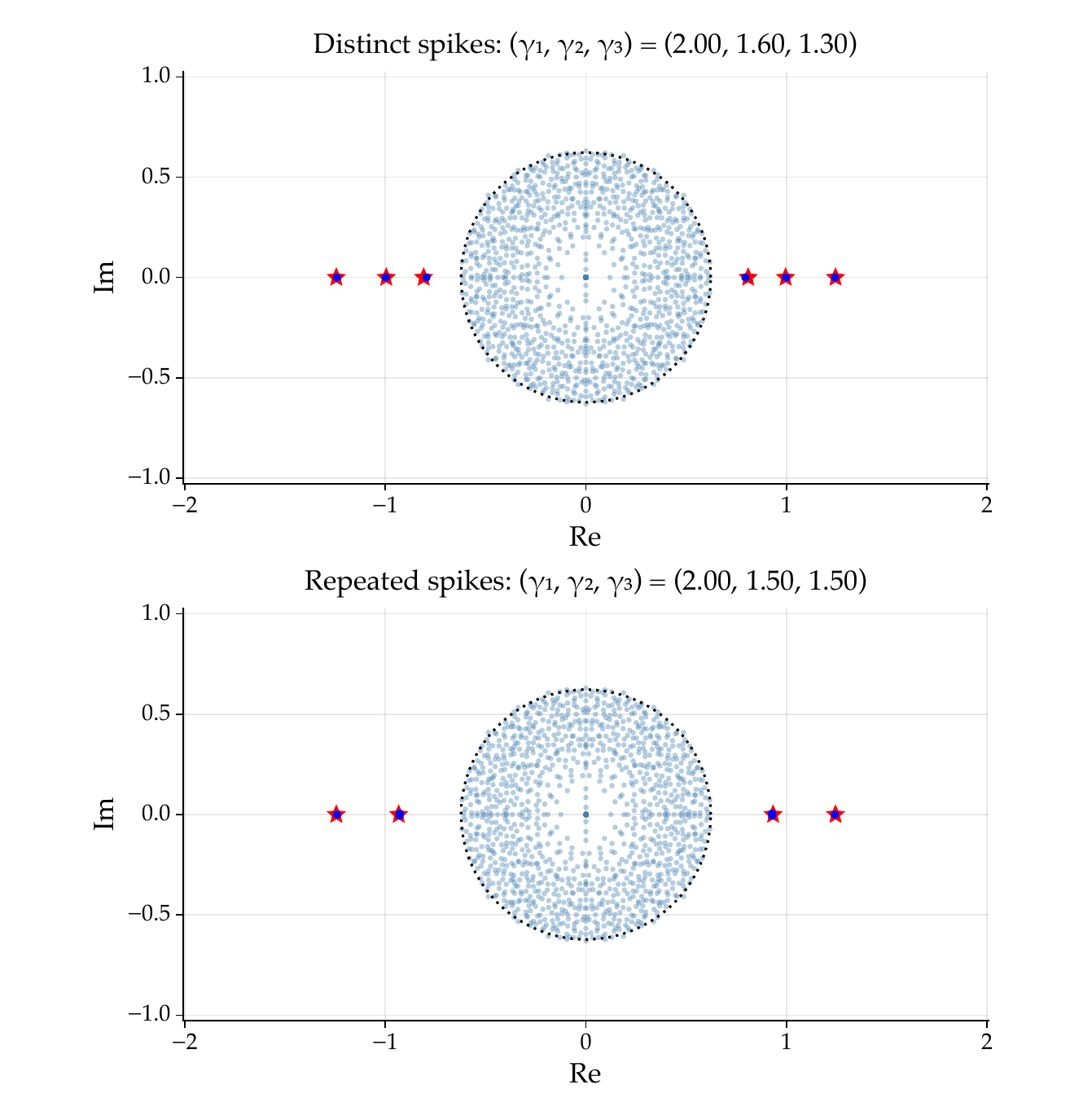}
\caption{Homogeneous rectangular variance profile}
\end{subfigure}
\hfill
\begin{subfigure}[b]{0.49\textwidth}
\centering
\includegraphics[width=\textwidth]{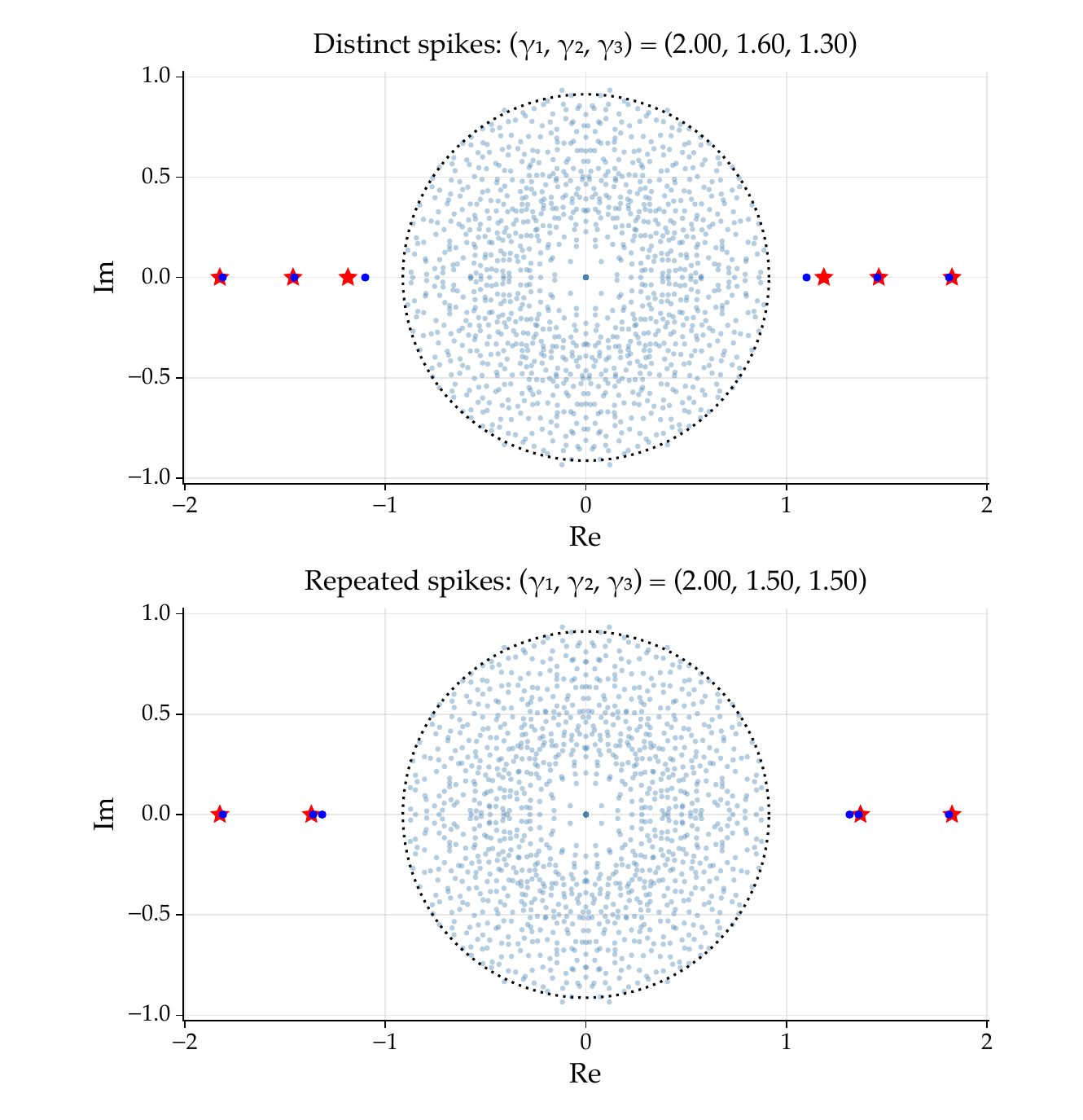}
\caption{Heterogeneous rectangular variance profile}
\end{subfigure}
\caption{Rectangular outlier locations for distinct and repeated rank-three
        spikes with $n=1000$ and $p=600$.  The top row uses
        $\boldsymbol\gamma=(2.00,1.60,1.30)$, and the bottom row uses
        $\boldsymbol\gamma=(2.00,1.50,1.50)$.  The left and right panels
        correspond to $T_{\mathrm{rec}}^{(1)}$ and
        $T_{\mathrm{rec}}^{(4)}$, respectively.  
        The plotting symbols have the same meanings as in
        Figure~\ref{fig:supp-phase-transition}. }
\label{fig:supp-distinct-repeated-spikes-rectangular}
\end{figure}

\subsection{Simulation for Algorithms~\ref{alg:rect-spike} and
\ref{alg:rect-correlation}}
\paragraph*{Common setup}
We use a rectangular analogue of the design in
Section~\ref{sec:numerical-simulations}.  For each configuration, take
$n\in\{250,500,750,1000\}$, $p=3n/5$, and $B=500$ independent data
replications.  For $J_{a,b}\deq\one_a\one_b^\top$ and
$\kappa\in\{1,4\}$, define
\[
        T_{\mathrm{rec}}^{(\kappa)}
        =\frac1n
        \begin{pmatrix}
        \kappa J_{p/2,n/2}&J_{p/2,n/2}\\
        J_{p/2,n/2}&J_{p/2,n/2}
        \end{pmatrix}.
\]
Thus $T_{\mathrm{rec}}^{(1)}$ is homogeneous, whereas
$T_{\mathrm{rec}}^{(4)}$ increases the variance in the upper-left block.
For a selected profile $T=(t_{ij})$, set
\[
        \mathbb T=\frac12
        \begin{pmatrix}0&T\\T^\top&0\end{pmatrix},
        \qquad
        b_{\star,n}=\sqrt{\|\mathbb T\|_2}
        =\sqrt{\frac{\|T\|_2}{2}},
\]
and independently generate
$X_{ij}=\sqrt{t_{ij}}\,\xi_{ij}$ for
$i\in\intset{p}$ and $j\in\intset{n}$.  As in Section~\ref{sec:numerical-simulations},
the primary and robustness experiments use independent standard Gaussian and
Rademacher variables $\xi_{ij}$, respectively.  Every auxiliary mask has
i.i.d. $\mathrm{Bernoulli}(1/2)$ entries, is independent of the observations,
and is redrawn across replications and samples.

For each even $a\in\{p,n\}$, put $m_a=a/2$ and, for
$i\in\intset{m_a}$, define
\[
        h_1^{(a)}(i)=\frac1{\sqrt{m_a}},
        \qquad
        h_{k+1}^{(a)}(i)
        =\sqrt{\frac2{m_a}}
        \cos\!\left\{\frac{\pi k(i-1/2)}{m_a}\right\},
\]
for the required values of $k$.  The block-oriented directions are
\[
        g_k^{(a)}
        =
        \left(
        \sqrt{0.9}\,(h_k^{(a)})^\top,
        \sqrt{0.1}\,(h_k^{(a)})^\top
        \right)^\top.
\]
For each $a$, the vectors $g_k^{(a)}$ are orthonormal, satisfy
Assumption~\ref{asm:rect-deloc}, and have squared mass $0.9$ in the first
half and $0.1$ in the second half.

\subsubsection*{Simulation for Algorithm~\ref{alg:rect-spike}}
Set $r=3$,
\[
        U=(g_1^{(p)},g_2^{(p)},g_3^{(p)}),
        \qquad
        V=(g_1^{(n)},g_2^{(n)},g_3^{(n)}),
\]
and consider
$\boldsymbol\gamma=(2.00,1.60,1.30)$ and
$\boldsymbol\gamma=(2.00,1.50,1.50)$.  For each variance profile, set
$d_k=2b_{\star,n}\gamma_k$, let $D=\diag(d_1,d_2,d_3)$, generate
$Y=X+UDV^\top$, and apply Algorithm~\ref{alg:rect-spike}.

Table~\ref{tab:algorithm3-results} reports the mean absolute errors and
empirical standard deviations for Algorithm~\ref{alg:rect-spike}; each entry
has the form
$\widehat{\mathsf{MAE}}_k\;(\widehat{\mathsf{SD}}_k)$.
For every fixed spike configuration, variance profile, noise law, and
component $k$, both quantities decrease monotonically over the four values of
$n$. 
The Gaussian and Rademacher results are close.

We also compare Algorithm~\ref{alg:rect-spike} with two other rectangular spike estimators.
For readers' convenience, we briefly summarize the comparison methods below. 

\begin{itemize}
        \item \citet{ShabalinNobel2013}'s \texttt{SN13} estimator: They assume homogeneous Gaussian noise and estimates the noise variance by
        \[
                \widehat\sigma_{\mathrm R}^{\,2}
                \deq \frac{\|Y\|_{\mathrm F}^2}{p}.
        \]
        Let $s_1\geq\cdots\geq s_p$ be the singular values of $Y$, and set
        \[
                A_k\deq
                s_k^2-\widehat\sigma_{\mathrm R}^{\,2}(1+\phi_n),
                \qquad
                \phi_n \deq p/n. 
        \]
        The proposed estimator is
        \[
        \widehat d_{k,\mathtt{SN13}}
        =
        \begin{cases}
        \displaystyle
        \left\{
        \frac{A_k+
        \sqrt{A_k^2-4\phi_n\widehat\sigma_{\mathrm R}^{\,4}}}{2}
        \right\}^{1/2},
        & s_k>\widehat\sigma_{\mathrm R}(1+\sqrt{\phi_n}),\\[9pt]
        0,
        & s_k\leq\widehat\sigma_{\mathrm R}(1+\sqrt{\phi_n}).
        \end{cases}
        \]
        \item \citet{Nadakuditi2014OptShrink}'s \texttt{OptShrink} estimator:
        For $z>s_{r+1}$, form the empirical left and companion transforms
        \[
        \widehat a_r(z)
        \deq\frac{1}{p-r}\sum_{j=r+1}^p\frac{z}{z^2-s_j^2},
        \qquad
        \widehat b_r(z)
        \deq\frac{1}{n-r}
        \left\{
        \sum_{j=r+1}^p\frac{z}{z^2-s_j^2}+\frac{n-p}{z}
        \right\},
        \]
        and set
        \[
                \widehat D_r(z)\deq\widehat a_r(z)\widehat b_r(z),
                \qquad
                \widehat d_{k,\mathtt{OptShrink}}
                \deq \widehat D_r(s_k)^{-1/2}.
        \]
        We assign zero if $\widehat D_r(s_k)$ is nonpositive or nonfinite. 

\end{itemize}

The \texttt{SN13} estimator is correctly specified for $T_{\mathrm{rec}}^{(1)}$.  
The \texttt{OptShrink} estimator requires bi-orthogonally invariant noise matrix, that is, $X \overset{\mathtt{d}}{=}UXV^{\top}$ for any orthogonal $U$ and $V$.

For numerical simulation, we set $r=1$ and $d_1=2b_{\star,n}\gamma$ with $\gamma=1.50$.  
In addition to $T_{\mathrm{rec}}^{(1)}$ and $T_{\mathrm{rec}}^{(4)}$, include $T_{\mathrm{rec}}^{(8)}$, defined by the same block-profile formula with $\kappa=8$, under both noise laws at $(p,n)=(600,1000)$.
Figure~\ref{fig:algorithm3-comparison} shows the violin plots of the normalized spike estimates $\widehat d_1/d_1$ to compare Algorithm~\ref{alg:rect-spike} with \texttt{SN13} and \texttt{OptShrink}. 
At $(p,n)=(600,1000)$, all estimators have similar distributions under $T_{\mathrm{rec}}^{(1)}$, whereas Algorithm~\ref{alg:rect-spike} is much closer to $1$ under $T_{\mathrm{rec}}^{(4)}$ and $T_{\mathrm{rec}}^{(8)}$ for both noise laws.

\subsubsection*{Simulation for Algorithm~\ref{alg:rect-correlation}}
Set $r=2$ and
\[
\begin{aligned}
        U_1&=(g_1^{(p)},g_2^{(p)}),
        &U_2&=U_1R_U(\omega)+(g_3^{(p)},g_4^{(p)})
        (I_2-R_U(\omega)^\top R_U(\omega))^{1/2},\\
        V_1&=(g_1^{(n)},g_2^{(n)}),
        &V_2&=V_1R_V(\omega)+(g_3^{(n)},g_4^{(n)})
        (I_2-R_V(\omega)^\top R_V(\omega))^{1/2}.
\end{aligned}
\]
To vary the overlap strength while retaining nonzero off-diagonal entries,
take
\[
        A_U=\begin{pmatrix}0.8&0.3\\-0.2&0.6\end{pmatrix},
        \qquad
        A_V=\begin{pmatrix}0.7&-0.1\\0.4&0.6\end{pmatrix},
        \qquad
        R_U(\omega)=\omega\frac{A_U}{\|A_U\|_2},
        \qquad
        R_V(\omega)=\omega\frac{A_V}{\|A_V\|_2},
\]
with $\omega\in\{0.25,0.60,0.90\}$.  Thus
$U_1^\top U_2=R_U(\omega)$, $V_1^\top V_2=R_V(\omega)$, and both overlap
matrices have operator norm $\omega<1$.

Use the variance-profile pairs
\[
        (T_1,T_2)
        \quad \in \quad \big\{
        (T_{\mathrm{rec}}^{(1)},T_{\mathrm{rec}}^{(1)}),
        (T_{\mathrm{rec}}^{(1)},T_{\mathrm{rec}}^{(4)}),
        (T_{\mathrm{rec}}^{(4)},T_{\mathrm{rec}}^{(4)})
        \big\}.
\]
For sample $a\in\{1,2\}$, let
$b_{\star,a,n}=\sqrt{\|T_a\|_2/2}$, use the distinct normalized singular
strengths
\[
        \boldsymbol\gamma_1=(2.00,1.60),
        \qquad
        \boldsymbol\gamma_2=(2.20,1.70),
        \qquad
        d_{ai}=2b_{\star,a,n}\gamma_{ai},
\]
	generate $X_1$ and $X_2$ independently, and form
	$Y_a=X_a+U_aD_aV_a^\top$ with $D_a=\diag(d_{a1},d_{a2})$.

Apply Algorithm~\ref{alg:rect-correlation} to each replication and record
\[
\begin{split}
        \mathsf E_{\mathrm{eqv}}^{(b)}
        &\deq\min_{\Xi_1,\Xi_2\in\mathscr S_2}
        \max\left\{
        \|\widehat R^{U,\mathrm{eqv},(b)}-\Xi_1R_U(\omega)\Xi_2\|_{\max},
        \|\widehat R^{V,\mathrm{eqv},(b)}-\Xi_1R_V(\omega)\Xi_2\|_{\max}
        \right\},\\
        \mathsf E_{\mathrm{sig}}^{(b)}
        &\deq\abs{\widehat\rho_{\mathrm{sig}}^{(b)}-\rho_{\mathrm{sig}}}.
\end{split}
\]
We report $\widehat{\mathsf E}_{\mathrm{eqv}}
=\frac1B\sum_{b=1}^B\mathsf E_{\mathrm{eqv}}^{(b)}$ and
$\widehat{\mathsf E}_{\mathrm{sig}}=\frac1B\sum_{b=1}^B\mathsf E_{\mathrm{sig}}^{(b)}$ in Tables~\ref{tab:algorithm4-joint-orbit-results} and \ref{tab:algorithm4-signal-correlation-results}, respectively.  For each fixed variance-profile pair, noise law, and overlap strength $\omega$, both quantities decrease monotonically over the four values of $n$.  The Gaussian and Rademacher results are close.

\begin{table}[htbp]
\centering
\footnotesize
\setlength{\tabcolsep}{2.5pt}
\begin{tabular}{@{}lccc@{\hspace{8pt}}ccc@{}}
\toprule
&\multicolumn{3}{c}{Homogeneous $T_{\mathrm{rec}}^{(1)}$}
 & \multicolumn{3}{c}{Heterogeneous $T_{\mathrm{rec}}^{(4)}$}\\
\cmidrule(lr){2-4}\cmidrule(lr){5-7}
$n$ & $k=1$ & $k=2$ & $k=3$
 & $k=1$ & $k=2$ & $k=3$\\
\midrule
& \multicolumn{6}{l}{$\boldsymbol\gamma=(2.00,1.60,1.30)$}\\
\cmidrule{2-7}
& \multicolumn{6}{l}{Gaussian}\\
250 & $0.0735\;(0.0722)$ & $0.0938\;(0.0951)$ & $0.1244\;(0.1248)$
 & $0.1276\;(0.1407)$ & $0.1610\;(0.1702)$ & $0.1962\;(0.2041)$\\
500 & $0.0487\;(0.0546)$ & $0.0593\;(0.0649)$ & $0.0722\;(0.0876)$
 & $0.0836\;(0.0990)$ & $0.1049\;(0.1154)$ & $0.1255\;(0.1530)$\\
750 & $0.0386\;(0.0453)$ & $0.0439\;(0.0481)$ & $0.0561\;(0.0667)$
 & $0.0679\;(0.0802)$ & $0.0797\;(0.0918)$ & $0.1010\;(0.1223)$\\
1000 & $0.0311\;(0.0381)$ & $0.0374\;(0.0421)$ & $0.0447\;(0.0556)$
 & $0.0548\;(0.0663)$ & $0.0676\;(0.0789)$ & $0.0871\;(0.1076)$\\
& \multicolumn{6}{l}{Rademacher}\\
250 & $0.0711\;(0.0744)$ & $0.0943\;(0.0961)$ & $0.1215\;(0.1240)$
 & $0.1295\;(0.1428)$ & $0.1589\;(0.1750)$ & $0.2143\;(0.2164)$\\
500 & $0.0479\;(0.0562)$ & $0.0557\;(0.0619)$ & $0.0722\;(0.0786)$
 & $0.0843\;(0.0987)$ & $0.1044\;(0.1193)$ & $0.1268\;(0.1484)$\\
750 & $0.0371\;(0.0425)$ & $0.0447\;(0.0493)$ & $0.0598\;(0.0693)$
 & $0.0681\;(0.0789)$ & $0.0808\;(0.0950)$ & $0.1049\;(0.1228)$\\
1000 & $0.0315\;(0.0384)$ & $0.0367\;(0.0424)$ & $0.0447\;(0.0520)$
 & $0.0584\;(0.0722)$ & $0.0674\;(0.0799)$ & $0.0824\;(0.1025)$\\
\midrule
& \multicolumn{6}{l}{$\boldsymbol\gamma=(2.00,1.50,1.50)$}\\
\cmidrule{2-7}
& \multicolumn{6}{l}{Gaussian}\\
250 & $0.0710\;(0.0745)$ & $0.0704\;(0.0851)$ & $0.1441\;(0.1018)$
 & $0.1368\;(0.1562)$ & $0.1187\;(0.1450)$ & $0.2465\;(0.1823)$\\
500 & $0.0490\;(0.0539)$ & $0.0438\;(0.0543)$ & $0.0820\;(0.0587)$
 & $0.0846\;(0.0984)$ & $0.0805\;(0.0998)$ & $0.1384\;(0.1097)$\\
750 & $0.0379\;(0.0449)$ & $0.0357\;(0.0445)$ & $0.0637\;(0.0483)$
 & $0.0693\;(0.0826)$ & $0.0661\;(0.0817)$ & $0.0981\;(0.0892)$\\
1000 & $0.0313\;(0.0371)$ & $0.0299\;(0.0374)$ & $0.0518\;(0.0393)$
 & $0.0586\;(0.0712)$ & $0.0605\;(0.0733)$ & $0.0883\;(0.0738)$\\
& \multicolumn{6}{l}{Rademacher}\\
250 & $0.0719\;(0.0739)$ & $0.0652\;(0.0793)$ & $0.1424\;(0.0982)$
 & $0.1292\;(0.1447)$ & $0.1219\;(0.1508)$ & $0.2454\;(0.1765)$\\
500 & $0.0507\;(0.0549)$ & $0.0446\;(0.0548)$ & $0.0840\;(0.0602)$
 & $0.0859\;(0.0974)$ & $0.0774\;(0.0982)$ & $0.1516\;(0.1150)$\\
750 & $0.0358\;(0.0420)$ & $0.0353\;(0.0440)$ & $0.0581\;(0.0456)$
 & $0.0667\;(0.0808)$ & $0.0655\;(0.0795)$ & $0.1049\;(0.0866)$\\
1000 & $0.0318\;(0.0362)$ & $0.0302\;(0.0379)$ & $0.0501\;(0.0403)$
 & $0.0571\;(0.0704)$ & $0.0566\;(0.0668)$ & $0.0921\;(0.0760)$\\
\bottomrule
\end{tabular}
\caption{Mean absolute errors and empirical standard deviations, shown in
parentheses, for Algorithm~\ref{alg:rect-spike} under the homogeneous and
heterogeneous rectangular variance profiles, for two normalized singular
spike configurations, based on $B=500$ replications.}
\label{tab:algorithm3-results}
\end{table}

\begin{figure}[htbp]
\centering
\includegraphics[width=\textwidth]{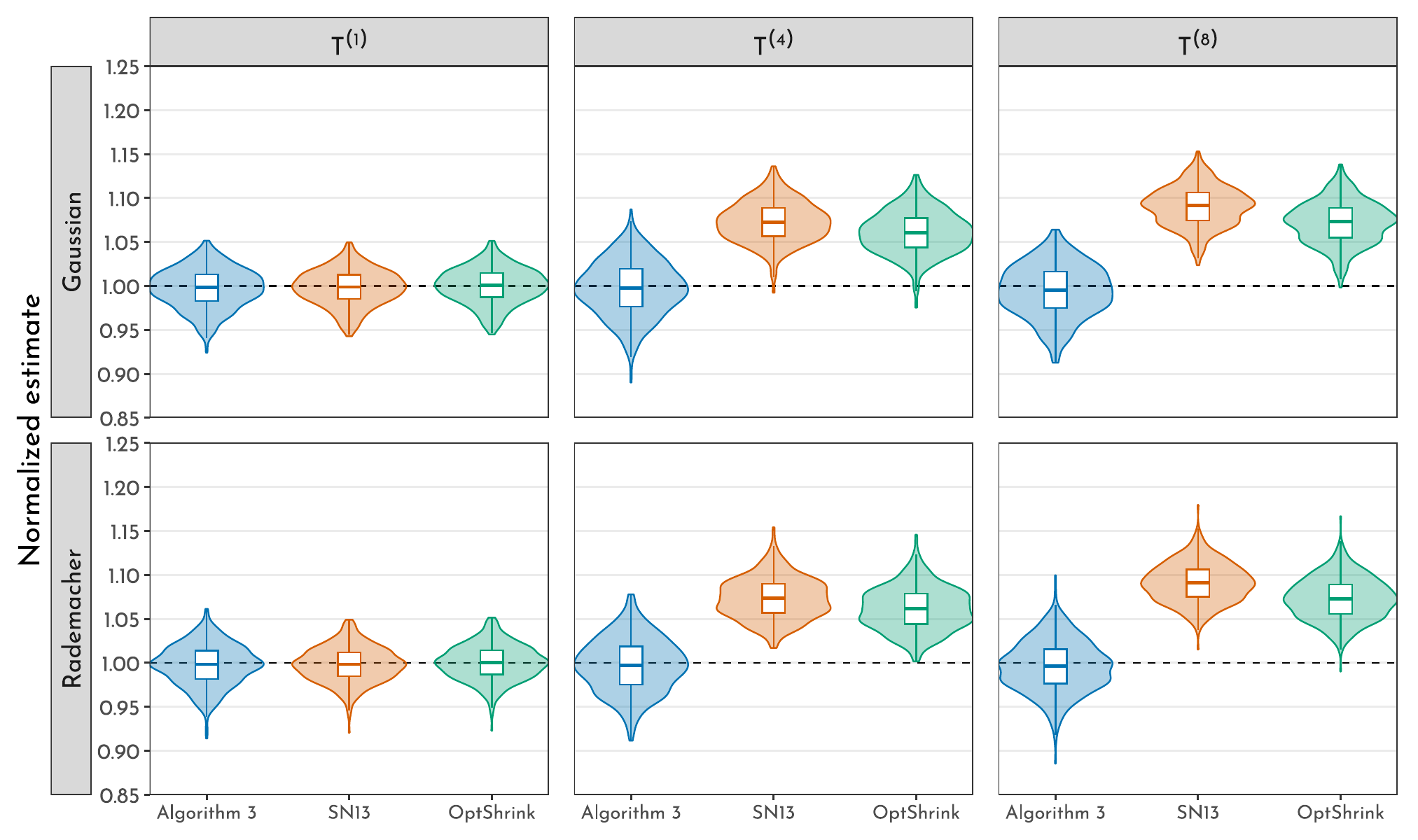}
\caption{Violin plots of the normalized spike estimates
$\widehat d_1/d_1$ at $(p,n)=(600,1000)$ for
Algorithm~\ref{alg:rect-spike}, \texttt{SN13}, and \texttt{OptShrink} under
$T_{\mathrm{rec}}^{(1)}$,
$T_{\mathrm{rec}}^{(4)}$, and $T_{\mathrm{rec}}^{(8)}$, with Gaussian
noise (top) and Rademacher noise (bottom).  Each violin is based on $B=500$
replications; the embedded boxplot shows the median and interquartile range,
and the dashed horizontal line marks the target value $1$.  The normalized
singular spike strength is $\gamma=1.50$.}
\label{fig:algorithm3-comparison}
\end{figure}

\begin{table}[htbp]
\centering
\footnotesize
\setlength{\tabcolsep}{2.5pt}
\begin{tabular}{@{}l*{9}{c}@{}}
\toprule
\multicolumn{1}{c}{}
 & \multicolumn{3}{c}{$\big(T_{\mathrm{rec}}^{(1)},T_{\mathrm{rec}}^{(1)}\big)$}
 & \multicolumn{3}{c}{$\big(T_{\mathrm{rec}}^{(1)},T_{\mathrm{rec}}^{(4)}\big)$}
 & \multicolumn{3}{c}{$\big(T_{\mathrm{rec}}^{(4)},T_{\mathrm{rec}}^{(4)}\big)$}\\
\cmidrule(lr){2-4}\cmidrule(lr){5-7}\cmidrule(lr){8-10}
$n$
 & $\omega=0.25$ & $\omega=0.60$ & $\omega=0.90$
 & $\omega=0.25$ & $\omega=0.60$ & $\omega=0.90$
 & $\omega=0.25$ & $\omega=0.60$ & $\omega=0.90$\\
\midrule
& \multicolumn{9}{l}{Gaussian}\\
250 & $0.1299$ & $0.1522$ & $0.1788$ & $0.1386$ & $0.1689$ & $0.1961$ & $0.1562$ & $0.1918$ & $0.2248$\\
500 & $0.0836$ & $0.0982$ & $0.1165$ & $0.0929$ & $0.1086$ & $0.1293$ & $0.1039$ & $0.1218$ & $0.1486$\\
750 & $0.0677$ & $0.0762$ & $0.0958$ & $0.0735$ & $0.0848$ & $0.1041$ & $0.0824$ & $0.0988$ & $0.1166$\\
1000 & $0.0536$ & $0.0663$ & $0.0786$ & $0.0625$ & $0.0729$ & $0.0891$ & $0.0719$ & $0.0840$ & $0.1026$\\
\cmidrule{2-10}
& \multicolumn{9}{l}{Rademacher}\\
250 & $0.1265$ & $0.1480$ & $0.1816$ & $0.1375$ & $0.1647$ & $0.1960$ & $0.1466$ & $0.1874$ & $0.2207$\\
500 & $0.0817$ & $0.0959$ & $0.1183$ & $0.0918$ & $0.1060$ & $0.1238$ & $0.1027$ & $0.1205$ & $0.1506$\\
750 & $0.0646$ & $0.0803$ & $0.0891$ & $0.0737$ & $0.0862$ & $0.1016$ & $0.0836$ & $0.1010$ & $0.1177$\\
1000 & $0.0562$ & $0.0657$ & $0.0781$ & $0.0609$ & $0.0727$ & $0.0860$ & $0.0721$ & $0.0847$ & $0.0992$\\
\bottomrule
\end{tabular}
\caption{Mean absolute errors $\widehat{\mathsf E}_{\mathrm{eqv}}$ for Algorithm~\ref{alg:rect-correlation} under three rectangular variance-profile pairs, based on $B=500$ replications.}
\label{tab:algorithm4-joint-orbit-results}
\end{table}

\begin{table}[htbp]
\centering
\footnotesize
\setlength{\tabcolsep}{2.5pt}
\begin{tabular}{@{}l*{9}{c}@{}}
\toprule
\multicolumn{1}{c}{}
 & \multicolumn{3}{c}{$\big(T_{\mathrm{rec}}^{(1)},T_{\mathrm{rec}}^{(1)}\big)$}
 & \multicolumn{3}{c}{$\big(T_{\mathrm{rec}}^{(1)},T_{\mathrm{rec}}^{(4)}\big)$}
 & \multicolumn{3}{c}{$\big(T_{\mathrm{rec}}^{(4)},T_{\mathrm{rec}}^{(4)}\big)$}\\
\cmidrule(lr){2-4}\cmidrule(lr){5-7}\cmidrule(lr){8-10}
$n$
 & $\omega=0.25$ & $\omega=0.60$ & $\omega=0.90$
 & $\omega=0.25$ & $\omega=0.60$ & $\omega=0.90$
 & $\omega=0.25$ & $\omega=0.60$ & $\omega=0.90$\\
\midrule
& \multicolumn{9}{l}{Gaussian}\\
250 & $0.0113$ & $0.0279$ & $0.0611$ & $0.0114$ & $0.0328$ & $0.0667$ & $0.0149$ & $0.0369$ & $0.0698$\\
500 & $0.0068$ & $0.0166$ & $0.0328$ & $0.0073$ & $0.0183$ & $0.0362$ & $0.0086$ & $0.0220$ & $0.0423$\\
750 & $0.0053$ & $0.0137$ & $0.0239$ & $0.0058$ & $0.0142$ & $0.0256$ & $0.0065$ & $0.0166$ & $0.0299$\\
1000 & $0.0043$ & $0.0118$ & $0.0188$ & $0.0052$ & $0.0136$ & $0.0225$ & $0.0057$ & $0.0149$ & $0.0244$\\
\cmidrule{2-10}
& \multicolumn{9}{l}{Rademacher}\\
250 & $0.0106$ & $0.0272$ & $0.0604$ & $0.0112$ & $0.0301$ & $0.0645$ & $0.0137$ & $0.0389$ & $0.0725$\\
500 & $0.0064$ & $0.0184$ & $0.0325$ & $0.0073$ & $0.0184$ & $0.0337$ & $0.0081$ & $0.0223$ & $0.0419$\\
750 & $0.0051$ & $0.0129$ & $0.0236$ & $0.0058$ & $0.0144$ & $0.0254$ & $0.0065$ & $0.0169$ & $0.0313$\\
1000 & $0.0045$ & $0.0112$ & $0.0187$ & $0.0048$ & $0.0128$ & $0.0210$ & $0.0059$ & $0.0134$ & $0.0234$\\
\bottomrule
\end{tabular}
\caption{Mean absolute errors $\widehat{\mathsf E}_{\mathrm{sig}}$ for
Algorithm~\ref{alg:rect-correlation} under three rectangular
variance-profile pairs, based on $B=500$ replications.}
\label{tab:algorithm4-signal-correlation-results}
\end{table}

\section{Proofs of Lemmas}

\subsection{Proof of Lemma~\ref{lem:mask-remainder-bound}}

By \eqref{eq:split-signal-decomposition}, $\|\mathcal E\|_2=\|\mathcal R\|_2$.
Let $Q$ be the upper triangular matrix with entries
\[
        Q_{ij}
        \deq
        \begin{cases}
                (P_{ij}-1/2)S_{ij}, & i<j,\\
                \frac12(P_{ii}-1/2)S_{ii}, & i=j,\\
                0, & i>j.
        \end{cases}
\]
Then $\mathcal R=Q+Q^\top$, and the entries of $Q$ are independent and centered.
By \cite[Theorem~2]{Latala2004some},
\begin{align}
        \E\|Q\|_2
        \lesssim
        &\max_i\bigg(\sum_j\E|Q_{ij}|^2\bigg)^{1/2}
        +
        \max_j\bigg(\sum_i\E|Q_{ij}|^2\bigg)^{1/2}
        +
        \bigg(\sum_{i,j}\E|Q_{ij}|^4\bigg)^{1/4}.
        \label{eq:Latala-mask-remainder-bound}
\end{align}
Since $S=UDU^\top$,
\begin{equation*}
        \sum_{j=1}^n S_{ij}^2
        =
        (S^2)_{ii}
        =
        \sum_{k=1}^r d_k^2u_k(i)^2
        \leq
        rd_{\max}^2\mu_n^2,
\end{equation*}
where $u_k(i)$ is the $i$-th entry of the vector $u_k$.
The same bound holds for each column because $S$ is symmetric.  As
$\E|P_{ij}-1/2|^2=1/4$, the first two terms on the right hand side of
\eqref{eq:Latala-mask-remainder-bound} are bounded by $C_r d_{\max}\mu_n$.

For the fourth moment term, Cauchy--Schwarz gives
\[
        |S_{ij}|
        =
        |e_i^\top UDU^\top e_j|
        \leq
        d_{\max}\|U^\top e_i\|_2\|U^\top e_j\|_2
        \leq
        rd_{\max}\mu_n^2.
\]
Moreover,
\[
        \sum_{i,j=1}^n S_{ij}^2
        =
        \|S\|_{\mathrm F}^2
        =
        \sum_{k=1}^r d_k^2
        \leq
        rd_{\max}^2.
\]
Using $\E|P_{ij}-1/2|^4=1/16$, we obtain
\begin{align*}
        \sum_{i,j}\E|Q_{ij}|^4
        \leq
        \frac1{16}\sum_{i,j}S_{ij}^4
        \leq
        \frac1{16}\max_{i,j}S_{ij}^2
        \sum_{i,j}S_{ij}^2
        \leq
        C_r d_{\max}^4\mu_n^4.
\end{align*}
Thus the last term in \eqref{eq:Latala-mask-remainder-bound} is also bounded by
$C_r d_{\max}\mu_n$.  Since $\mathcal R=Q+Q^\top$,
\[
        \E\|\mathcal E\|_2
        =
        \E\|\mathcal R\|_2
        \leq
        2\E\|Q\|_2
        \leq
        C_r d_{\max}\mu_n.
\]
Markov's inequality now gives
$\|\mathcal E\|_2=\bigop(\mu_n) = \smallop(1)$. 

\subsection{Proof of Lemma~\ref{lem:block-structure}}
Recall that $\mathcal H \deq H_z - \E H_z.$
Let
\[
        J\deq \diag(I_n,-I_n,I_n,-I_n).
\]
Then $J(\E H_z)J=\E H_z$ and $J\mathcal HJ=-\mathcal H$.
For any deterministic matrix $D\in\mathbb C^{4n\times 4n}$,
\[
        \mathcal S[J D J]
        =
        \E[\mathcal HJDJ\mathcal H]
        =
        J\,\E[(J\mathcal HJ)D(J\mathcal HJ)]\,J
        =
        J\,\mathcal S[D]J.
\]
For notational simplicity, we abbreviate $M_z(\zeta)$ as $M$ in this proof.
Conjugating the MDE by $J$ therefore gives
\[
        -(JMJ)^{-1}=\zeta I_{4n}-\E H_z +\mathcal S[JMJ].
\]
By uniqueness of the MDE solution, $JMJ=M$.
This forces the entries in the $(1,2)$, $(1,4)$, $(2,1)$, $(2,3)$, $(3,2)$, $(3,4)$, $(4,1)$, and $(4,3)$ block positions to vanish.

Next fix a diagonal sign matrix
\[
        Q_\varepsilon=\diag(\varepsilon_1,\dots,\varepsilon_n),
        \qquad
        \varepsilon_i\in\{-1,+1\},
        \qquad
        \mathcal Q_\varepsilon\deq \diag(Q_\varepsilon,Q_\varepsilon,Q_\varepsilon,Q_\varepsilon).
\]
We claim that for every deterministic $D$,
\begin{equation}\label{eq:self-energy-sign-change}
        \mathcal S[\mathcal Q_\varepsilon D\mathcal Q_\varepsilon]=\mathcal Q_\varepsilon\mathcal S[D]\mathcal Q_\varepsilon.
\end{equation}
Indeed, for $a,b,c,d\in\intset{4}$ and $i,j,k,\ell\in\intset{n}$,
\[
        (\mathcal S[\mathcal Q_\varepsilon D\mathcal Q_\varepsilon])_{(a,i),(b,j)}
        =
        \sum_{c,d=1}^4\sum_{k,\ell=1}^n
        \E\Bigl(\mathcal H_{(a,i),(c,k)}\,D_{(c,k),(d,\ell)}\,\mathcal H_{(d,\ell),(b,j)}\,\varepsilon_k\varepsilon_\ell\Bigr),
\]
where the notation $A_{(a,i),(b,j)}$ denotes the $((a-1)n+i, (b-1)n+j)$-th entry of $A\in\mathbb C^{4n \times 4n}$.
If
\[
        \E\bigl(\mathcal H_{(a,i),(c,k)}\mathcal H_{(d,\ell),(b,j)}\bigr)\neq0,
\]
then the two $\mathcal H$ entries must involve the same underlying entry of $X$ and the same mask type ($\mathsf A$ or $\mathsf B$), hence
$
        \{i,k\}=\{j,\ell\}.
$
Therefore $\varepsilon_k\varepsilon_\ell=\varepsilon_i\varepsilon_j$, and the previous display becomes
\[
        (\mathcal S[\mathcal Q_\varepsilon D\mathcal Q_\varepsilon])_{(a,i),(b,j)}
        =
        \varepsilon_i\varepsilon_j\,(\mathcal S[D])_{(a,i),(b,j)}
        =
        (\mathcal Q_\varepsilon\mathcal S[D]\mathcal Q_\varepsilon)_{(a,i),(b,j)}.
\]
This proves \eqref{eq:self-energy-sign-change}.
Since $\mathcal Q_\varepsilon (\E H_z) \mathcal Q_\varepsilon = \E H_z$, uniqueness of the MDE implies $\mathcal Q_\varepsilon M\mathcal Q_\varepsilon=M$.
Fix any surviving block, say $M_{11}$, and let $i\neq j$.
Choose $Q_\varepsilon$ with $\varepsilon_i=-1$, $\varepsilon_j=1$, and $\varepsilon_k=1$ for all $k\notin\{i,j\}$.
Then
\[
        (M_{11})_{ij}=(Q_\varepsilon M_{11}Q_\varepsilon)_{ij}=\varepsilon_i\varepsilon_j(M_{11})_{ij}=-(M_{11})_{ij},
\]
so $(M_{11})_{ij}=0$.
The same argument applies to
$M_{22},M_{33},M_{44},M_{13},M_{24},M_{31},M_{42}$; hence all
surviving blocks are diagonal.

Now let
\[
        \Pi=
        \begin{pmatrix}
                0   & I_n & 0   & 0   \\
                I_n & 0   & 0   & 0   \\
                0   & 0   & 0   & I_n \\
                0   & 0   & I_n & 0
        \end{pmatrix}.
\]
Then $\Pi (\E H_z) \Pi=\E H_z$.
Because $P_{ij}\sim \mathrm{Bernoulli}(1/2)$, the pair $(\mathsf A, \mathsf B)$ has the same distribution as $(\mathsf B,\mathsf A)$, hence $\Pi \mathcal H\Pi\overset{d}=\mathcal H$.
Therefore, for deterministic $D$,
\[
        \mathcal S[\Pi D\Pi]
        =
        \E[\mathcal H\Pi D\Pi \mathcal H]
        =
        \Pi\,\E[(\Pi \mathcal H\Pi)D(\Pi \mathcal H\Pi)]\,\Pi
        =
        \Pi\mathcal S[D]\Pi.
\]
Uniqueness of the MDE gives $\Pi M\Pi=M$, so
\begin{equation}\label{eq:block-equality-pi}
        M_{11}=M_{22},
        \qquad
        M_{33}=M_{44},
        \qquad
        M_{13}=M_{24},
        \qquad
        M_{31}=M_{42}.
\end{equation}

Since the surviving blocks are diagonal, the $i$th diagonal coordinate of the MDE reduces to
\[
        -\begin{pmatrix}
                m_{z,i}(\zeta)          &\; w_{z,i}(\zeta) \\
                \widetilde w_{z,i}(\zeta) & \; \widetilde m_{z,i}(\zeta)
        \end{pmatrix}^{-1}
        =
        \begin{pmatrix}
                \zeta+\alpha_{z,i}(\zeta) & z         \\    
                \overline z    & \zeta+\beta_{z,i}(\zeta)
        \end{pmatrix},
\]
where
\[
        m_{z,i}(\zeta)\deq (M_{11})_{ii},
        \qquad
        \widetilde m_{z,i}(\zeta)\deq (M_{33})_{ii},
        \qquad
        w_{z,i}(\zeta)\deq (M_{13})_{ii},
        \qquad
        \widetilde w_{z,i}(\zeta)\deq (M_{31})_{ii},
\]
and
\[
        \alpha_{z,i}(\zeta)=\E(\mathsf AM_{33} \mathsf A)_{ii}=\sum_{k=1}^n(\mathbb T)_{ik}\widetilde m_{z,k}(\zeta),
        \qquad
        \beta_{z,i}(\zeta)=\E(\mathsf BM_{11}\mathsf B)_{ii}=\sum_{k=1}^n(\mathbb T)_{ik}m_{z,k}(\zeta).
\]
Hence,
\begin{align*}
        m_{z,i}(\zeta) & =-\frac{\zeta+\beta_{z,i}(\zeta)}{\{\zeta+\alpha_{z,i}(\zeta)\}\{\zeta+\beta_{z,i}(\zeta)\} - |z|^2},  \\
        \widetilde m_{z,i}(\zeta) & =-\frac{\zeta+\alpha_{z,i}(\zeta)}{\{\zeta+\alpha_{z,i}(\zeta)\}\{\zeta+\beta_{z,i}(\zeta)\} - |z|^2}, \\
        w_{z,i}(\zeta) & =\frac{z}{\{\zeta+\alpha_{z,i}(\zeta)\}\{\zeta+\beta_{z,i}(\zeta)\} - |z|^2}, \\
        \widetilde w_{z,i}(\zeta) & =\frac{\overline z}{\{\zeta+\alpha_{z,i}(\zeta)\}\{\zeta+\beta_{z,i}(\zeta)\} - |z|^2}.
\end{align*}
Swapping $m_{z,i}(\zeta)$ and $\widetilde m_{z,i}(\zeta)$ interchanges
$\alpha_{z,i}(\zeta)$ and $\beta_{z,i}(\zeta)$ while leaving
$w_{z,i}(\zeta)$ and $\widetilde w_{z,i}(\zeta)$ unchanged.  Therefore,
using \eqref{eq:block-equality-pi}, the block matrix
\[
        \widetilde M
        \deq
        \begin{pmatrix}
                M_{33} & 0      & M_{13} & 0      \\
                0      & M_{44} & 0      & M_{24} \\
                M_{31} & 0      & M_{11} & 0      \\
                0      & M_{42} & 0      & M_{22}
        \end{pmatrix}
\]
satisfies the same reduced equations in both paired block coordinates
and hence solves the full MDE.  By uniqueness, $\widetilde M=M$, so
\[
        M_{11}=M_{33},
        \qquad
        M_{22}=M_{44}.
\]
Combining this with \eqref{eq:block-equality-pi} yields
\[
        M_{11}=M_{22}=M_{33}=M_{44},
        \qquad
        M_{13}=M_{24},
        \qquad
        M_{31}=M_{42}.
\]
This proves the lemma.

\subsection{Proof of Lemma~\ref{lem:support}}

Let 
\begin{equation}\label{eq:ell-def}
        \ell(z)
        \deq
        \limsup_{\eta\downarrow0}
        \frac1\eta
        \max_{j\in\intset{n}}
        |\Im m_{z,j}(\ii\eta)|.
\end{equation}
For $\tau>0$, define
\[
        \mathbb D_\tau
        \deq
        \big\{
        z\in\mathbb C \;:\;
        \dist(0,\supp\nu_z)\leq\tau
        \big\},
        \qquad
        \widetilde{\mathbb D}_\tau
        \deq
        \big\{
        z\in\mathbb C \;:\;
        \ell(z)
        \geq 1/\tau
        \big\}.
\]
The set $\mathbb D_\tau$ is the self consistent $\tau$-pseudospectrum of $\mathcal X$. 
The following lemma compares the self consistent pseudospectra $\mathbb D_\tau$ and $\widetilde{\mathbb D}_\tau$.

\begin{lemma}
\label{lem:pseudospectrum-comparison}
Suppose that Assumption~\ref{asm:moment} hold.
There exist constants $c,\tau_0>0$, depending only on $C_*$ and $C^*$, such that, for every $\tau\in(0,\tau_0)$,
\[
        \mathbb D_{c\tau^{3/2}}
        \subseteq
        \widetilde{\mathbb D}_\tau
        \subseteq
        \mathbb D_{\tau^{1/2}}.
\]
\end{lemma}

The proof of Lemma~\ref{lem:pseudospectrum-comparison} is given in Section~\ref{sec:proof-lem-pseudospectrum-comparison}.

Recall the definitions of $m_{z,i}(\zeta)$ and $a_{z,i}(\zeta)$ in \eqref{eq:MDE_m_a}.
Set $\zeta=\ii\eta$ with $\eta>0$.
For each $i\in\intset{n}$, define
\[
        m^{\sharp}_{z,i}(\ii\eta)
        \deq
        -\overline{m_{z,i}(\ii\eta)},
        \qquad
        a^{\sharp}_{z,i}(\ii\eta)
        \deq
        -\overline{a_{z,i}(\ii\eta)}.
\]
Since the entries of $\mathbb T$ are real,
\[
        a^{\sharp}_{z,i}(\ii\eta)
        =
        \sum_{k=1}^n
        \mathbb T_{ik}m^{\sharp}_{z,k}(\ii\eta).
\]
Taking the negative complex conjugate of \eqref{eq:MDE_m_a} at $\zeta=\ii\eta$ gives
\[
        m^{\sharp}_{z,i}(\ii\eta)
        =
        -\frac{\ii\eta+a^{\sharp}_{z,i}(\ii\eta)}
        {\{\ii\eta+a^{\sharp}_{z,i}(\ii\eta)\}^2-|z|^2}.
\]
Thus $m^{\sharp}$ satisfies the same vector MDE as $m$.
Moreover, $\Im m^{\sharp}_{z,i}(\ii\eta)=\Im m_{z,i}(\ii\eta)>0$.
Uniqueness of the MDE solution therefore gives
\[
        m^{\sharp}_{z,i}(\ii\eta)
        =
        m_{z,i}(\ii\eta).
\]
Consequently,
\[
        m_{z,i}(\ii\eta)
        =
        -\overline{m_{z,i}(\ii\eta)},
\]
and $m_{z,i}(\ii\eta)$ is purely imaginary for every $i\in\intset{n}$.
Define
\[
        u
        \deq
        \bigl(
        \Im m_{z,1}(\ii\eta),
        \ldots,
        \Im m_{z,n}(\ii\eta)
        \bigr)^\top.
\]
Substituting $m_z(\ii\eta)=\ii u$ and $a_z(\ii\eta)=\ii\mathbb T u$ into \eqref{eq:MDE_m_a} yields
\begin{equation}
\label{202604211349}
        \frac1u
        =
        \eta+\mathbb T u
        +
        \frac{|z|^2}{\eta+\mathbb T u},
\end{equation}
where division is interpreted entrywise.
This equation is scaling invariant under $\mathbb T\mapsto c\mathbb T$ for $c>0$.
Indeed, if $\mathbb T$ is replaced by $c\mathbb T$, then
\[
        \widetilde\eta
        =
        c^{1/2}\eta,
        \qquad
        \widetilde z
        =
        c^{1/2}z,
        \qquad
        \widetilde u
        =
        c^{-1/2}u
\]
satisfy the corresponding rescaled equation
\begin{equation*}
        \frac{1}{\widetilde u}
        =
        \widetilde\eta+c\mathbb T\widetilde u
        +
        \frac{|\widetilde z|^2}{\widetilde\eta+c\mathbb T\widetilde u}.
\end{equation*}
Since $\mathbb T$ is symmetric and nonnegative, $\varrho(\mathbb T)=b_{\star,n}^2$.
By this rescaling, it suffices below to work under the normalization $\varrho(\mathbb T)=1$, for which $b_{\star,n}=1$.

The lemma below provides an estimate of the solution to \eqref{202604211349}.

\begin{lemma}\label{202604211403}
        For some $\varepsilon>0$, uniformly for $0<\eta\leq 1$ and $|z|>b_{\star,n}+\varepsilon$, the solution of \eqref{202604211349} satisfies
        \begin{equation}\label{202604221448}
                u(\eta)\asymp \langle u(\eta)\rangle \asymp \frac{\eta}{|z|^2-1+\eta^{2/3}}
        \end{equation}
        where $\langle u(\eta)\rangle \deq  \frac{1}{n}\sum_{i=1}^n u_i(\eta)$.
\end{lemma}
The proof of Lemma~\ref{202604211403} is given in Section~\ref{sec:proof-202604211403}.

For every fixed $\varepsilon>0$, Lemma~\ref{202604211403} gives, uniformly for $|z|>1+\varepsilon$,
\[
        \ell(z)
        \leq
        \frac{C}{\varepsilon^2+2\varepsilon}
        \eqd C_\varepsilon.
\]
Let $c$ and $\tau_0$ be the constants in Lemma~\ref{lem:pseudospectrum-comparison}, and choose
\[
        \tau_\varepsilon
        \deq
        \min
        \left\{
        \frac{\tau_0}{2},
        \frac{1}{2C_\varepsilon}
        \right\}.
\]
Then $\ell(z)<\tau_\varepsilon^{-1}$, so $z\notin\widetilde{\mathbb D}_{\tau_\varepsilon}$.
Using Lemma~\ref{lem:pseudospectrum-comparison}, we have
$
        z\notin
        \mathbb D_{c\tau_\varepsilon^{3/2}}.
$
Therefore,
\[
        \dist(0,\supp\nu_z)
        >
        c\tau_\varepsilon^{3/2}.
\]
The claim follows with
\[
        \delta'(\varepsilon)
        \deq
        c\tau_\varepsilon^{3/2}.
\]

\subsection{Proof of Lemma~\ref{lem:pseudospectrum-comparison}}\label{sec:proof-lem-pseudospectrum-comparison}

Set
\[
        g(z)\deq\dist(0,\supp\nu_z),
        \qquad
        \delta\deq |z|^2-b_{\star,n}^2,
\]
and write $\langle x\rangle\deq n^{-1}\sum_i x_i$ for vectors
$x\in\mathbb C^n$.  We also use the vector notation
\[
        m_z(\zeta)
        \deq
        \bigl(m_{z,1}(\zeta),\ldots,m_{z,n}(\zeta)\bigr)^\top.
\]
All constants below depend only on $C_*$ and $C^*$.

We first prove the second inclusion.
For each $j\in\intset{n}$, there exists a unique probability measure $\mu_{z,j}$ such that
\[
        m_{z,j}(\zeta)
        =
        \int_{\mathbb R}\frac{1}{x-\zeta}\,\dif\mu_{z,j}(x).
\]
Lemma~\ref{lem:block-structure} and equation~\eqref{eq:self consistent-density}
imply that
\[
        \nu_z=\frac1n\sum_{j=1}^n\mu_{z,j}.
\]
Since all these measures are nonnegative,
$\supp\mu_{z,j}\subseteq\supp\nu_z$.  If $g(z)>0$, then
\[
        \frac{\Im m_{z,j}(\ii\eta)}{\eta}
        =
        \int_{\mathbb R}\frac{\dif\mu_{z,j}(x)}{x^2+\eta^2}
        \leq
        \frac{1}{g(z)^2}.
\]
The definition of $\ell(z)$ in \eqref{eq:ell-def} therefore gives $\ell(z)\leq g(z)^{-2}$.
Thus $\ell(z)\geq\tau^{-1}$ implies $g(z)\leq\tau^{1/2}$.
If $g(z)=0$, this inequality holds trivially.
Hence
\[
        \widetilde{\mathbb D}_\tau
        \subseteq
        \mathbb D_{\tau^{1/2}}.
\]

We next analyze $\ell(z)$ separately in the regimes $\delta \leq 0$ and $\delta > 0$.
By the derivation preceding \eqref{202604211349}, for $\eta>0$ we may write
$m_z(\ii\eta)=\ii u(\eta)$ for a vector $u(\eta)>0$.  Define
\[
        h(\eta)\deq\eta\one_n+\mathbb T u(\eta).
\]
Equation~\eqref{202604211349} is equivalent to
\begin{equation}
\label{eq:corrected-u-equation}
        u_i(\eta)
        =
        \frac{h_i(\eta)}{|z|^2+h_i(\eta)^2},
        \qquad i\in\intset{n}.
\end{equation}

Suppose first that $\delta\leq0$. We claim that $\ell(z)=\infty$ and prove it by contradiction.
Assume $\ell(z)<\infty$, then $x_\eta\deq u(\eta)/\eta$ is bounded as $\eta\downarrow0$.  
From \eqref{eq:corrected-u-equation}, we have
\begin{equation}\label{eq:corrected-x-equation}
        |z|^2x_\eta
        +
        \eta^2x_\eta\circ
        (\one_n+\mathbb T x_\eta)^{\circ 2}
        =
        \one_n+\mathbb T x_\eta.
\end{equation}
Along a convergent subsequence of $x_\eta$ as $\eta\downarrow0$, its limit $x\geq0$ would therefore satisfy
$
        (|z|^2I_n-\mathbb T)x=\one_n.
$
By the Perron--Frobenius theorem, there exists a positive eigenvector $v_{\mathrm{PF}}$ of $\mathbb T$ corresponding to the largest eigenvalue $b_{\star,n}^2$.
Since $\mathbb T$ is symmetric, multiplying the last equation by $v_{\mathrm{PF}}^\top$
gives
$
        \delta\,v_{\mathrm{PF}}^\top x
        =
        v_{\mathrm{PF}}^\top\one_n,
$
which is impossible for $\delta\leq0$.  Consequently,
\begin{equation}
\label{eq:ell-inside}
        \ell(z)=\infty,
        \qquad \delta\leq0.
\end{equation}

Suppose now that $\delta>0$ and set $\mathbb T_z\deq |z|^2I_n-\mathbb T$.  Since
$|z|^2>b_{\star,n}^2$,
\begin{equation*}
        \mathbb T_z^{-1}
        =
        \frac1{|z|^2}\sum_{k=0}^\infty
        \left(\frac{\mathbb T}{|z|^2}\right)^k
        \geq0
\end{equation*}
entrywise.  
This, together with \eqref{eq:corrected-x-equation}, gives
\begin{equation}
\label{eq:x-eta-bound}
        0\leq x_\eta\leq \mathbb T_z^{-1}\one_n.
\end{equation}
Every subsequential limit $x$ for $x_\eta$ as $\eta\downarrow0$ satisfies $\mathbb T_z x=\one_n$ by \eqref{eq:corrected-x-equation}.  
Uniqueness of this solution thus gives
\begin{equation*}
        \ell(z)
        =
        \left\|\mathbb T_z^{-1}\one_n\right\|_\infty.
\end{equation*}
Multiplying $\mathbb T_z(\mathbb T_z^{-1}\one_n)=\one_n$ by
$v_{\mathrm{PF}}^\top$ gives
$
        \delta\,v_{\mathrm{PF}}^\top (\mathbb T_z^{-1}\one_n)
        =
        v_{\mathrm{PF}}^\top\one_n.
$
Since $v_{\mathrm{PF}}>0$, it follows that
\begin{equation}
\label{eq:ell-lower}
        \ell(z)\geq\frac1\delta, \qquad \delta>0.
\end{equation}

It remains to relate $\delta$ to the central gap of $\nu_z$.  We claim that
there is a constant $\kappa>0$ such that
\begin{equation}
\label{eq:gap-lower}
        g(z)
        \geq
        \kappa\min\{\delta^{3/2},1\},
        \qquad \delta>0.
\end{equation}
We prove this claim by constructing a local MDE solution near zero and then
identifying it with the MDE solution $m_z$ associated with the self consistent measure $\nu_z$.
Normalize $v_{\mathrm{PF}}$ by $\langle v_{\mathrm{PF}}^{\circ 2}\rangle=1$.  
Since $C_*/(2n)\leq(\mathbb T)_{ij}\leq C^*/(2n)$ and $\mathbb T v_{\mathrm{PF}}=b_{\star,n}^2v_{\mathrm{PF}}$, we have
\[
        \frac{C_*}{2}\langle v_{\mathrm{PF}}\rangle
        \leq
        b_{\star,n}^2 (v_{\mathrm{PF}})_i
        \leq
        \frac{C^*}{2}\langle v_{\mathrm{PF}}\rangle.
\]
This, together with the chosen normalization and $C_*/2\leq b_{\star,n}^2\leq C^*/2$, gives
\[
        b_{\star,n}^2\asymp1,
        \qquad
        (v_{\mathrm{PF}})_i\asymp1,
        \qquad i\in\intset{n}.
\]
Equip $\mathbb C^n$ with the weighted norm
\[ 
        \|q\|_v
        \deq
        \max_{i\in\intset{n}}
        \frac{|q_i|}{(v_{\mathrm{PF}})_i}.
\]
Positivity of $\mathbb T$ and
$\mathbb T v_{\mathrm{PF}}=b_{\star,n}^2 v_{\mathrm{PF}}$ imply
\begin{equation}
\label{eq:weighted-resolvent-bound}
        \|\mathbb T\|_{v\to v}\leq b_{\star,n}^2,
        \qquad
        \left\|
        \left(I_n-\frac{\mathbb T}{|z|^2}\right)^{-1}
        \right\|_{v\to v}
        \leq
        \frac{|z|^2}{\delta}.
\end{equation}
For a complex spectral parameter $\zeta$ near zero, set
\[
        q(\zeta,z)\deq\zeta\one_n+\mathbb T m_z(\zeta).
\]
Equation~\eqref{eq:MDE_m_a} then gives
\[
        m_z(\zeta)=\frac{q(\zeta,z)}{|z|^2-q(\zeta,z)^{\circ 2}}.
\]
For brevity, write $q=q(\zeta,z)$ below. 
The vector MDE then reduces to
the fixed-point equation $q=\Phi_\zeta(q)$, where
\begin{equation}
\label{eq:corrected-fixed-point}
        \Phi_\zeta(q)
        \deq
        \left(I_n-\frac{\mathbb T}{|z|^2}\right)^{-1}
        \left\{
        \zeta\one_n+\mathbb T\Psi_z(q)
        \right\},
        \qquad
        (\Psi_z(q))_i
        \deq
        \frac{q_i^3}{|z|^2(|z|^2-q_i^2)}.
\end{equation}
Since $\delta>0$, we have $|z|^2>b_{\star,n}^2\geq C_*/2$, while $(v_{\mathrm{PF}})_i\asymp1$.  
Hence there is a fixed $\varepsilon_0>0$ such that $\|q\|_v\vee\|p\|_v\leq\varepsilon_0$ implies
\[
        |q_i|\vee|p_i|\leq\frac{|z|}{2},
        \qquad
        \bigl||z|^2-q_i^2\bigr|
        \wedge
        \bigl||z|^2-p_i^2\bigr|
        \geq
        \frac{3|z|^2}{4}.
\]
Hence, 
\[
        |(\Psi_z(q))_i|\leq C|q_i|^3,
\qquad
\bigg|\frac{\partial (\Psi_z(q))_i}{\partial q_i}\bigg| = \bigg|\frac{q_i^2(3|z|^2-q_i^2)}{|z|^2(|z|^2-q_i^2)^2}\bigg| \leq C |q_i|^2.
\]
By the fundamental theorem of calculus, we have
\[
        |(\Psi_z(q)-\Psi_z(p))_i|
        \leq
        C(|q_i|^2+|p_i|^2)|q_i-p_i|.
\]
Using $|q_i|\leq(v_{\mathrm{PF}})_i\|q\|_v$ and the analogous bounds for $p_i$ and $q_i-p_i$, then taking the maximum over $i$, yields
\begin{equation}\label{eq:corrected-cubic-bounds}
        \|\Psi_z(q)\|_v
        \leq C\|q\|_v^3,
        \qquad
        \|\Psi_z(q)-\Psi_z(p)\|_v
        \leq
        C(\|q\|_v^2+\|p\|_v^2)\|q-p\|_v.
\end{equation}
Fix a sufficiently small $\delta_0>0$.  If $0<\delta\leq\delta_0$, then
\eqref{eq:weighted-resolvent-bound} and
\eqref{eq:corrected-cubic-bounds} give, whenever
$\|q\|_v\vee\|p\|_v\leq\varepsilon_0$,
\[
        \|\Phi_\zeta(q)\|_v
        \leq
        \frac{C}{\delta}
        \bigl(|\zeta|+\|q\|_v^3\bigr),
        \qquad
        \|\Phi_\zeta(q)-\Phi_\zeta(p)\|_v
        \leq
        \frac{C}{\delta}
        (\|q\|_v^2+\|p\|_v^2)\|q-p\|_v.
\]
Choose $\rho>0$ and then $\kappa_1>0$ sufficiently small so that
\[
        2C\rho^2<1,
        \qquad
        C(\kappa_1+\rho^3)\leq\rho,
        \qquad
        \rho\sqrt{\delta_0}\leq\varepsilon_0.
\]
If $q$ and $p$ belong to the ball
\[
        B_\delta\deq\{q:\|q\|_v\leq\rho\sqrt\delta\}
\]
and $|\zeta|<\kappa_1\delta^{3/2}$, then the preceding inequalities imply
\[
        \|\Phi_\zeta(q)\|_v
        \leq
        C(\kappa_1+\rho^3)\sqrt\delta
        \leq
        \rho\sqrt\delta,
        \qquad 
        \|\Phi_\zeta(q)-\Phi_\zeta(p)\|_v
        \leq
        2C\rho^2\|q-p\|_v.
\]
Thus $\Phi_\zeta$ maps $B_\delta$ into itself and is a contraction there.
If $\delta\geq\delta_0$, the second bound in \eqref{eq:weighted-resolvent-bound} is uniform, and the same argument on a fixed ball gives a contraction whenever $|\zeta|<\kappa_2$.  
After decreasing $\kappa>0$ if necessary, both regimes yield a unique fixed
point, denoted by $q_{\star}(\zeta,z)$, in the corresponding contraction ball for
\[
        |\zeta|<R_\delta
        \deq\kappa\min\{\delta^{3/2},1\}.
\]
Moreover, $\Psi_z(0)=0$, and hence $\Phi_0(0)=0$.  Since $q_{\star}(0,z)$ is
the unique fixed point of $\Phi_0$ in the $B_{\delta}$, it follows that $q_{\star}(0,z)=0$.  Let $E\in\mathbb R$ satisfy
$|E|<R_\delta$.  By construction,
$
        q_{\star}(E,z)=\Phi_E(q_{\star}(E,z)).
$
This, together with \eqref{eq:corrected-fixed-point}, implies that
\[
        \Phi_E(\overline{q_{\star}(E,z)})
        =
        \overline{\Phi_E(q_{\star}(E,z))}
        =
        \overline{q_{\star}(E,z)}.
\]
Uniqueness of the fixed point therefore gives
$\overline{q_{\star}(E,z)}=q_{\star}(E,z)$, and hence
$q_{\star}(E,z)\in\mathbb R^n$.

We finally identify $q_{\star}(\zeta,z)$ with the function $q(\zeta,z)$ obtained from the MDE solution $m_z$.  
For $\zeta=\ii\eta$, the bound \eqref{eq:x-eta-bound} gives $0\leq u(\eta)\leq\eta \mathbb T_z^{-1}\one_n.$
Therefore, the associated function $q$ satisfies
\[
        q(\ii\eta,z)
        =
        \ii\{\eta\one_n+\mathbb T u(\eta)\}
        \to 0
        \qquad (\eta\downarrow0).
\]
For all sufficiently small $\eta$, this solution lies in $B_{\delta}$ and hence equals $q_{\star}(\ii\eta,z)$.
The identity theorem shows that the identity $q(\zeta,z)=q_{\star}(\zeta,z)$ holds in $\{\zeta \in \mathbb C : |\zeta|<R_\delta, \Im \zeta >0 \}$. 
Consequently, every component $m_{z,j}$ extends holomorphically through $(-R_\delta,R_\delta)$ via $m_z(\zeta)=q_{\star}(\zeta,z)/ (|z|^2-q_{\star}(\zeta,z)^{\circ 2})$, and is real on this interval.
The Stieltjes inversion formula then gives
\[
        (-R_\delta,R_\delta)\cap\supp\nu_z=\varnothing,
\]
which proves \eqref{eq:gap-lower}.

We now prove the first inclusion.  Choose $c\leq\kappa$ and
$0<\tau_0<1$.  Let $0<\tau<\tau_0$ and suppose that
$z\in\mathbb D_{c\tau^{3/2}}$, so that $g(z)\leq c\tau^{3/2}$.
If $\delta\leq0$, then \eqref{eq:ell-inside} gives $\ell(z)=\infty$.
If $\delta>0$, then \eqref{eq:gap-lower} gives
\[
        \kappa\min\{\delta^{3/2},1\}
        \leq
        c\tau^{3/2}.
\]
The case $\delta\geq1$ is impossible because $c\leq\kappa$ and $\tau<1$.
Thus
\[
        \delta
        \leq
        \left(\frac{c}{\kappa}\right)^{2/3}\tau
        \leq
        \tau.
\]
By \eqref{eq:ell-lower}, $\ell(z)\geq\delta^{-1}\geq\tau^{-1}$.
In both cases $z\in\widetilde{\mathbb D}_\tau$, and therefore
\[
        \mathbb D_{c\tau^{3/2}}
        \subseteq
        \widetilde{\mathbb D}_\tau.
\]
This completes the proof.

\subsection{Proof of Lemma~\ref{202604211403}}\label{sec:proof-202604211403}

Here we follow the proof strategy of Proposition 3.2 in \cite{AltErdosKruger2018}. The algebraic operations and comparisons are to be understood entrywise.

We first need to show an auxiliary bound,
\begin{equation}\label{202604221205}
        \eta \lesssim \langle u\rangle \lesssim 1.
\end{equation}
From \eqref{202604211349} and Assumption \ref{asm:moment}, we have
\begin{equation}\label{202604221153}
        u=\frac{\eta+\mathbb T u}{(\eta+\mathbb T u)(\eta+\mathbb T u)+|z|^2}\asymp \frac{\eta+\langle u\rangle}{(\eta+\langle u\rangle)(\eta+\langle u\rangle)+|z|^2}.
\end{equation}
Since $|z|>b_{\star,n} + \varepsilon \gtrsim 1$, then the lower bound follows immediately from \eqref{202604221153}.

To show the upper bound, we have from \eqref{202604211349} that
\begin{equation*}
        1=\eta u+u(\mathbb T u)+\frac{|z|^2u}{\eta+\mathbb T u}\geq u(\mathbb T u).
\end{equation*}
Taking the average gives
\begin{equation*}
        1\geq \langle u(\mathbb T u)\rangle\gtrsim \langle u\rangle ^2.
\end{equation*}
Thus the upper bound in \eqref{202604221205} follows. Now \eqref{202604221153} implies 
\begin{equation}\label{202604221450}
        u\asymp \frac{1}{\eta +\langle u \rangle+\frac{|z|^2}{\eta+\langle u\rangle}}\asymp \langle u\rangle,
\end{equation}
and
\begin{equation}\label{202604221434}
        \eta=u(\eta+\mathbb T u)(\eta+\mathbb T u)+|z|^2u-\mathbb T u.
\end{equation}
Since we assume $\varrho(\mathbb T)=1$, by the Perron-Frobenius theorem, there exists a vector $v \in \mathbb{R}^{n}_+$  such that
\begin{equation*}
        \mathbb T v = v , \quad \langle v  \rangle=1, \quad v \asymp 1.
\end{equation*} 
Taking the scalar product of \eqref{202604221434} with $v$, we get
\begin{equation*}
        \eta=\langle (u(\eta+\mathbb T u)(\eta+\mathbb T u))v \rangle+(|z|^2-1)\langle v u\rangle \asymp \langle u\rangle^3+(|z|^2-1)\langle u\rangle.
\end{equation*}
Now we can conclude the estimate \eqref{202604221448} for $\langle u\rangle$, and thus for $u$ by \eqref{202604221450}.

\subsection{Proof of Lemma~\ref{lem:local-law}}\label{sec:proof-local-law}

\subsubsection{Cumulant norms and correlation assumptions}\label{sec:EKS-cumulant-assumptions}

In this section, we recall the cumulant norms and correlation assumptions from \cite{Erdos2019random}, which are used to establish the local law for $H_z$ in Lemma~\ref{lem:local-law}.
We state the real-symmetric version, which is the version needed for the random part of the linearization considered in this paper.  Let $\mathcal J$ be the row--column index set of the matrix under consideration, let $N\deq |\mathcal J|$, and set $\mathcal I\deq \mathcal J\times\mathcal J$.  For $a,b\in\mathcal J$ we abbreviate the ordered pair $(a,b)\in\mathcal I$ by $ab$.  If $\mathcal W=(\mathfrak w_\alpha)_{\alpha\in\mathcal I}$ denotes the centered, $N^{1/2}$-rescaled random part of the matrix, then for $k\geq 2$ we write
\[
  \kappa(\alpha_1,\ldots,\alpha_k)
  \deq
  \kappa(\mathfrak w_{\alpha_1},\ldots,\mathfrak w_{\alpha_k}),
  \qquad \alpha_1,\ldots,\alpha_k\in\mathcal I,
\]
for the joint cumulant. 
We use the same notation $\kappa(\xi_1,\ldots,\xi_k)$ for the joint cumulant of arbitrary random variables. 
All sums over variables of the form $a_i,b_i$ are over $\mathcal J$, while all sums over variables of the form $\alpha_i$ are over $\mathcal I$.

\begin{definition}[Cumulant norms]\label{def:EKS-cumulant-norms}
For each $k\geq2$, define
\[
  |\!|\!|\kappa|\!|\!|_k
  \deq
  |\!|\!|\kappa|\!|\!|^{\mathrm{av}}_k
  +
  |\!|\!|\kappa|\!|\!|^{\mathrm{iso}}_k,
  \qquad
  |\!|\!|\kappa|\!|\!|_{\le R}^{\mathrm{av}}
  \deq
  \max_{2\leq k\leq R}|\!|\!|\kappa|\!|\!|_k^{\mathrm{av}},
  \qquad
  |\!|\!|\kappa|\!|\!|_{\le R}^{\mathrm{iso}}
  \deq
  \max_{2\leq k\leq R}|\!|\!|\kappa|\!|\!|_k^{\mathrm{iso}}.
\]
The averaged norms are
\begin{align}
  |\!|\!|\kappa|\!|\!|_2^{\mathrm{av}}
  &\deq
  \left\|
    \bigl(|\kappa(\alpha,\beta)|\bigr)_{\alpha,\beta\in\mathcal I}
  \right\|,
  \label{eq:EKS-av2}
  \\
  |\!|\!|\kappa|\!|\!|_k^{\mathrm{av}}
  &\deq
  N^{-2}\sum_{\alpha_1,\ldots,\alpha_k\in\mathcal I}
  |\kappa(\alpha_1,\ldots,\alpha_k)|,
  \qquad k\geq4,
  \label{eq:EKS-av-high}
\end{align}
and, for third cumulants,
\begin{align}
  |\!|\!|\kappa|\!|\!|_3^{\mathrm{av}}
  \deq{}&
  \left\|
    \left(
      \sum_{\alpha_1\in\mathcal I}|\kappa(\alpha_1,\beta,\gamma)|
    \right)_{\beta,\gamma\in\mathcal I}
  \right\|
  \notag\\
  &\quad+
  \inf_{\kappa=\kappa_{dd}+\kappa_{dc}+\kappa_{cd}+\kappa_{cc}}
  \Bigl(
    |\!|\!|\kappa_{dd}|\!|\!|_{dd}
    +|\!|\!|\kappa_{dc}|\!|\!|_{dc}
    +|\!|\!|\kappa_{cd}|\!|\!|_{cd}
    +|\!|\!|\kappa_{cc}|\!|\!|_{cc}
  \Bigr),
  \label{eq:EKS-av3}
\end{align}
where the infimum is over all decompositions of the third cumulant into four symmetric functions $\kappa_{dd}$, $\kappa_{dc}$, etc. 
The letters $d$ and $c$ in the subscripts refer to ``direct'' and ``cross''.
The corresponding norms appearing in \eqref{eq:EKS-av3} are defined by
\begin{align}
  |\!|\!|\kappa|\!|\!|_{cc}
  =|\!|\!|\kappa|\!|\!|_{dd}
  &\deq
  N^{-1}
  \left[
    \sum_{b_2,a_3}
    \left(
      \sum_{a_2,b_3}\sum_{\alpha_1}
      |\kappa(\alpha_1,a_2b_2,a_3b_3)|
    \right)^2
  \right]^{1/2},
  \label{eq:EKS-dd-cc}
  \\
  |\!|\!|\kappa|\!|\!|_{cd}
  &\deq
  N^{-1}
  \left[
    \sum_{b_3,a_1}
    \left(
      \sum_{a_3,b_1}\sum_{\alpha_2}
      |\kappa(a_1b_1,\alpha_2,a_3b_3)|
    \right)^2
  \right]^{1/2},\notag
  \\
  |\!|\!|\kappa|\!|\!|_{dc}
  &\deq
  N^{-1}
  \left[
    \sum_{b_1,a_2}
    \left(
      \sum_{a_1,b_2}\sum_{\alpha_3}
      |\kappa(a_1b_1,a_2b_2,\alpha_3)|
    \right)^2
  \right]^{1/2}.\notag
\end{align}
For $x=(x_a)_{a\in\mathcal J}\in\ell^2(\mathcal J)$ define two nonnegative $\mathcal J\times\mathcal J$ matrices
\begin{align*}
  \bigl(A_d^{\kappa,x}\bigr)_{ij}
  \deq
  \left[
    \sum_{\ell\in\mathcal J}
    \left|
      \sum_{a\in\mathcal J}x_a\kappa(ai,\ell j)
    \right|^2
  \right]^{1/2},
	\qquad
  \bigl(A_c^{\kappa,x}\bigr)_{ij}
  \deq
  \left[
    \sum_{\ell\in\mathcal J}
    \left|
      \sum_{a\in\mathcal J}x_a\kappa(ai,j\ell)
    \right|^2
  \right]^{1/2}.
\end{align*}
Then
\[
  |\!|\!|\kappa|\!|\!|_d
  \deq
  \sup_{\lVert x\rVert_2\leq1}
  \lVert A_d^{\kappa,x}\rVert,
  \qquad
  |\!|\!|\kappa|\!|\!|_c
  \deq
  \sup_{\lVert x\rVert_2\leq1}
  \lVert A_c^{\kappa,x}\rVert.
\]
The second cumulant isotropic norm is
\begin{equation}
  |\!|\!|\kappa|\!|\!|_2^{\mathrm{iso}}
  \deq
  \inf_{\kappa=\kappa_d+\kappa_c}
  \Bigl(
    |\!|\!|\kappa_d|\!|\!|_d
    +|\!|\!|\kappa_c|\!|\!|_c
  \Bigr),
  \label{eq:EKS-iso2}
\end{equation}
where the infimum is over all decompositions of $\kappa$ into the sum of symmetric $\kappa_d$ and $\kappa_c$.  
Finally, for $k\geq3$,
\begin{equation}
  |\!|\!|\kappa|\!|\!|_k^{\mathrm{iso}}
  \deq
  \left\|
    \left(
      \sum_{\alpha_1,\ldots,\alpha_{k-2}\in\mathcal I}
      |\kappa(\alpha_1,\ldots,\alpha_{k-2},\beta,\gamma)|
    \right)_{\beta,\gamma\in\mathcal I}
  \right\|.
  \label{eq:EKS-isok}
\end{equation}
\end{definition}

\begin{assumption}[Assumption (C) in \cite{Erdos2019random}, $\kappa$--correlation decay]\label{asm:EKS2019-C}
There exist constants $C, C_{\varepsilon,R}<\infty$, independent of $N$, such that for every $R\in\mathbb N$ and every $\varepsilon>0$,
\[
  |\!|\!|\kappa|\!|\!|_2^{\mathrm{iso}}\leq C,
  \qquad
  \max_{2\leq k\leq R}|\!|\!|\kappa|\!|\!|_k
  \leq C_{\varepsilon,R}N^{\varepsilon}.
\]
\end{assumption}

\begin{assumption}[Assumption (D) in \cite{Erdos2019random}, higher-order correlation decay]\label{asm:EKS2019-D}
There exists $\mu>0$ such that the following holds: For every $\alpha\in\mathcal I$ and every $q,R\in\mathbb N$, there is a sequence of nested sets
\[
  \alpha\in\mathcal N_1(\alpha)
  \subset \mathcal N_2(\alpha)
  \subset \cdots
  \subset \mathcal N_R(\alpha)
  =:\mathcal N(\alpha)
  \subset\mathcal I,
  \qquad
  |\mathcal N(\alpha)|\leq N^{1/2-\mu}.
\]
For any $n_1,\ldots,n_q<R$, any $\alpha_1,\ldots,\alpha_q\in\mathcal I$, and any functions $f,g_1,\ldots,g_q$ of the indicated variables,
\begin{align*}
&\left|
\kappa\Bigl(
  f\bigl(\mathcal W_{\mathcal I\setminus\bigcup_{j=1}^q\mathcal N_{n_j+1}(\alpha_j)}\bigr),
  g_1\bigl(\mathcal W_{\mathcal N_{n_1}(\alpha_1)\setminus\bigcup_{j\neq1}\mathcal N(\alpha_j)}\bigr),
  \ldots,
  g_q\bigl(\mathcal W_{\mathcal N_{n_q}(\alpha_q)\setminus\bigcup_{j\neq q}\mathcal N(\alpha_j)}\bigr)
\Bigr)
\right|
\notag\\
&\hspace{5cm}
\leq
C_{R,q,\mu}N^{-3q}
\lVert f\rVert_{q+1}
\prod_{j=1}^q\lVert g_j\rVert_{q+1}.
\end{align*}
Here, for $\mathcal A\subset\mathcal I$, $\mathcal W_{\mathcal A}$ denotes the family $(\mathfrak w_\alpha)_{\alpha\in\mathcal A}$, and $\lVert\cdot\rVert_p$ denotes the $L^p$ norm with respect to the underlying randomness. 
\end{assumption}

\subsubsection{Proof of Lemma~\ref{lem:local-law}}

We verify the assumptions of the outside-spectrum local law of
\cite[Theorem~2.1]{Erdos2019random} for the Hermitian matrix $H_z$.

Since $|z|\le C_z$, we have $\|\E H_z\|=|z|\le C_z$, which verifies the bounded expectation assumption (see \cite[Assumption (A)]{Erdos2019random}).  
Moreover, we define $\mathcal W\deq \sqrt N\,\mathcal H$ and write
\[
    H_z=\E H_z+N^{-1/2}\mathcal W.
\]

Let
\[
    \mathcal J\deq \intset{4}\times\intset{n},
    \qquad
    \mathcal I\deq \mathcal J\times\mathcal J.
\]
The entries of $\mathcal W$ are
\[
    \mathfrak w_{(r,i),(s,j)}
    =
    \begin{cases}
        \sqrt N\,P_{ij}X_{ij}, & (r,s)\in\{(1,4),(4,1)\},\\[2mm]
        \sqrt N\,(1-P_{ij})X_{ij}, & (r,s)\in\{(2,3),(3,2)\},\\[2mm]
        0, & \text{otherwise}.
    \end{cases}
\]
Since $\sqrt N=2\sqrt n$ and $0\le P_{ij}\le1$, Assumption~\ref{asm:moment} implies that, for every fixed $q\in\mathbb N$,
\[
    \sup_{\alpha\in\mathcal I}\mathbb E|\mathfrak w_\alpha|^q\le C_q.
\]
This verifies the finite moment assumption (see \cite[Assumption (B)]{Erdos2019random}).

We now verify the cumulant assumptions.  For each unordered pair
\[
    e=\{i,j\},\qquad i,j\in\intset{n},\quad i\leq j,
\]
define its companion class by
\[
    \mathscr C_e
    \deq 
    \left\{
      ((r,k),(s,\ell))\in\mathcal I:
      (r,s)\in\{(1,4),(4,1),(2,3),(3,2)\},
      \ \{k,\ell\}=e
    \right\}.
\]
Thus $|\mathscr C_e|\le8$, and $|\mathscr C_e|\le4$ when $i=j$.  
We call a label $\alpha=((r,i),(s,j))\in\mathcal I$ \emph{nonzero} if
$(r,s)\in\{(1,4),(4,1),(2,3),(3,2)\}$, so that $\mathfrak w_\alpha$ is not identically zero.  We call it \emph{zero} if
$\mathfrak w_\alpha\equiv0$.  
If $\alpha$ is a nonzero label, let $\mathscr C(\alpha)$ denote the unique
class $\mathscr C_e$ containing $\alpha$.  If $\alpha$ is a zero label,
we set $\mathscr C(\alpha)\deq \{\alpha\}$.

The random variables attached to different nonzero companion classes are
independent, because they depend on disjoint families
\[
    (X_{ij},P_{ij}),\qquad i,j\in\intset{n},\quad i\leq j.
\]
Therefore, by the basic vanishing property of joint cumulants,
\[
    \kappa(\alpha_1,\ldots,\alpha_k)
    \deq 
    \kappa(\mathfrak w_{\alpha_1},\ldots,\mathfrak w_{\alpha_k})
    =0
\]
unless all nonzero labels $\alpha_1,\ldots,\alpha_k$ belong to the same
companion class.  

For every fixed $q\in\mathbb N$, the rescaled entries
$\mathfrak w_\alpha$ have uniformly bounded $q$-th moments.  Indeed,
if $\mathfrak w_\alpha$ is not identically zero, then it is either
$\sqrt N P_{ij}X_{ij}$ or $\sqrt N(1-P_{ij})X_{ij}$.  Since $N=4n$
and $0\le P_{ij}\le1$,
\[
    \mathbb E|\mathfrak w_\alpha|^q
    \le
    2^q\mathbb E|\sqrt n X_{ij}|^q
    \le C_q'.
\]
If $\mathfrak w_\alpha\equiv0$, the same bound is trivial.  Hence
\[
    \sup_{\alpha\in\mathcal I}\mathbb E|\mathfrak w_\alpha|^q
    \le C_q'.
\]
By the moment--cumulant formula,
\[
    \kappa(\mathfrak w_{\alpha_1},\ldots,\mathfrak w_{\alpha_k})
    =
    \sum_{\pi\in\mathcal P_k}
    (|\pi|-1)!(-1)^{|\pi|-1}
    \prod_{B\in\pi}
    \mathbb E\prod_{r\in B}\mathfrak w_{\alpha_r},
\]
where $\mathcal P_k$ is the set of partitions of $\{1,\ldots,k\}$, and $|\pi|$ is the number of blocks in the partition $\pi$.
Applying H{\" o}lder's inequality to
each block $B\in\pi$, we get
\[
    \left|
    \mathbb E\prod_{r\in B}\mathfrak w_{\alpha_r}
    \right|
    \le
    \prod_{r\in B}
    \left(\mathbb E|\mathfrak w_{\alpha_r}|^{|B|}\right)^{1/|B|}
    \le (C_k')^{|B|/k}.
\]
Therefore
\[
    \prod_{B\in\pi}
    \left|
    \mathbb E\prod_{r\in B}\mathfrak w_{\alpha_r}
    \right|
    \le (C_k')^{\frac{1}{k}\sum_{B\in\pi}|B|} = C_k'.
\]
Since the number of partitions of $\{1,\ldots,k\}$ depends only on $k$,
we conclude that
\[
    \bigl|\kappa(\alpha_1,\ldots,\alpha_k)\bigr|
    =
    \bigl|\kappa(\mathfrak w_{\alpha_1},\ldots,\mathfrak w_{\alpha_k})\bigr|
    \le C_k' \sum_{\pi\in\mathcal P_k}(|\pi|-1)! =: C_k
\]
uniformly in $\alpha_1,\ldots,\alpha_k\in\mathcal I$.

\medskip
\noindent\emph{Verification of Assumption~\ref{asm:EKS2019-C}.}

First consider $k\ge4$.  Since a nonzero cumulant forces all labels to
lie in one companion class, we have
\[
    |\!|\!|\kappa|\!|\!|_k^{\mathrm{av}}
    \overset{\eqref{eq:EKS-av-high}}{=}
    N^{-2}
    \sum_{\alpha_1,\ldots,\alpha_k\in\mathcal I}
    |\kappa(\alpha_1,\ldots,\alpha_k)|  
    \le
    N^{-2}
    \sum_{\alpha_1\in\mathcal I}
    \sum_{\alpha_2,\ldots,\alpha_k\in\mathscr C(\alpha_1)}
    C_k
    \le
    C_k\,8^{k-1},
\]
because $|\mathcal I|=N^2$.

For $k\ge3$, define the nonnegative matrix
\[
    M^{(k)}_{\beta\gamma}
    \deq 
    \sum_{\alpha_1,\ldots,\alpha_{k-2}\in\mathcal I}
    |\kappa(\alpha_1,\ldots,\alpha_{k-2},\beta,\gamma)|,
    \qquad
    \beta,\gamma\in\mathcal I.
\]
Then $M^{(k)}_{\beta\gamma}=0$ unless $\gamma\in\mathscr C(\beta)$, and
in the latter case
$
    M^{(k)}_{\beta\gamma}\le C_k\,8^{k-2}.
$
Each row and each column of $M^{(k)}$ has at most $8$ nonzero entries.  Thus, \label{errpage:s4-cumulant-matrix}
\[
	\max\big\{ \|M^{(k)}\|_1, \; \|M^{(k)}\|_{\infty} \big\} \le C_k\,8^{k-1}.
\]
Therefore,
\[
    |\!|\!|\kappa|\!|\!|_k^{\mathrm{iso}}
    \overset{\eqref{eq:EKS-isok}}{=}
    \|M^{(k)}\|
    \le
	\sqrt{\|M^{(k)}\|_1 \cdot \|M^{(k)}\|_{\infty}}
	\leq
    C_k\,8^{k-1},
    \qquad k\ge3.
\]

For $k=2$, the averaged norm satisfies the same estimate:
\[
    |\!|\!|\kappa|\!|\!|_2^{\mathrm{av}}
    \overset{\eqref{eq:EKS-av2}}{=}
    \left\|
        \bigl(|\kappa(\alpha,\beta)|\bigr)_{\alpha,\beta\in\mathcal I}
    \right\|
    \le C.
\]

It remains to bound the special second-order isotropic norm.  
For a nonzero label $\alpha=((r,i),(s,j))\in\mathcal I$, we write
\[
    \varpi(\alpha)\deq (i,j),
    \qquad
    \varpi(\alpha)^t\deq (j,i).
\]
Define relations $D(\alpha,\beta)$ and $C(\alpha,\beta)$ as follows:
\[
    D(\alpha,\beta)
    \quad\Longleftrightarrow\quad
    \varpi(\beta)=\varpi(\alpha),
\]
and
\[
    C(\alpha,\beta)
    \quad\Longleftrightarrow\quad
    \varpi(\beta)=\varpi(\alpha)^t
    \ \text{and the two coordinates of }\varpi(\alpha)\text{ are distinct}.
\]
Set
\[
    \kappa_d(\alpha,\beta)
    \deq 
    \kappa(\alpha,\beta)\indicator_{D(\alpha,\beta)},
    \qquad
    \kappa_c(\alpha,\beta)
    \deq 
    \kappa(\alpha,\beta)\indicator_{C(\alpha,\beta)}.
\]
Then $\kappa=\kappa_d+\kappa_c$, and both $\kappa_d,\kappa_c$ are
symmetric functions on $\mathcal I^2$.

We prove $|\!|\!|\kappa_d|\!|\!|_d=\bigo(1)$.  Let
$u=(s,j), v=(t,\ell) \in \mathcal J$.\label{errpage:s4-cumulant-direct}
By definition,
\[
    \bigl(A_d^{\kappa_d,x}\bigr)_{uv}^2
    =
    \sum_{m\in\mathcal J}
    \left|
        \sum_{a\in\mathcal J}
        x_a\,\kappa_d(au,mv)
    \right|^2.
\]
Write $a=(r,i)$ and $m=(p,k)$.  If
$\kappa_d(au,mv)\neq0$, then the direct relation forces
\[
    (k,\ell)=\varpi(mv)=\varpi(au)=(i,j).
\]
Therefore the entry $\bigl(A_d^{\kappa_d,x}\bigr)_{uv}$ vanishes unless
$\ell=j$.  In each nonzero summand, the direct relation also forces $k=i$.
Under these conditions, we have
\[
    \bigl(A_d^{\kappa_d,x}\bigr)_{uv}^2
    \le
    C
    \sum_{p=1}^4\sum_{i=1}^n
    \left(
        \sum_{r=1}^4 |x_{(r,i)}|
    \right)^2                                      
    \le
    C\sum_{r=1}^4\sum_{i=1}^n |x_{(r,i)}|^2
    =
    C\|x\|_2^2.
\]
Thus, for $\|x\|_2\le1$, all entries of $A_d^{\kappa_d,x}$ are bounded
by a constant, and each row and each column has at most $4$ nonzero entries.  Hence
\[
    \|A_d^{\kappa_d,x}\|\le C,
    \qquad
    |\!|\!|\kappa_d|\!|\!|_d\le C.
\]
Similarly, we have
$
    |\!|\!|\kappa_c|\!|\!|_c\le C.
$
Hence, we have
\[
    |\!|\!|\kappa|\!|\!|_2^{\mathrm{iso}}
    \overset{\eqref{eq:EKS-iso2}}{\le}
    |\!|\!|\kappa_d|\!|\!|_d
    +
    |\!|\!|\kappa_c|\!|\!|_c
    \le C.
\]

Finally we treat the averaged third-order norm $|\!|\!|\kappa|\!|\!|_3^{\mathrm{av}}$.  The operator-norm term
in $|\!|\!|\kappa|\!|\!|_3^{\mathrm{av}}$ was already bounded above as the case $k=3$ of $|\!|\!|\kappa|\!|\!|_k^{\mathrm{iso}}$.  For the special direct-cross
part, use the trivial decomposition
\[
    \kappa_{dd}\deq \kappa,
    \qquad
    \kappa_{dc}\deq \kappa_{cd}\deq \kappa_{cc}\deq 0.
\]
This is admissible because $\kappa$ is symmetric in its three label arguments and the zero functions are symmetric.
We claim that $|\!|\!|\kappa|\!|\!|_{dd}=\bigo(1)$.  
Recall the definition of $|\!|\!|\kappa|\!|\!|_{dd}$ from \eqref{eq:EKS-dd-cc}.
Fix $b_2,a_3\in\mathcal J$.  If
\[
    \kappa(\alpha_1,a_2b_2,a_3b_3)\neq0,
\]
then $\varpi(a_2b_2)=\varpi(a_3b_3)$.
Write $a_2b_2=((r_2,i_2),(s_2,j_2))$ and $a_3b_3=((r_3,i_3),(s_3,j_3))$.
For fixed $b_2$ and $a_3$, the coordinates $j_2$ and $i_3$ are fixed.
The direct relation forces $i_2=i_3$ and $j_3=j_2$.
Thus, $a_2$ has only $4$ possible block choices, $b_3$ has only
$4$ possible block choices, and $\alpha_1$ must belong to the same
companion class associated with the unordered pair $\{i_3,j_2\}$, which
has size at most $8$.  Therefore
\[
	    \sum_{a_2,b_3}\sum_{\alpha_1}
	    |\kappa(\alpha_1,a_2b_2,a_3b_3)|
	    \le 128C_3 = \bigo(1)
	\]
	\label{errpage:s4-third-cumulant-bound}
	uniformly in $b_2,a_3$.
Therefore, using the definition \eqref{eq:EKS-dd-cc}, we have $|\!|\!|\kappa|\!|\!|_{dd} \le C$, and thus
\[
    |\!|\!|\kappa|\!|\!|_3^{\mathrm{av}}\le C.
\]

Combining the estimates above, for every fixed $R\in\mathbb N$,
\[
    |\!|\!|\kappa|\!|\!|_2^{\mathrm{iso}}\le C,
    \qquad
    \max_{2\le k\le R}|\!|\!|\kappa|\!|\!|_k\le C_R.
\]
This is stronger than Assumption~\ref{asm:EKS2019-C}.

\medskip
\noindent\emph{Verification of Assumption~\ref{asm:EKS2019-D}.}

For a nonzero label $\alpha$, set
\[
    \mathcal N_1(\alpha)=\cdots=\mathcal N_R(\alpha)\deq \mathscr C(\alpha).
\]
For a zero label, set
\[
    \mathcal N_1(\alpha)=\cdots=\mathcal N_R(\alpha)\deq \{\alpha\}.
\]
Then $\alpha\in\mathcal N_1(\alpha)$, the sets are nested, and
\[
    |\mathcal N_R(\alpha)|\le8\le N^{1/2-\mu}
\]
for any fixed $\mu\in(0,1/2)$ and all sufficiently large $N$.

Consider the cumulant in Assumption~\ref{asm:EKS2019-D}.  If the companion classes
$\mathscr C(\alpha_1),\ldots,\mathscr C(\alpha_q)$ are pairwise distinct, then the variables entering
$
    g_j\bigl(
        \mathcal W_{\mathcal N_{n_j}(\alpha_j)
        \setminus\cup_{\ell\neq j}\mathcal N(\alpha_\ell)}
    \bigr)
$
depend on mutually independent companion classes, while the variable
$
    f\bigl(
        \mathcal W_{\mathcal I
        \setminus\cup_{j=1}^q \mathcal N_{n_j+1}(\alpha_j)}
    \bigr)
$
depends only on the remaining classes.  Therefore the collection of arguments of the cumulant can be split into independent subfamilies, and the cumulant is zero.
If two companion classes coincide, say
$\mathscr C(\alpha_i)=\mathscr C(\alpha_j)$ with $i\neq j$, then
\[
    \mathcal N_{n_i}(\alpha_i)
    \setminus
    \bigcup_{\ell\neq i}\mathcal N(\alpha_\ell)
    =
    \varnothing.
\]
Thus the corresponding argument $g_i(\mathcal W_\varnothing)$ is
constant.  A joint cumulant of order at least two vanishes whenever one
of its arguments is constant.  Hence the cumulant in
Assumption~\ref{asm:EKS2019-D} is again zero.
Thus the left-hand side of Assumption~\ref{asm:EKS2019-D} is identically
zero for the above choice of neighbourhoods, and Assumption~(D) follows.

We have verified \citet{Erdos2019random}'s assumptions.  
Applying \cite[Theorem~2.1]{Erdos2019random} gives the stated outside-spectrum
isotropic and averaged local laws.

\subsection{Proof of Lemma~\ref{lem:uniform-exterior-singular-gap}}
Before proving Lemma~\ref{lem:uniform-exterior-singular-gap}, we first establish a lemma which shows that, with high probability, there are no eigenvalues of $H_z$ outside the support of the self-consistent density  $\supp\nu_z$.

\begin{lemma}[No eigenvalues outside $\supp\nu_z$]\label{lem:no-eigs-gap}
	Fix $\varepsilon,\delta'>0$ and $C_z<\infty$.
	Let $z\in\mathbb C$ satisfy
	\[
		b_{\star,n}+\varepsilon<|z|\leq C_z,
	\]
	and let $I\subset\R$ be a fixed compact interval such that
	\[
		\dist(I,\supp\nu_z)\geq\delta'.
	\]
	Under Assumption~\ref{asm:moment},
	\[
		\Pr\bigl(\spec(H_z)\cap I=\varnothing\bigr)\to1.
	\]
\end{lemma}

Set
$
        \delta\deq \varepsilon/2.
$
For every $z\in\mathscr A_{n,\varepsilon,C}$,
$
        |z|>b_{\star,n}+\delta.
$
Lemma~\ref{lem:support}, applied with the separation parameter $\delta$
on the bounded range $|z|\leq C$, gives a deterministic constant
$d_{\varepsilon,C}>0$, independent of $n$ and $z$, such that
\[
        \dist(0,\supp\nu_z)>d_{\varepsilon,C}
        \qquad
        \text{for all }z\in\mathscr A_{n,\varepsilon,C}.
\]
Let
\[
        a_{\varepsilon,C}\deq\frac{d_{\varepsilon,C}}{2},
        \qquad
        I_{\varepsilon,C}\deq[-a_{\varepsilon,C},a_{\varepsilon,C}].
\]
If $x\in I_{\varepsilon,C}$ and $t\in\supp\nu_z$, then
\[
        |t-x|
        \geq \bigl| |t|-|x|\bigr|
        \geq \dist(0,\supp\nu_z)-a_{\varepsilon,C}
        >a_{\varepsilon,C}.
\]
Consequently,
\[
        \dist(I_{\varepsilon,C},\supp\nu_z)\geq a_{\varepsilon,C}
        \qquad
        \text{uniformly for }z\in\mathscr A_{n,\varepsilon,C}.
\]

% We next transfer the deterministic support gap to a samplewise spectral
% gap on the whole annulus.  
Set
\[
        h\deq\frac{a_{\varepsilon,C}}{4}
\]
and choose a $h$-net $\mathcal N_{n,\varepsilon,C}$ of $\mathscr A_{n,\varepsilon,C}$. 
The disks of radius $h/2$ centered at the points of
$\mathcal N_{n,\varepsilon,C}$ have pairwise disjoint interiors and are
contained in the disk $\{z\in\mathbb C:|z|\leq C+h/2\}$.  Comparing their
areas gives
\[
        |\mathcal N_{n,\varepsilon,C}|
        \pi\left(\frac h2\right)^2
        \leq
        \pi\left(C+\frac h2\right)^2.
\]
Hence
\[
        |\mathcal N_{n,\varepsilon,C}|
        \leq
        \left(1+\frac{2C}{h}\right)^2
        \leq M_{\varepsilon,C}
\]
for a deterministic integer $M_{\varepsilon,C}$ independent of $n$.

For each $w\in\mathcal N_{n,\varepsilon,C}$, we have
\[
        b_{\star,n}+\delta<|w|\leq C,
        \qquad
        \dist(I_{\varepsilon,C},\supp\rho_w)\geq a_{\varepsilon,C}.
\]
With parameters $(\varepsilon,\delta',C_z) =(\delta,a_{\varepsilon,C},C)$, Lemma~\ref{lem:no-eigs-gap} applies to every deterministic sequence $w_n\in\mathcal N_{n,\varepsilon,C}$ and gives
\[
        \Pr\bigl(\spec(H_{w_n})\cap I_{\varepsilon,C}\neq\varnothing\bigr)
        \to 0.
\]
The union-bound gives	\[
        \Pr\left(
                \bigcup_{w\in\mathcal N_{n,\varepsilon,C}}
                \{\spec(H_w)\cap I_{\varepsilon,C}\neq\varnothing\}
        \right)
        \leq
        \sum_{j=1}^{M_{\varepsilon,C}}
        \Pr\bigl(\spec(H_{w_{n,j}})\cap I_{\varepsilon,C}
        \neq\varnothing\bigr)
        \to 0.
\]
Therefore the event
\[
        \Omega_n
        \deq
        \bigcap_{w\in\mathcal N_{n,\varepsilon,C}}
        \{\spec(H_w)\cap I_{\varepsilon,C}=\varnothing\}
\]
satisfies $\Pr(\Omega_n)\to1$.  On $\Omega_n$,
\begin{equation}\label{eq:exterior-gap-on-net}
        \dist(0,\spec(H_w))\geq a_{\varepsilon,C}
        \qquad
        \text{for every }w\in\mathcal N_{n,\varepsilon,C}.
\end{equation}

Fix an arbitrary $z\in\mathscr A_{n,\varepsilon,C}$ and choose
$w\in\mathcal N_{n,\varepsilon,C}$ such that $|z-w|\leq h$.  Directly
from the definition of the Hermitian linearization, we have
\[
        H_z-H_w
        =
        \begin{pmatrix}
                0 & -(z-w)I_{2n}\\
                -(\overline z-\overline w)I_{2n} & 0
        \end{pmatrix},
\]
and hence
\begin{equation}\label{eq:linearization-Lipschitz}
        \|H_z-H_w\|_2=|z-w|\leq h.
\end{equation}
Let
$\lambda_1(H_z)\geq\cdots\geq\lambda_{4n}(H_z)$ and
$\lambda_1(H_w)\geq\cdots\geq\lambda_{4n}(H_w)$.
Weyl's perturbation inequality and
\eqref{eq:exterior-gap-on-net}--\eqref{eq:linearization-Lipschitz} imply,
for every $j\in\intset{4n}$,
\[
        |\lambda_j(H_z)|
        \geq
        |\lambda_j(H_w)|
        -|\lambda_j(H_z)-\lambda_j(H_w)|
        \geq
        a_{\varepsilon,C}-h
        =
        \frac{3a_{\varepsilon,C}}{4}.
\]
Taking the minimum over $j$ shows that, on $\Omega_n$,
\[
        \inf_{z\in\mathscr A_{n,\varepsilon,C}}
        \dist(0,\spec(H_z))
        \geq
        \frac{3a_{\varepsilon,C}}{4}.
\]
Combining this identity with
\[
        \dist(0,\spec(H_z))		
        =
        \sigma_{\min}(zI_{2n}-\mathcal X).
\]
proves \eqref{eq:uniform-exterior-gap-split-noise} with
\[
        \gamma_{\varepsilon,C}\deq\frac{3a_{\varepsilon,C}}{4}>0.
\]

Finally, suppose that $\|\mathcal E\|_2=\smallop(1)$.  Since
$\gamma_{\varepsilon,C}$ is deterministic and positive,
\[
        \Pr\left(
                \|\mathcal E\|_2\leq\frac{\gamma_{\varepsilon,C}}2
        \right)
        \to 1.
\]
No independence between $\mathcal E$ and $\mathcal X$ is needed: by the union
bound, the event
\[
        \widetilde\Omega_n
        \deq
        \Omega_n
        \cap
        \left\{
                \|\mathcal E\|_2\leq\frac{\gamma_{\varepsilon,C}}2
        \right\}
\]
satisfies
\[
        \Pr(\widetilde\Omega_n)
        \geq
        1-\Pr(\Omega_n^c)
        -\Pr\left(
                \|\mathcal E\|_2>\frac{\gamma_{\varepsilon,C}}2
        \right)
        \to 1.
\]
For any
$z\in\mathscr A_{n,\varepsilon,C}$, the variational characterization of
the least singular value and the triangle inequality give
\begin{align*}
        \sigma_{\min}(zI_{2n}-\mathcal X-\mathcal E)
        &=
        \inf_{\|v\|_2=1}
        \|(zI_{2n}-\mathcal X-\mathcal E)v\|_2\\
        &\geq
        \inf_{\|v\|_2=1}
        \bigl\{\|(zI_{2n}-\mathcal X)v\|_2-\|\mathcal E v\|_2\bigr\}\\
        &\geq
        \sigma_{\min}(zI_{2n}-\mathcal X)-\|\mathcal E\|_2.
\end{align*}
On $\widetilde\Omega_n$, we have
\[
        \inf_{z\in\mathscr A_{n,\varepsilon,C}}
        \sigma_{\min}(zI_{2n}-\mathcal X-\mathcal E)
        \geq
        \gamma_{\varepsilon,C}
        -\frac{\gamma_{\varepsilon,C}}2
        =
        \frac{\gamma_{\varepsilon,C}}2.
\]
This proves the second assertion.

\subsection{Proof of Lemma~\ref{lem:no-eigs-gap}}

Let 
\[
        m_N(\zeta)\deq \frac{1}{N}\Tr \mathfrak G_z(\zeta).
\]
Choose $0<\tau<\frac12$ and set
\[
        \eta_N\deq N^{-1+\tau}.
\]
Let $m$ be the Stieltjes transform of $\nu_z$.  For every $E\in I$, we have
\[
        \Im m(E+\ii\eta_N)
        =
        \int_{\mathbb R}
        \frac{\eta_N}{(x-E)^2+\eta_N^2}\,\dif\nu_z(x)
        \le
        \frac{\eta_N}{\delta'^2}.
\]
Hence
\[
        \sup_{E\in I}\Im m(E+\ii\eta_N)\le \frac{\eta_N}{\delta'^2}= \smallo(N^{-\tau}).
\]

From Lemma~\ref{lem:local-law}, we have
\[
        \sup_{E\in I}
        \left|
        m_N(E+\ii\eta_N)-m(E+\ii\eta_N)
        \right|
        \prec
        \frac{1}{N}.
\]
This implies that
\[
        \sup_{E\in I}
        \left|
        m_N(E+\ii\eta_N)-m(E+\ii\eta_N)
        \right|
        =\smallo(N^{-\tau})
\]
with probability tending to one for all sufficiently large $N$. Combining this with the previous deterministic bound gives
\[
        \sup_{E\in I}\Im m_N(E+\ii\eta_N)=\smallo(N^{-\tau})
\]
with probability tending to one.

Now suppose that $\spec(H_z)\cap I\neq\varnothing$, and choose $\lambda\in \spec(H_z)\cap I$. Then
\[
        \Im m_N(\lambda+\ii\eta_N)
        =
        \frac{1}{N}\sum_{k=1}^N
        \frac{\eta_N}{(\lambda_k(H_z)-\lambda)^2+\eta_N^2}
        \ge
        \frac{1}{N\eta_N}
        =
        N^{-\tau}.
\]
This contradicts the previous high-probability upper bound. Therefore,
\[
        \Pr \bigl(\spec(H_z)\cap I=\varnothing\bigr)\to  1.
\]
This completes the proof of Lemma~\ref{lem:no-eigs-gap}.

\subsection{Proof of Lemma~\ref{lem:quadratic form-convergence}}

Set $N=4n$.  For $z\in \mathscr K_n$, let
\[
        \mathcal X_z\deq \mathcal X-zI_{2n}
        =
        \begin{pmatrix}
                -zI_n & \mathsf A\\
                \mathsf B & -zI_n
        \end{pmatrix}.
\]
When $\mathcal X_z$ is invertible,
\[
        H_z^{-1}
        =
        \begin{pmatrix}
                0 & (\mathcal X_z^*)^{-1}\\
                \mathcal X_z^{-1} & 0
        \end{pmatrix},
        \qquad
        \mathcal X_z^{-1}
        =-
        \begin{pmatrix}
                zG(z) & \mathsf A\mathcal G(z)\\
                \mathsf B G(z) & z\mathcal G(z)
        \end{pmatrix}.
\]
Hence, in the $4\times4$ block decomposition of $H_z^{-1}$,
\begin{equation}\label{eq:Hz-inverse-blocks}
\begin{aligned}
        (H_z^{-1})_{31}=-zG(z),
        \qquad
        (H_z^{-1})_{41}=-\mathsf B G(z),\\
        (H_z^{-1})_{42}=-z\mathcal G(z),
        \qquad
        (H_z^{-1})_{32}=-\mathsf A\mathcal G(z).
\end{aligned}
\end{equation}

Since
$\mathscr K_n\subset\{z\in\mathbb C:|z|>b_{\star,n}+\varepsilon/2\}$,
Lemma~\ref{lem:support}, applied with $\varepsilon/2$, gives
$\gamma>0$, depending only on $\varepsilon$ and $C$, such that
\[
        (-\gamma,\gamma)\cap\supp\nu_z=\varnothing,
        \qquad z\in \mathscr K_n.
\]
Evaluating the MDE at $\zeta=0$ gives
\begin{equation}\label{eq:M-zero-limit}
        M_z(0)
        =
        \begin{pmatrix}
                0 & 0 & -\overline z^{-1}I_n & 0\\
                0 & 0 & 0 & -\overline z^{-1}I_n\\
                -z^{-1}I_n & 0 & 0 & 0\\
                0 & -z^{-1}I_n & 0 & 0
        \end{pmatrix}.
\end{equation}

We first upgrade the exterior singular gap to overwhelming probability.
Fix $\tau\in(0,1/2)$, set $\eta_0=N^{-1+\tau}$, and let
$I=[-\gamma/2,\gamma/2]$.  The support gap and the Stieltjes
representation of $M_z$ imply
\[
        \sup_{z\in \mathscr K_n,\; E\in I}
        \Im\frac1N\Tr M_z(E+\ii\eta_0)
        \leq C\eta_0.
\]
If $H_z$ had an eigenvalue $\lambda\in I$, then
\[
        \Im\frac1N\Tr\mathfrak G_z(\lambda+\ii\eta_0)
        \geq\frac1{N\eta_0}=N^{-\tau}.
\]
Since $\tau<1/2$, one has $\eta_0=\smallo(N^{-\tau})$.  The averaged estimate
in Lemma~\ref{lem:local-law}, which is
uniform for $z\in \mathscr K_n$ and $E\in I$, gives
\[
        \sup_{E\in I}
        \left|
        \frac1N\Tr\bigl(\mathfrak G_z(E+\ii\eta_0)
        -M_z(E+\ii\eta_0)\bigr)
        \right|\prec N^{-1}.
\]
It follows that, for every fixed $D>0$, the probability that
$\spec(H_z)\cap I\neq\varnothing$ is at most $N^{-D}$ for all sufficiently
large $N$.  A finite $\gamma/8$-net of $\mathscr K_n$, followed by a union bound and
Weyl's inequality using $\|H_z-H_w\|_2=|z-w|$, therefore yields a constant
$\gamma_0>0$ such that
\begin{equation}\label{eq:overwhelming-uniform-exterior-gap}
        \Pr\left(
        \inf_{z\in \mathscr K_n}\dist(0,\spec(H_z))\geq\gamma_0
        \right)
        \geq1-N^{-D}
\end{equation}
for every fixed $D>0$ and all sufficiently large $N$.

For $\alpha\in\{1,2,3,4\}$, let $\iota_\alpha(x)\in\mathbb C^{4n}$
denote the vector whose $\alpha$-th $n$-block equals $x$ and whose other
blocks are zero.  Set $\eta=N^{-1}$ and let $\mathcal N_N$ be an
$N^{-1}$-net of $\mathscr K_n$.  Since $\mathscr K_n\subset\{|z|\leq C\}$,
$|\mathcal N_N|=\bigo(N^2)$.  
On the event in \eqref{eq:overwhelming-uniform-exterior-gap}, the resolvent identity gives
\[
        \sup_{w\in\mathcal N_N}
        \|H_w^{-1}-\mathfrak G_w(\ii\eta)\|_2
        \leq C\eta.
\]
The matrix valued Stieltjes representation and the deterministic support
gap similarly give
\[
        \sup_{w\in\mathcal N_N}
        \|M_w(\ii\eta)-M_w(0)\|_2
        \leq C\eta.
\]
For fixed $\alpha,\beta\in\{1,2,3,4\}$, the $z$-uniform isotropic
estimate in Lemma~\ref{lem:local-law} and a union bound over
$\mathcal N_N$ yield
\begin{equation}\label{eq:quantitative-zero-energy-net}
        \sup_{w\in\mathcal N_N}
        \left|
        \iota_\alpha(x)^*\bigl(H_w^{-1}-M_w(0)\bigr)
        \iota_\beta(y)
        \right|
        \prec N^{-1/2}.
\end{equation}

On the event in \eqref{eq:overwhelming-uniform-exterior-gap}, the resolvent
identity also gives
\[
        \|H_z^{-1}-H_w^{-1}\|_2\leq C|z-w|.
\]
The explicit formula \eqref{eq:M-zero-limit} and
$\inf_{z\in \mathscr K_n}|z|\geq\varepsilon$ imply
\[
        \|M_z(0)-M_w(0)\|_2\leq C|z-w|.
\]
Thus \eqref{eq:quantitative-zero-energy-net} extends from the net to all of
$\mathscr K_n$:
\begin{equation}\label{eq:quantitative-zero-energy-uniform}
        \sup_{z\in \mathscr K_n}
        \left|
        \iota_\alpha(x)^*\bigl(H_z^{-1}-M_z(0)\bigr)
        \iota_\beta(y)
        \right|
        \prec N^{-1/2}.
\end{equation}

Taking $(\alpha,\beta)=(3,1),(4,2),(4,1),(3,2)$ in
\eqref{eq:quantitative-zero-energy-uniform}, and using
\eqref{eq:Hz-inverse-blocks}--\eqref{eq:M-zero-limit}, proves respectively
the four scalar estimates in
\eqref{eq:quantitative-G-calG-forms}--\eqref{eq:quantitative-BG-AcalG-forms},
because $N=4n$ and $\varepsilon\leq |z|\leq C$ on $\mathscr K_n$.

\subsection{Proof of Lemma~\ref{lem:resolvent-est}}

We first establish the resolvent estimates for $R_{\mathcal X}(z)$, and then pass from $R_{\mathcal X}(z)$ to $\widetilde R_{\mathcal X}(z)$ using the resolvent identity and Lemma~\ref{lem:uniform-exterior-singular-gap}.

Recall the definitions of $G(z)$ and $\mathcal G(z)$ in \eqref{eq:G-calG-def}.
From Lemma~\ref{lem:quadratic form-convergence}, for deterministic vectors $x,y\in\mathbb C^n$ with bounded Euclidean norm, we have the following estimates uniformly for $z\in K$:
\begin{align}
        x^*G(z)y-\frac{x^*y}{z^2} \prec n^{-1/2}, 
        \qquad
        x^*\mathcal G(z)y-\frac{x^*y}{z^2} \prec n^{-1/2}, \label{eq:xGy-est}\\
        x^*\mathsf B G(z)y \prec n^{-1/2},
        \qquad
        x^*\mathsf A\mathcal G(z)y \prec n^{-1/2}. \label{eq:xABGy-est}
\end{align}
Write $a=(\begin{smallmatrix}a_1\\a_2\end{smallmatrix})$, with $a_1,a_2\in\mathbb C^n$.  For $\nu=(k,\mathsf s)$, using \eqref{eq:block-inverse} and \eqref{eq:normalized-spike-basis},
\[
        a^*R_{\mathcal X}(z)\mathfrak u_\nu
        =
        \frac{1}{\sqrt{2}}
        \Bigl[
                z\,a_1^*G(z)u_k
                +
                \mathsf s\,a_1^*\mathsf A\mathcal G(z)u_k
                +
                a_2^*\mathsf B G(z)u_k
                +
                \mathsf s z\,a_2^*\mathcal G(z)u_k
        \Bigr].
\]
Applying \eqref{eq:xGy-est} and \eqref{eq:xABGy-est} gives
\[
        a^*R_{\mathcal X}(z)\mathfrak u_\nu
        =
        \frac1z\,a^*\mathfrak u_\nu+\bigo_\prec(n^{-1/2}),
        \qquad z\in \mathscr K_n.
\]
Since $r$ is fixed, this implies
\begin{equation}\label{eq:split-noise-resolvent-signal-basis}
        \sup_{z\in K}
        \left\|a^*R_{\mathcal X}(z)\mathcal U-\frac1z a^*\mathcal U\right\|_2
        \prec n^{-1/2}.
\end{equation}

It remains to pass from $R_{\mathcal X}$ to $\widetilde R_{\mathcal X}$.
By Lemma~\ref{lem:uniform-exterior-singular-gap},
\[
        \sup_{z\in \mathscr K_n}\|R_{\mathcal X}(z)\|_2+\sup_{z\in \mathscr K_n}\|\widetilde R_{\mathcal X}(z)\|_2=\bigop(1).
\]
This, together with the resolvent identity $\widetilde R_{\mathcal X}(z)-R_{\mathcal X}(z) = R_{\mathcal X}(z)\mathcal E\widetilde R_{\mathcal X}(z)$ and Lemma~\ref{lem:mask-remainder-bound}, implies that
\begin{align}
\label{eq:perturbed-resolvent-difference-signal-basis}
\sup_{z\in \mathscr K_n}
\left\|a^*\bigl(\widetilde R_{\mathcal X}(z)-R_{\mathcal X}(z)\bigr)\mathcal U\right\|_2
=
\bigop(\mu_n).
\end{align}
By \eqref{eq:perturbed-resolvent-difference-signal-basis} and \eqref{eq:split-noise-resolvent-signal-basis}, we complete the proof.

\subsection{Proof of Lemma~\ref{lem:samplewise-projector-approximation}}

Recall the event $\mathcal G_{a,n}$ defined above and the bound
\eqref{eq:samplewise-good-event}.
On $\mathcal G_{a,n}$, $\widetilde{\mathcal P}_{a,i,\mathsf s}$ is the Riesz projection appearing in Proposition~\ref{prop:split-evec-projection}.
Hence, for any deterministic sequences $x_n,y_n\in\mathbb C^{2n}$ with $\|x_n\|_2\vee\|y_n\|_2\leq1$, that theorem and \eqref{eq:samplewise-outlier-separation} give, for every fixed $\eta>0$,
\begin{equation}\label{eq:samplewise-proof-out-isotropic}
	\Pr\!\left(
	\left|
	x_n^*\mathscr E^{\mathrm{out}}_{a,i,\mathsf s}y_n
	\right|>\eta,\,
	\mathcal G_{a,n}
	\right)
	=
	\smallo(1).
\end{equation}
If \eqref{eq:samplewise-direction-isotropic} failed for some fixed $\eta>0$, then, in view of \eqref{eq:samplewise-good-event}, there would exist $c>0$, a subsequence $n_\ell$, and deterministic vectors $x_{n_\ell},y_{n_\ell}\in\mathbb C^{2n_\ell}$ with norms at most one such that
\[
	\Pr\!\left(
	\left|
	x_{n_\ell}^*
	\mathscr E^{\mathrm{out}}_{a,i,\mathsf s}
	y_{n_\ell}
	\right|>\eta,\,
	\mathcal G_{a,n_\ell}
	\right)
	\geq c
\]
for every $\ell$.
Extend these vectors arbitrarily to deterministic sequences along all $n$.  Equation~\eqref{eq:samplewise-proof-out-isotropic} then makes the displayed probability tend to zero along $n_\ell$, a
contradiction.
This proves \eqref{eq:samplewise-direction-isotropic}.

For the corresponding assertion about the upper-left block, define
\[
        J_1\deq
        \begin{pmatrix}I_n\\0\end{pmatrix}
        \in\mathbb R^{2n\times n}.
\]
Since $[A]_{11}=J_1^{\top}AJ_1$ and $J_1^{\top}J_1=I_n$, we have
\[
        |x^*\mathscr E_{a,i}y|
        \leq
        |(J_1x)^*\mathscr E^{\mathrm{out}}_{a,i,+}(J_1y)|
        +
        |(J_1x)^*\mathscr E^{\mathrm{out}}_{a,i,-}(J_1y)|.
\]
Applying \eqref{eq:samplewise-direction-isotropic} to the two terms proves the
claimed deterministic-vector bound for $\mathscr E_{a,i}$.

We next prove the operator-norm bounds.  Let
\[
        \Gamma_{a,i,\mathsf s}
        \deq
        \left\{
        z\in\mathbb C:
        |z-\alpha_{a,i,\mathsf s}|=\rho_{a,i,\mathsf s}
        \right\},
        \qquad
        \rho_{a,i,\mathsf s}
        \deq\frac{\Delta_{a,i,\mathsf s}}{4}.
\]
We orient $\Gamma_{a,i,\mathsf s}$ counterclockwise.
Define $\widetilde R_{\mathcal X_a}(z)$, $\mathcal U_a$, $\mathcal D_a$ for the samplewise analogues of the quantities in \eqref{eq:evec-proj-woodbury}, and let
\[
        L_a(z)
        \deq
        \mathcal D_a^{-1}
        -\mathcal U_a^*\widetilde R_{\mathcal X_a}(z)\mathcal U_a.
\]
On $\mathcal G_{a,n}$, the contour formula and
the Woodbury identity give
\begin{equation}\label{eq:samplewise-proof-contour}
        \widetilde{\mathcal P}_{a,i,\mathsf s}
        =
        \frac{1}{2\pi\ii}
        \oint_{\Gamma_{a,i,\mathsf s}}
        \widetilde R_{\mathcal X_a}(z)\mathcal U_a
        L_a(z)^{-1}
        \mathcal U_a^*\widetilde R_{\mathcal X_a}(z)\,\dif z.
\end{equation}
Lemma~\ref{lem:uniform-exterior-singular-gap} and
\eqref{eq:L-inverse-bound} imply
\begin{equation}\label{eq:samplewise-proof-resolvent-bounds}
        \sup_{z\in\Gamma_{a,i,\mathsf s}}
        \|\widetilde R_{\mathcal X_a}(z)\|_2=\bigop(1),
        \qquad
        \sup_{z\in\Gamma_{a,i,\mathsf s}}
        \|L_a(z)^{-1}\|_2
        =\bigop(\Delta_{a,i,\mathsf s}^{-1}).
\end{equation}
Since $\|\mathcal U_a\|_2=1$ and
$|\Gamma_{a,i,\mathsf s}|=2\pi\rho_{a,i,\mathsf s}$,
\eqref{eq:samplewise-proof-contour}--\eqref{eq:samplewise-proof-resolvent-bounds}
give
\begin{align}
        \|\widetilde{\mathcal P}_{a,i,\mathsf s}\|_2
        &\leq
        \rho_{a,i,\mathsf s}
        \bigg(
        \sup_{z\in\Gamma_{a,i,\mathsf s}}
        \|\widetilde R_{\mathcal X_a}(z)\|_2
        \bigg)^2
        \sup_{z\in\Gamma_{a,i,\mathsf s}}
        \|L_a(z)^{-1}\|_2        
        =
        \bigop\!\left(
        \rho_{a,i,\mathsf s}\Delta_{a,i,\mathsf s}^{-1}
        \right)
        =\bigop(1).
        \label{eq:samplewise-proof-riesz-norm}
\end{align}
For every $M>0$,
\[
        \Pr\!\left(
        \|\widetilde{\mathcal P}_{a,i,\mathsf s}\|_2>M
        \right)
        \leq
        \Pr(\mathcal G_{a,n}^{\mathsf c})
        +
        \Pr\!\left(
        \|\widetilde{\mathcal P}_{a,i,\mathsf s}\|_2>M,\,
        \mathcal G_{a,n}
        \right).
\]
It follows from \eqref{eq:samplewise-good-event} and
\eqref{eq:samplewise-proof-riesz-norm} that the same norm bound holds without
restricting to $\mathcal G_{a,n}$.  Thus
\begin{equation}\label{eq:samplewise-proof-error-norm}
        \|\mathscr E^{\mathrm{out}}_{a,i,\mathsf s}\|_2
        \leq
        \|\widetilde{\mathcal P}_{a,i,\mathsf s}\|_2
        +\|\mathcal P_{a,i,\mathsf s}\|_2
        =\bigop(1).
\end{equation}
By construction,
$\operatorname{rank}(\widetilde{\mathcal P}_{a,i,\mathsf s})\leq1$, while
$\operatorname{rank}(\mathcal P_{a,i,\mathsf s})=1$.  Hence
\begin{equation}\label{eq:samplewise-proof-out-rank}
        \operatorname{rank}
        (\mathscr E^{\mathrm{out}}_{a,i,\mathsf s})
        \leq
        \operatorname{rank}
        (\widetilde{\mathcal P}_{a,i,\mathsf s})
        +
        \operatorname{rank}(\mathcal P_{a,i,\mathsf s})
        \leq2.
\end{equation}
Finally,
\begin{align}
        \|\mathscr E_{a,i}\|_2
        &\leq
        \|\widehat\Pi_{a,i}\|_2+\|\Pi_{a,i}\|_2
        \leq
        \sum_{\mathsf s\in\{+,-\}}
        \|\widetilde{\mathcal P}_{a,i,\mathsf s}\|_2+1
        =\bigop(1),
        \label{eq:samplewise-proof-top-norm}\\
        \operatorname{rank}(\mathscr E_{a,i})
        &\leq
        \operatorname{rank}(\widehat\Pi_{a,i})
        +\operatorname{rank}(\Pi_{a,i})
        \leq 3.
        \label{eq:samplewise-proof-top-rank}
\end{align}
Equations~\eqref{eq:samplewise-proof-error-norm},
\eqref{eq:samplewise-proof-out-rank},
\eqref{eq:samplewise-proof-top-norm}, and
\eqref{eq:samplewise-proof-top-rank} prove
\eqref{eq:samplewise-direction-rank-norm}.

It remains to treat random test vectors.  Let $Z_{a,n}$ denote either
$\mathscr E^{\mathrm{out}}_{a,i,\mathsf s}$ or $\mathscr E_{a,i}$, with the
corresponding dimension, and
\[
        \mathcal B_{M,n}
        \deq
        \left\{
        \|x_n\|_2\vee\|y_n\|_2\leq M
        \right\},
\]
where $M$ is a fixed constant. 
For any deterministic vectors $x,y$ and constant $\eta>0$, define 
\[
p_n(x,y) \deq \Pr(|x^*Z_{a,n}y|>\eta ).
\]
Since $(x_n,y_n)$ is independent of $Z_{a,n}$, we have
\[
\Pr(|x_n^*Z_{a,n}y_n|>\eta, \mathcal B_{M,n}) 
= \mathbb E \big\{\indicator{(\mathcal B_{M,n})} p_n(x_n,y_n) \big\}
\leq \sup_{\substack{x,y\ \mathrm{deterministic}\\\|x\|_2\vee\|y\|_2\leq M}} p_n(x,y).
\]
Hence, for every fixed $M<\infty$ and $\eta>0$,
\begin{align}
        \Pr(|x_n^*Z_{a,n}y_n|>\eta)
        &\leq
        \Pr(\mathcal B_{M,n}^{\mathsf c})
        +
        \sup_{\substack{x,y\ \mathrm{deterministic}\\
                         \|x\|_2\vee\|y\|_2\leq M}}
        \Pr(|x^*Z_{a,n}y|>\eta).                                   
        \label{eq:samplewise-proof-random-tests}
\end{align}
For fixed $M$, write $x=M\widetilde x$ and $y=M\widetilde y$.  The
deterministic-vector bound applied to
$\|\widetilde x\|_2\vee\|\widetilde y\|_2\leq1$ with threshold
$\eta/M^2$ shows that the last supremum is $\smallo(1)$.  The assumption
$\|x_n\|_2\vee\|y_n\|_2=\bigop(1)$ gives
\[
        \lim_{M\to\infty}\limsup_{n\to\infty}
        \Pr(\mathcal B_{M,n}^{\mathsf c})=0.
\]
Taking first $n\to\infty$ and then $M\to\infty$ in \eqref{eq:samplewise-proof-random-tests} yields $x_n^*Z_{a,n}y_n=\smallop(1)$.
This proves \eqref{eq:samplewise-direction-isotropic-random} and completes the proof.

\bibliographystyle{abbrvnat}
\bibliography{ref}

\end{document}